\documentclass[11pt, reqno]{amsart}

\usepackage{adjustbox}
\usepackage{amsmath, amsthm, amsfonts}
\usepackage[toc,page]{appendix}
\usepackage[english]{babel}
\usepackage{bbm}
\usepackage{booktabs}
\usepackage{enumitem} 
\usepackage{float}
\usepackage[T1]{fontenc}
\usepackage[margin=1in, marginparwidth=1.25cm, marginparsep=0.25cm]{geometry}
\usepackage{graphicx}
\usepackage{mathrsfs}  
\usepackage{mathtools}
\usepackage{microtype}
\usepackage{newtxtext} 
\usepackage{newtxmath} 
\usepackage{pdflscape}
\usepackage{relsize}
\usepackage{makecell}
\usepackage[normalem]{ulem}
\usepackage[dvipsnames]{xcolor}
\usepackage{tikz}
\usepackage{setspace}
\usetikzlibrary{patterns.meta}
\usepackage[colorlinks=true, linkcolor=black, citecolor=teal, urlcolor=teal]{hyperref}

\usepackage[ibrackets]{mismath}
\mleftright

\setlist[itemize]{label=\textbullet, leftmargin=15pt}

\theoremstyle{plain}
\newtheorem{theorem}{Theorem}[section]
\newtheorem{lemma}[theorem]{Lemma}
\newtheorem{proposition}[theorem]{Proposition}
\newtheorem{corollary}[theorem]{Corollary}

\theoremstyle{definition}
\newtheorem{definition}[theorem]{Definition}

\newtheorem{question}{Question}

\theoremstyle{remark} 
\newtheorem{remark}[theorem]{Remark}

\renewcommand{\R}{\mathbb{R}}

\renewcommand{\Z}{\mathbb{Z}}
\renewcommand{\N}{\mathbb{N}}
 
\newcommand{\T}{\mathbb{T}} 
\newcommand{\Prb}{\mathbb{P}}
\newcommand{\pat}{\partial}
\makeatletter

\newcommand{\deepuwave}{\bgroup \markoverwith{\lower8\p@\hbox{\sixly \char58}}\ULon}
\makeatother

\makeatletter
\newcommand{\doubleeq}[2]{%
  \par\vspace{\abovedisplayskip}\noindent
  \begin{minipage}{.48\textwidth}
    \def\maketag@@@##1{\hbox{\m@th\normalfont##1\,,}}%
    \begin{equation}
      #1
    \end{equation}
  \end{minipage}%
  \hfill
  \begin{minipage}{.48\textwidth}
    \begin{equation}
      #2
    \end{equation}
  \end{minipage}%
  \par\vspace{\belowdisplayskip}%
}
\makeatother

\makeatletter
\newcommand{\doubleeqflex}[4]{%
  \par\vspace{\abovedisplayskip}\noindent
  \begin{minipage}{#1\textwidth}
    \def\maketag@@@##1{\hbox{\m@th\normalfont##1,}}%
    \begin{equation}
      #3
    \end{equation}
  \end{minipage}%
  \hfill
  \begin{minipage}{#2\textwidth}
    \begin{equation}
      #4
    \end{equation}
  \end{minipage}%
  \par\vspace{\belowdisplayskip}%
}
\makeatother

\DeclareMathOperator{\supp}{supp}

\DeclareMathOperator{\Rea}{Re}

\numberwithin{equation}{section}
\counterwithin{table}{section}
\makeatletter
\def\l@subsection{\@tocline{2}{0pt}{2.5pc}{5pc}{}} 
\makeatother

\title{Well-posedness for the generalised fractional BBM equation}
\author[J.~Barratt]{Jackson Barratt}

\address{Jackson Barratt, School of Mathematics, Monash University, Clayton, VIC 3800, Australia}
\email{jackson.barratt1@gmail.com}

\begin{document}
\maketitle
\begin{abstract}
In this paper, we study a generalised family of fractional BBM equations on the $d$ dimensional torus $\T^d$ for $d\geq 1$. We establish low regularity local and global well-posedness results on the $H^s(\T^d)$ scale that improve upon the works of Bona and Chen in \cite{bona2009} and Kim and Kwak in \cite{kim2026}. We also prove endpoint global well-posedness for the cubic regularised Benjamin-Ono equation in $H^{1/2}(\T)$. Furthermore, we  obtain a well-posedness result for the generalised fractional BBM in the spaces $L^{k+1}(\T^d)$ where $k$ is order of the polynomial nonlinearity. As an application of this result, we show that for a certain range of values and for a suitable randomisation of the initial data, the model is almost surely globally well-posed in $H^s(\T^d)$ for $s\geq 0$. Lastly, we show that the family exhibits a strong form of ill-posedness in $H^s(\T^d)$ for $s<0$ in the form of an infinite loss of regularity at all initial datum.
\end{abstract}
\tableofcontents

\section{Introduction}
In this paper, we study the generalised fractional Benjamin-Bona-Mahony (BBM) equation on the $d$ dimensional torus, $\T^d \coloneqq [0,2\pi]^d$,
\begin{equation}\label{Gen-BOBBM-eq}
    \begin{cases}
        \pat_t u  + |D_x|^\beta\pat_t u+  \sum_{i=1}^d \pat_{x_i} \left(u + \mu \frac{1}{k}u^k\right) = 0 \\
        u(x,0) = u_0(x)
    \end{cases}
\end{equation}
where $u: \T^d \times [0,\infty) \to \R$,
\begin{align*}
    \quad x = (x_1,\dots,x_d)\in \T^d, \quad 1\leq \beta\leq 2, \quad k\geq 2, \quad \mu \in \{-1,1\}, \quad d\geq 1
\end{align*}
 and $|D_x|^\beta$ is a Fourier multiplier with symbol $|n|^\beta$.

This family contains the standard BBM equation (\eqref{Gen-BOBBM-eq}, $d=1$, $k=2$, $\beta =2$, $\mu =1$) that was first derived by Peregrine \cite{peregrine1966} and studied by Benjamin \textit{et al.} \cite{benjamin1972a} as a model for the propagation of long wavelength water waves with short amplitude. As opposed to another model of such waves, the Korteweg-de Vries (KdV) equation, the BBM model has both favourable physcial properties  due to it's bounded dispersion \cite{albert1991,benjamin1972a} and favourable analytic properties facilitated by the smoothing nonlinearity \cite{bonaSharpWellposednessResults2009,forlano2020,forlano2025}. This was first noted by Benjamin \textit{et al.} who, via a process known as \emph{regularisation}, derived the BBM equation from the KdV equation. Formally, the regularisation process involves substituting the first order approximation of travelling waves, $\pat_x u = -\pat_t u$, exchanging a derivative in space for one in time. The reason for this substitution as well as it's physical implications is discussed in depth in \cite{benjamin1972a,benjamin1967}.

 This regularisation process can also be applied to the fractional and higher-order nonlinearity generalisations of the KdV equation; which have physical relevance  \cite{micu2001,wazwaz2005} and are also of analytical interest \cite{molinet2025,johnson2013,cunha2025,klein2022}. In doing so, we encounter the generalised fractional BBM model (\eqref{Gen-BOBBM-eq}, $d=1$, $k=2$, $\beta =2$, $\mu =1$) which was first studied by Albert and Bona in \cite{albert1991} and is also known as the regularised long wave equation.

 Furthermore, the higher dimensional analogue of the BBM equation (\eqref{Gen-BOBBM-eq}, $d\geq 2$, $\beta =2$, $\mu =1$) was first introduced by Goldstein and Wichnoski in \cite{goldstein1980} where they found the equation retained both its bounded dispersion and smoothing nonlinearity in higher dimensions. The well-posedness of the equation was further studied in \cite{avrin1985,wang2024}. In the interest of furthering this avenue of study, we generalise the model of Goldstein and Wichnoski to \eqref{Gen-BOBBM-eq} and establish well-posedness results on both the Sobolev and Lebesgue scale. We also introduce a focusing/defocusing variable $\mu$ and show, rather suprisingly, that the focusing/defocusing nature of the nonlinearity has little impact on the global well-posedness theory.

 Fortunately, the desirable analytic properties of the BBM equation also persist in the generalised form \eqref{Gen-BOBBM-eq}. To see this, we factor the time derivative terms and invert the resulting spatial operator to obtain,
 \begin{align*}
    \pat_t u = - \phi_\beta(D_x)\left(u + \mu \frac{1}{k} u^k\right), \quad \phi_\beta(D_x) \coloneqq \left(1+|D_x|^\beta\right)^{-1}\sum_{i=1}^d \pat_{x_i}.
 \end{align*}
 Here, we witness the nonlinearity being smoothed by $\beta-1$ derivatives, which will aid our efforts in proving local and global well-posedness. While we direct our focus to $\T^d$, most of the results in this paper also hold in $\R^d$; with the only differences stemming from the defocusing case $\mu =-1$ and Regime 3 of Theorem \ref{Global:thm}. Here, the presented global theory relies on the finite measure of $\T^d$ and thus cannot be generalised to $\R^d$.

We will now move to discussing our results on the well/ill-posedness theory of \eqref{Gen-BOBBM-eq} and how they fit into the existing literature.

\subsection{Local theory}
Using the standard fixed point arguments we obtain the following:
\begin{theorem} Let $d\geq 1$, $1\leq\beta\leq 2$, $k\geq 2$ and $\mu \in \{-1,1\}$. Then \label{Local_theory:lwp_main}
    \eqref{Gen-BOBBM-eq} is locally well-posed in \(H^s(\mathbb{T}^d)\) for
    \begin{equation*}
        \begin{cases}
            s > \tfrac{d}{2}, & \beta = 1\\
            s \geq \max\left(d\left(\tfrac{1}{2} - \tfrac{1}{k}\right), \tfrac{d}{2} - \tfrac{\beta-1}{k-1}\right), & \beta > 1
        \end{cases}
    \end{equation*}
    with existence time \(T \sim \|u_0\|_{H^s}^{1-k}\).
\end{theorem}
Let $\mathcal{M}=\R$ or $\T$. The local well-posedness theory for the standard BBM equation in $H^s(\mathcal{M})$ was established by Bona and Tzvetkov \cite{bonaSharpWellposednessResults2009} for $s\geq 0$. The theory was then extended by the work of Kim and Kwak in \cite{kim2026} and Bona and Chen in \cite{bona2009}. Kim and Kwak proved local well-posedness of the generalised BBM equations (\eqref{Gen-BOBBM-eq} with $d=1,\beta=2, \mu =1, k\geq 2$) in $H^s(\mathcal{M})$ for $s\geq \tfrac{k-2}{2k}$. In \cite{bona2009}, Bona and Chen established local well-posedness for \eqref{Gen-BOBBM-eq} below the $H^s(\R)$ algebra threshold for
\begin{align*}
    \beta > \max\left(\tfrac{3}{2} - \tfrac{1}{2k},2 - \tfrac{1}{k} \right).
\end{align*}
 To do so, they use Young's inequality on the Fourier side to establish a multilinear estimate. To extend local well-posedness to all $\beta >1$ and bridge the gap between the results of Kim and Kwak and Bona and Chen, we instead make use of a case distinction in our associated product estimate based on the values of $\beta$ and $s$. The first case is where $\beta$ is large enough so that the smoothing from $\phi_\beta(D_x)$ is sufficient to conduct a duality argument. Then, in the second case where the smoothing is weaker, we can estimate the product directly via the fractional Leibniz rule. This allows for finer handling of the worst case high-low frequency interactions. A theme of case seperation dependant on the value of $\beta,k,d$ and $s$ (or as we will call them in our global aguments, \emph{regimes}) will continue throughout the paper.

 As mentioned, at no point during the proof of Theorem \ref{Local_theory:lwp_main} do we use an assumption of finite measure, meaning an identical result will hold in $\R^d$.

Remarking on the sharpness of these results, we discuss two approaches in the literature. In \cite{carvajal2014}, Carvajal and Panthee showed that the flow map $\Phi: H^s(\R)\to C([0,T]; H^s(\R))$ for \eqref{Gen-BOBBM-eq} is not $C^{k}$-differentiable in the sense of Fr\'echet at the origin when $s<\max(0,\tfrac{1}{2}- \tfrac{\beta-1}{k-1})$. Therefore, when
\begin{align*}
    \tfrac{1}{2} - \tfrac{\beta-1}{k-1} \geq \tfrac{1}{2}-\tfrac{1}{k} \iff \beta \leq 2-\tfrac{1}{k}
\end{align*}
and $d=1$, Theorem \ref{Local_theory:lwp_main} is sharp. Furthermore, Kim and Kwak established that when $d=1$, $k\geq 2$, $\beta =2$ and $\mu =1$, Theorem \ref{Local_theory:lwp_main} is sharp (see the proof of Theorem 1.1 in \cite{kim2026}). The argument primarily relies on the boundedness of the dispersion relation of the BBM equation, something that \eqref{Gen-BOBBM-eq} retains for any $1\leq \beta \leq 2$. The reason this approach does not cover the ill-posedness established by Carvajal and Panthee is that the argument specifically requires $s< 1/2 - 1/k$, missing a portion of the scale when $\beta \leq 2- 1/k$. Therefore, the method of Kim and Kwak can be easily adapted to show that Theorem \ref{Local_theory:lwp_main} is sharp for $d=1$, $k\geq 2$, $\mu \in \{-1,1\}$ and $2-1/k< \beta \leq 2$. Combining both of these approaches, we see Theorem \ref{Gen-BOBBM-eq} is sharp in the full range of $1< \beta\leq 2$ when $d=1$. As far as the author is aware, the sharpness of the local theory for $d\geq 2$ is still an open problem.

  Now that we understand the sharpness of Theorem \ref{Local_theory:lwp_main} in dimension 1, we may resolve the following open problem posed by Carvajal and Panthee in \cite{carvajal2014}.
\begin{question}[Carvajal and Panthee \cite{carvajal2014}]
Let $d,\mu =1$. For the following values of $s$, $k$ and $\beta$
\begin{itemize}
    \item \underline{Case 1:} $1\leq \beta \leq 1 + \tfrac{k-1}{2k}$ and $\tfrac{1}{2} - \tfrac{\beta -1}{k-1} \leq s \leq \tfrac{1}{2}$,
    \item \underline{Case 2:} $\tfrac{3}{2}-\tfrac{1}{2k} < \beta \leq 2-\tfrac{1}{k}$ and $\tfrac{1}{2} - \tfrac{\beta -1}{k-1} \leq s < \tfrac{k-1}{2k}$,
    \item \underline{Case 3:} $2- \tfrac{1}{k} < \beta \leq 2$ and $\tfrac{1}{2} - \tfrac{\beta -1}{k-1} \leq s < \tfrac{1}{2} - \tfrac{1}{k}$,
\end{itemize}
is \eqref{Gen-BOBBM-eq} well or ill-posed in $H^s$?
\end{question}
 Using Theorem \ref{Local_theory:lwp_main} we know that \eqref{Gen-BOBBM-eq} is locally well-posed in Case 1 and Case 2. Furthermore, by the sharpness discussed in the previous paragraphs, we see that in Case 3, the solution map for \eqref{Gen-BOBBM-eq} fails to be $C^k$ differentiable, thereby presenting mild ill-posedness.
\subsection{Global theory}
The conservation laws of \eqref{Gen-BOBBM-eq} will be key to our global theory. They are as follows (see Lemma \ref{Global:lemma:Conservervation}),
\begin{align}
         \mathscr{M}_d(u(t))&\coloneqq \int_{\T^d} u(t) \, dx = \mathscr{M}_d(u_0)\label{Conservation:mass}                                                \\
         \mathscr{E}_{\beta,d}(u(t)) &\coloneqq \int_{\T^d} u(t)^2 + u|D_x|^\beta u(t) \, dx = \mathscr{E}_{\beta,d}(u_0) \label{Conservation:H1_like}                      \\
         \mathscr{H}_{k,\mu, d}(u(t))&\coloneqq \int_{\T^d} \tfrac{1}{2}u(t)^2 + \tfrac{\mu}{k(k+1)}u(t)^{k+1}\, dx = \mathscr{H}_{k,\mu,d}(u_0)\label{Conservation:Hamiltonian}.
    \end{align}
For $1<\beta \leq 2$, one can easily obtain global well-posedness of \eqref{Gen-BOBBM-eq} in $H^{\beta/2}(\T^d)$ using the conservation of \eqref{Conservation:H1_like} (see Corollary~\ref{Global:supercritical:Global}). The low regularity theory of \eqref{Gen-BOBBM-eq} begins when $s<\tfrac{\beta}{2}$ and will make heavy use of the Hamilonian conservation \eqref{Conservation:Hamiltonian} which allows us to obtain uniform in time control over the $L^{k+1}(\T^d)$ norm of the solution to \eqref{Gen-BOBBM-eq} for initial data $u_0 \in L^{k+1}(\T^d)$ (See Lemma \ref{Global:critical:general:Lk1_control}). While this is to be expected for the focusing equation ($\mu =1$), the control can also be obtained in the defocusing case ($\mu =-1$) by using the finite measure of $\T$.
\subsubsection{Dimension 1}
 The following theorem is the main result of this paper.
\begin{theorem}\label{Global:thm}
    Let $d=1$, $\mu \in \{-1,1\}$, $k\geq 3$ be an odd integer and $\tfrac{k-2}{k+1} +1 < \beta \leq 2$ be a real number. Then, \eqref{Gen-BOBBM-eq} is globally well-posed in $H^s(\T)$ in the following regimes:
    \begin{itemize}
        \item \underline{Regime 1:} $2-\tfrac{2}{k+1} < \beta \leq 2$ and $\tfrac{1}{2} - \tfrac{1}{k+1} \leq s < \tfrac{\beta}{2}$
        \item \underline{Regime 2:} $k=3$, $\tfrac{5}{4} <\beta \leq \tfrac{3}{2}$ and $\tfrac{9}{8} - \tfrac{\beta}{2}< s <\tfrac{\beta}{2}$
        \item \underline{Regime 3:} $\tfrac{2k-2}{k+1} \leq \beta \leq 2- \tfrac{2}{k+1}$ and $s > \tfrac{1}{2}$.
    \end{itemize}
    Furthermore, we obtain a polynomial growth bound in time on the $H^s$ norm of the solution $u(t)$ for Regimes 1 and 2 and a super exponential bound in time for Regime 3.
\end{theorem}
 To help summarise these results we include the below table.
\begin{table}[H]
    \centering
    \caption{Global Well-posedness of \eqref{Gen-BOBBM-eq}}
    \label{table}
    \renewcommand{\arraystretch}{1.1}
    \begin{tabular}{l cccc}
        \toprule
        & \multicolumn{4}{c}{Nonlinearity $k$} \\
        \cmidrule(lr){2-5}
        Smoothing $\beta$ & $k=2$  & $k=3$ & $k=5$ & $k$ odd \\
        \midrule
        $\beta=1$ & $s\geq \frac{1}{2}$ \cite{angulo2011,he2018} & {\boldmath $s\geq  \frac{1}{2}$} & Unknown & Unknown\\
        $1<\beta\leq \frac{k-2}{k+1}+1$    & $s\geq \frac{3}{2}-\beta$ \cite{wang2018}  & {\boldmath $s>\frac{1}{2}$}   & \textcolor{Mahogany}{$s\geq \frac{\beta}{2}$}  & \textcolor{Mahogany}{$s\geq \frac{\beta}{2}$} \\
        $\frac{k-2}{k+1}+1 < \beta \leq \frac{2k-2}{k+1}$    & $s\geq\frac{3}{2}-\beta$ \cite{wang2018}   & {\boldmath $s> \frac{1}{2}$} & \textcolor{Mahogany}{$s\geq \frac{\beta}{2}$} &  \textcolor{Mahogany}{$s\geq \frac{\beta}{2}$} \\
        $\frac{2k-2}{k+1} \leq \beta \leq 2-\frac{2}{k+1}$    & $s\geq \frac{3}{2}-\beta$ \cite{wang2018}  & {\boldmath $s> \frac{1}{2}$} & {\boldmath $s>\frac{1}{2}$} &  {\boldmath $s>\frac{1}{2}$} \\
        $2-\frac{2}{k+1} <\beta < 2$    & $s\geq 0 \vee \frac{3}{2}-\beta$ \cite{wang2018}   & {\boldmath $s\geq \frac{1}{4}$} & {\boldmath $s\geq  \frac{1}{3}$} &{\boldmath $s\geq \frac{1}{2} - \frac{1}{k+1}$} \\
        $\beta=2$ & $s\geq0$ \cite{bonaSharpWellposednessResults2009} & $s\geq \frac{1}{4}$ \cite{kim2026} & {\boldmath $s\geq  \frac{1}{3}$} & {\boldmath $s\geq \frac{1}{2} - \frac{1}{k+1}$}\\
        \bottomrule
    \end{tabular}
\end{table}
 The entries in \textcolor{Mahogany}{red} denote regimes where as far as the author is aware, there is no known low regularity theory. The `Unknown' entries are where well-posedness results at and above the energy are also unknown. Bolded entries are the contributions of this paper.  In the below tables we display the growth of the $H^s$ norms of the solutions in their respective regimes. We provide the equation numbers for each of the explicit bounds. 
\begin{table}[H]
    \centering
    \caption{$H^s$ norm growth of solution to \eqref{Gen-BOBBM-eq} for $k=3$}
    \resizebox{\columnwidth}{!}{
    \begin{tabular}{l cccc}
        \toprule
        & \multicolumn{4}{c}{Regime} \\
        \cmidrule(lr){2-5}
       Regularity $s$  & $\beta=1$ & $1< \beta \leq 5/4$ &  $5/4<\beta \leq 3/2$ & $3/2 < \beta \leq 2$  \\
        \midrule
        $1/4 \leq s \leq 9/8 -\beta/2$ & Unknown & Unknown & Unknown & {\boldmath $\mathcal{O}(t^{\bullet})$} (\ref{Regime_1:poly_post_gron}, \ref{Regime_1:other_growth_bound})\\
        $9/8 -\beta/2 < s \leq 1/2$ & Unkown & Unknown  & {\boldmath $\mathcal{O}(t^{\bullet})$} \eqref{Reg_2:growth_bound} &  {\boldmath $\mathcal{O}(t^{\bullet})$} \eqref{Regime_1:poly_post_gron}\\
        $s > \frac{1}{2}$ & {\boldmath $\mathcal{O}(\exp\exp(t))$} \eqref{Regime_3_growth_bound2} &{\boldmath $\mathcal{O}(\exp(t^{\bullet}))$} \eqref{Regime_3_growth_bound}   & {\boldmath $\mathcal{O}(t^{\bullet})$} \eqref{Reg_2:growth_bound}  & {\boldmath $\mathcal{O}(t^{\bullet})$} (\ref{Regime_1:poly_post_gron})\\
        \bottomrule
    \end{tabular}}
    \label{table2}
\end{table}
\begin{table}[H]
    \centering
    \caption{ $H^s$ norm growth of solution to \eqref{Gen-BOBBM-eq} for $k\geq 5$ odd}
    \resizebox{\columnwidth}{!}{
    \begin{tabular}{l cccc}
        \toprule
        & \multicolumn{4}{c}{Regime} \\
        \cmidrule(lr){2-5}
       Regularity $s$  & $\beta= \frac{2k-2}{k+1}$ & $\frac{2k-2}{k+1} < \beta < 2 - \frac{2}{k+1}$ & $ \beta = 2 - \frac{2}{k+1}$ &$2-\frac{2}{k+1} < \beta \leq 2$   \\
        \midrule
        $\frac{1}{2} - \frac{1}{k+1} \leq s \leq \frac{1}{2}$ & Unknown  & Unknown &  {\boldmath $\mathcal{O}(\exp(t))$} \eqref{Regime_1:exponential_post_gron} & {\boldmath $\mathcal{O}(t^{\bullet})$} (\ref{Regime_1:poly_post_gron}, \ref{Regime_1:other_growth_bound})\\
        $s > \frac{1}{2}$ & {\boldmath $\mathcal{O}(\exp\exp(t))$} \eqref{Regime_3_growth_bound2} & {\boldmath $\mathcal{O}(\exp(t^{\bullet}))$} \eqref{Regime_3_growth_bound} & {\boldmath $\mathcal{O}(\exp(t))$} \eqref{Regime_3:Plain_exp} & {\boldmath $\mathcal{O}(t^{\bullet})$} (\ref{Regime_1:poly_post_gron})\\
        \bottomrule
    \end{tabular}
    }
    \label{table3}
\end{table}
 
 Here, $t^\bullet$ denotes some finite power of $t$, often depending on $s,\beta, k$ and $\mu$.

  We note that in Regime 1 ($2-\tfrac{2}{k+1} < \beta \leq 2$) we obtain a quadratic growth bound for $\tfrac{1}{2}-\tfrac{1}{k+1}\leq s \leq \beta - \tfrac{3(k+1)-2}{2(k+1)}$, which is \eqref{Regime_1:other_growth_bound} and a more involved polynomial bound \eqref{Regime_1:poly_post_gron} where the exponent depends on $k,s$ and $\beta$. Distinguishing these two cases complicates the tables significantly, so we present them both as a generic polynomial growth. We now discuss the cited results in Table \ref{table} as well as our proofs in the context of their associated regimes.

\underline{Regime 1:}
Regime 1 includes the most well studied cases of \eqref{Gen-BOBBM-eq}, the BBM equation itself,  the cubic and quintic BBM as well as its higher order nonlinearities. The first wellposedness results on the BBM equation came from the original paper of Benjamin \textit{et al.} \cite{benjamin1972a} where they showed that the Cauchy problem for the BBM equation was globally well-posed in $H^k(\R)$ for all $k\in \N$ with $k\geq 1$. The study of the low regularity theory of the BBM equation was initiated by Bona and Tzvetkov, who proved global well-posedness in $H^s(\T)$ and $H^s(\R)$ for $s\geq 0$ using a high-low wavelength decomposition technique that we will discuss in detail in Regime 2. This method was then adapted by Kim and Kwak in \cite{kim2026} where they proved the following:
\begin{enumerate}
    \item The cubic BBM equation (\eqref{Gen-BOBBM-eq}, $\mu=1,d=1$, $\beta=2$, $k=3$) is globally well-posed in $H^s(\mathcal{M})$ for $s\geq \tfrac{1}{4}$.
    \item The quintic BBM equation  (\eqref{Gen-BOBBM-eq}, $\mu=1,d=1$, $\beta=2$, $k=5$) is globally well-posed in $H^s(\mathcal{M})$ for $s> \tfrac{1}{2}$.
    \item There is an exponential growth bound in time on the $H^s$ norm of the global solutions to the cubic and quintic BBM equations.
\end{enumerate}
These results were obtained using the high-low decomposition techniques used in the $k=2$ case. The lack of even power results is because the Hamiltonian, the only\footnote{While the exact amount of conservation laws is not known for the generalised BBM equation, it is shown in  \cite{duzhin1984}  that the standard BBM equation has precisely 3.} conservation law of \eqref{Gen-BOBBM-eq} below $s=1$, fails to control the $L^{k+1}$ norm for even $k$. In \cite[Remark 1.13]{kim2026} Kim and Kwak note that the method of obtaining their global result is also restricted to the cubic and quintic nonlinearity, the reason for which is discussed in Remark \ref{Regime_2:restriction_rem2}. While, Regime 1 of Theorem~\ref{Global:thm} makes no progress with the even power issues, it's novelty is threefold:
\begin{enumerate}
    \item We improve Kim and Kwak's result \cite{kim2026} for the quintic BBM equation from $s>\tfrac{1}{2}$ to $s\geq \tfrac{1}{3}$.
    \item We provide the first (as far as the author is aware) low regularity theory for all odd higher order powers and surrounding fractional cases permittited within the regime.
    \item We obtain polynomial growth bounds on the solution to all equations within this regime.
\end{enumerate}
On the fractional BBM side, Bona and Chen first proved global well-posedness for (\eqref{Gen-BOBBM-eq}, $\mu =1$, $d=1$, $1< \beta \leq 2$, $k=2$) in $H^s(\R)$ when $s \geq 1-\tfrac{\beta}{2}$ \cite{bona2009}. This was later improved by Wang and Zhang in \cite{wang2018} to $s\geq \tfrac{3}{2}-\beta$, closing the gap between the global theory and the sharp local theory established by Carvajal and Panthee  in \cite{carvajal2014}. In this and the following regimes, we extend the fractional theory to the higher order odd nonlinearities. As far as the author is aware, the higher order even powers remain an open problem.

To obtain our results in this regime, we use a physical side energy method approach that involves first deriving an energy identity for a Sobolev space $H^s_\beta$, which is adapated to the dispersion of \eqref{Gen-BOBBM-eq}. Then, we divide into cases and carefully split derivatives in order to use fractional product estimates and interpolation to obtain a differential inequality in terms of the $H^s_\beta$ norm of the solution.  Integrating this inequality, we derive a uniform growth bound on the $H^s_\beta$ norm . Finally, we use a multiplier estimate to transfer this uniform bound to the target $H^s(\T)$ space. We remark that the result holds almost identically for $\R$, except we require that $\mu =1$ as the $\mu =-1$ argument uses the finite measure of $\T$.

\underline{Regime 2:}
Here, we obtain global results only for $k=3$, due to the technical nature of this restriction, we remark on it after the proof of the global result in Remark \ref{Local:remark_verification}. In this regime, we adapt the high-low wavelength decomposition technique used in \cite{kim2026,bonaSharpWellposednessResults2009}. The technique involves splitting the initial data $u_0$ into  a sum of its low frequencies $w_0$ and its high frequencies $v_0$ \footnote{taken high enough so that that the local solution exists to the desired time $T$.}. We then take $v$ as the local solution of \eqref{Gen-BOBBM-eq} correspondig to the initial data $v_0\in H^s$. With this, we show that $w = u-v$ satisfies a difference equation upon which we prove a local and subsequently global result. What makes this technique effective is that $w_0$ is bounded in frequency, meaning it is smooth.  Therefore we may conduct our well-posedness proofs of the difference equation in $H^{\beta/2}$, where we may take advantage of an energy identity to derive a growth bound on the $H^{\beta/2}$ norm of $w$. After establishing the global result for $w$, we then patch $v$ and $w$ back together to form the global solution of the original equation, $u$.

\underline{Regime 3}
In this regime we adapt the method of Angulo \textit{et al.} in \cite{angulo2011} where they prove the global well-posedness of the regularised BO equation (\eqref{Gen-BOBBM-eq} $k=2$, $\beta=1$, $\mu =1$, $d=1$) in $H^s(\R)$ for $s>\tfrac{1}{2}$. To do this they use the Brezis-Gallouet inequality  (Lemma \ref{lem:Brezis-Gallouet}), a sharpening of the Sobolev embedding theorem that allows one to quantify the failure of the $H^{1/2}(\T)\hookrightarrow L^{\infty}(\T)$ embedding. This allows them (and us) to integrate a differential inequality that would otherwise diverge under standard estimates. Furthermore, in our estimates we encounter a $W^{s,p}$ space for $p<1$. To make this step rigorous, we work in the Triebel-Lizokin space $F^s_{p,2}(\T)$ for which we establish a fractional Leibniz rule in the appendix (Proposition \ref{Frac_F_rule}). To do so, we transfer a result of Naibo and Thomposon in \cite{naibo2019} from $\R$ to $\T$. Given $F^s_{p,2}(\T)$ is no longer Banach for $p<1$, this transferrence involves some non-standard tools, with our method being adapted from Triebel's localistation framework outlined in \cite[Chapter 3]{triebel1983}. For more background in the theory of function spaces, see the seminal works of Triebel in \cite{triebel1983,triebel1992,triebel2006,triebel2020}. We show that along with the energy method, this technique can be used to establish a global theory for lower values of $\beta$ and remark on the failure of a simpler technique in Remark \ref{Regime_3:tech_remark}. The proof of this regime relies on the finite measure of $\T$ and thus does not apply to $\R$.
\subsubsection{Endpoint well-posedness}
Specialising \eqref{Gen-BOBBM-eq} to $d=1$, $k=3$ and $\beta=1$ we encounter the \emph{cubic regularised BO equation}
\begin{align}
    \pat_t u + |D_x|\pat_t u + \mu \tfrac{1}{3}(u^3)=0.\label{Intro:cRBO}
\end{align}
So called, as (\eqref{Gen-BOBBM-eq}, $d=1$, $k=2$, $\beta=1$, $\mu =1$) is the regularisation of the Benjamin-Ono (BO) equation. Theorem~\ref{Local_theory:lwp_main} ensures that \eqref{Intro:cRBO} is locally well-posed in $H^s(\T)$ for $s>1/2$. It is known that for $s<1/2$, the solution map fails to $C^3$ differentiable in the sense of Fr\'echet \cite{carvajal2014}, ruling out fixed point methods for such $s$ values. However, it is not known if this result extends to $s=1/2$. On a practical level, the failure of the algebra property of $H^s(\T)$ for $s\leq 1/2$ makes this prospect unlikely. Therefore, to prove local well-posedness of \eqref{Intro:cRBO} in $H^{1/2}$ we take inspiration from the work of He and Mammeri in \cite{he2018} which is itself based on the compactness method pioneered by Yudovich in \cite{judovic1963}. After doing so, we may use the conservation of \eqref{Global:critical:general:case_two:sobolev} to establish the following global result:
\begin{theorem}\label{thm:global:endpoint}
    Let $\mu \in \{-1,1\}$. Then \eqref{Intro:cRBO} is globally well-posed in $H^{1/2}(\T)$.
\end{theorem}
In \cite{angulo2011}, Angulo \textit{et al.}  proved that the \emph{standard} regularised BO equation (\eqref{Gen-BOBBM-eq}, $\mu =1$, $d=1$, $\beta=1$, $k=2$) was globally well-posed in $H^s(\R)$ for $s>1/2$. Following this, He and Mammeri in \cite{he2018} showed that the standard regularised BO equation was globally well-posed in $H^{1/2}(\R)$, using the global solutions obtained by Angulo \textit{et al.} in $H^{1}$ to approximate the solution in $H^{1/2}$. Similarly, to approximate the $H^{1/2}$ solution for the \emph{cubic} regularised BO equation, we rely on the global well-posedness of \eqref{Intro:cRBO} in $H^1$ established by Regime 3 of Theorem \ref{Global:thm}.

The theory for the quintic and higher powers of the regularised BO equation remains unknown as the compactness method used in this paper is insufficient (see Remark \ref{Local:endpoint:failure_remark}).
\subsubsection{Dimension 2}
Here, we specialise to the two dimensional cubic BBM equation (\eqref{Gen-BOBBM-eq},  $d=2, k=3, \beta=2$, $\mu \in \{-1,1\}$)  which we write as the following Cauchy problem:
\begin{align}
    \begin{cases}
        \pat_t u + |D_x|^2 \pat_t u +\left(\pat_x + \pat_y\right)\left(u + \mu \tfrac{1}{3} u^3\right) =0\\
        u(x,0) = u_0(x)  
    \end{cases} x\in \T^2 \label{dim_2:cubic_BBM}.
\end{align}
Our local result, Theorem \ref{Local_theory:lwp_main} tells us that \eqref{dim_2:cubic_BBM} is locally well-posed in $H^s(\T^2)$ for $s\geq \tfrac{1}{2}$. With this in mind, we state the following result.
\begin{theorem}\label{THM:dim_2}
    \eqref{dim_2:cubic_BBM} is globally wellposed in $H^s$ for $s\geq \tfrac{1}{2}$. Moreover, we obtain an exponential growth bound on the $H^s$ norm of the solution for $\tfrac{1}{2} \leq s <2$.
\end{theorem}
 We are inspired by the work of Wang in \cite{wang2024} who established global well-posedness of the BBM equation (with an added parabolic term) on $\T^3$ in $H^s(\T^3)$ for $\tfrac{1}{2}\leq s\leq 1$. To prove this, Wang used the high-low wavelength decomposition technique to establish his result. Our proof of Theorem \ref{THM:dim_2} instead uses an energy method similar to Regime 1 but with a different energy identity. We remark on the pitfalls of using the energy identity from Regime 1 in dimension two in Remark \ref{Dim_two_restriction_remark} as well as the failure of the high-low wavelength decomposition technique from Regime 2. Furthermore, there are restrictions on expanding this result to higher dimensions $d$ and powers $k$ due to the current proof requiring the embedding $L^{k+1}(\T^d)\hookrightarrow H^s(\T^d)$. For $d\geq 3$, this embedding is no longer satisfied for $s<\beta/2$, meaning a low regularity theory is not possible with the current method.
\subsection{$L^{k+1}$ and random data}

In \cite{bona2003}, Bona and Chen proved that a generalised model of \eqref{Gen-BOBBM-eq} in dimension 1 is globally well-posed in $L^2(\R) \cap L^{p}(\R)$ where the admissible range of $p$ depended on the dispersion and the growth rate of the nonlinearity. Given that \eqref{Gen-BOBBM-eq} has a polynomial nonlinearity and fractional $\beta$ dispersion, a special case of the result of Bona and Chen is that for $\mu,d=1$, \eqref{Gen-BOBBM-eq} is globally well-posed in $L^{k+1}(\T)$ for, $\beta > 1 +\tfrac{k-1}{k+1}$. Here, we prove a similar result on $\T^d$ that includes the $\beta = 1+\tfrac{k-1}{k+1}$ endpoint and discuss an application to the almost sure wellposedness of \eqref{Gen-BOBBM-eq} in $H^s(\T^d)$ for $s\geq 0$ with respect to a randomisation of the initial data.
\begin{theorem}\label{Intersection:Lk1:local:thm}
    Let, $d=1$, $\mu \in \{-1,1\}$, $k\geq 3$ be odd and,
    \begin{align*}
        1+ \tfrac{k-1}{k+1} \leq \beta \leq 2.
    \end{align*} 
    Then \eqref{Gen-BOBBM-eq} is globally well-posed in $L^{k+1}(\T)$.
    
    Furthermore, if $d=2$, $\mu \in \{-1,1\}$, $k=3$ and $\beta=2$, then \eqref{Gen-BOBBM-eq} is globally well-posed in $L^{k+1}(\T^2)$.
\end{theorem}
While the general result of Bona and Chen in \cite{bona2003} largely covers the range of $\beta$ in Theorem \ref{Intersection:Lk1:local:thm}, we present a shorter and simpler proof facilitated by the polynomial nonlinearity and the Hamiltonian conservation \eqref{Conservation:Hamiltonian} that also extends to dimension 2. If we let $k=2$, $\beta=2$ and $d=3$, we can recover local well-posedness in $L^{k+1}$ but it does not transfer to the global argument due to $k$ being even. Higher values of $k$ in dimensions $2$ and $3$ are precluded by a lower bound condition on $\beta$, \eqref{Intersection:Lk1:multilinear:beta_assumption} imposed in our multilinear estimate, Lemma \ref{Intersection:Local:multilinear_lem}.

Using Theorem \ref{Intersection:Lk1:local:thm}, we may obtain a different type of well-posedness result for \eqref{Gen-BOBBM-eq} in the form of \emph{almost sure} global well-posedness. The purpose of such a result is to show that \eqref{Gen-BOBBM-eq} is well-posed in $H^s$ for $s$ below the deterministic threshold imposed by Theorem \ref{Local_theory:lwp_main} if the initial data is sufficiently randomised. Here, the randomisation serves to ensure initial data whose solution might blow up in finite time is a set of measure $0$ with respect to a certain probability measure. Consequently, almost every initial condition in the desired low regularity space has a corresponding global solution.

For $u_0\in H^s $, we consider randomised initial data of the form,
\begin{align}
    u_0^\omega(x) = \sum_{n\in \Z}\widehat{u}_0(n)g_n(\omega)e^{in\cdot x}\label{Intro:data_randomisation}
\end{align}
where $\left\{g_n\right\}_{n\in \Z}$ is a sequence of independent, standard, complex valued Gaussian random variables on a probability space $(\Omega, \Prb, \mathcal{F})$ with $g_0\in \R$ and $g_n = \overline{g_{-n}}$ for any $n\in \Z$. We say that a Cauchy problem is \emph{almost surely globally well-posed} if for any $u_0\in H^s(\T^d)$, there exists a set $\Sigma \subset \Omega $ such that $\Prb(\Sigma) = 1$  and for every $\omega \in \Sigma$ there is a unique, stable, global solution to \eqref{Gen-BOBBM-eq} with initial data $u_0^\omega$. The randomisation \eqref{Intro:data_randomisation} was first used by Burq and Tzvetkov in \cite{burq2008} as an alternative to the Gaussian measure used by Bourgain in \cite{bourgain1996}. For more background on almost sure well-posedness see the survey article by B\'enyi \textit{et al.} \cite{benyi2019} and the references therein.    
   
With these definitions, we present the following corollary of Theorem \ref{Intersection:Lk1:local:thm}:

\begin{corollary}\label{cor:almost_sure_global}
    Let $d=1$, $\mu \in \{-1,1\}$, $k\geq 3$ be odd and,
    \begin{align*}
        1+ \tfrac{k-1}{k+1} \leq \beta \leq 2.
    \end{align*}
    Then \eqref{Gen-BOBBM-eq} is almost-surely globally well-posed in $H^s(\T^d)$ for all $s \geq 0$.
    
    Furthermore, if $d=2, \mu \in \{-1,1\}$, $k=3$ and $\beta=2$, \eqref{Gen-BOBBM-eq} is almost surely globally well-posed in $H^s(\T^2)$ for all $s\geq 0$.

\end{corollary}
 To prove this we use the improved integrability estimates that are afforded by a Gaussian randomisation of the initial data's Fourier series modes (see Lemma \ref{Appendix:lem:strichartz}). This ensures that data in $L^2$ have their corresponding randomisations in $L^{k+1}$ almost surely, allowing us to apply Theorem \ref{Intersection:Lk1:local:thm} to obtain a global solution. By the standard persistence of regularity argument, one can then establish the result in $H^s$ for any $s \geq 0$.

\subsection{Ill-posedness}
In this section, we study the ill-posedness of the solution map
\begin{equation}
    \Phi: u_0 \in H^s(\T^d)\to u \in C([0,T]; H^s(\T^d))\label{solution_map}
\end{equation}
where $s<0$, taking initial data $u_0$ to its corresponding solution $u$ to \eqref{Gen-BOBBM-eq}.
Following previous works by Panthee \cite{panthee2011} and Bona and Dai \cite{bona2003}, Forlano \cite{forlano2020}, showed that the solution map \eqref{solution_map} to the BBM equation underwent \emph{norm inflation everywhere}; that is for $s<0$,  $u_0\in H^s(\T)$ and any $\varepsilon>0$, there exists a smooth solution $u_\varepsilon$ to the BBM equation and a time $t_\varepsilon = (0,\varepsilon)$ such that,
\begin{align*}
    \left\|u_\varepsilon(0) - u_0\right\|_{H^s} <\varepsilon \quad \text{and} \quad \left\|u_\varepsilon(t_\varepsilon)\right\|_{H^s} > \varepsilon^{-1}.
\end{align*}
Subsequently, Bhimani and Haque \cite{bhimani2021} showed that for $s<0$, the solution map for the BBM equation was not continuous into any regularity. This phenomenon is known as \emph{infinite loss of regularity} and was first studied by Benejaru and Tao in \cite{bejenaru2006}. Here, we extend the result of Bhimani and Haque to \eqref{Gen-BOBBM-eq}:
\begin{theorem}\label{InfiniteLossofRegs}
    Let $k\geq 2 $, $1\leq \beta\leq 2$, $\mu \in \{-1,1\}$, $s<0$ and $d\geq 1$. Then \eqref{Gen-BOBBM-eq} is ill-posed in $H^s(\T^d)$. In particular, it exhibits an infinite loss of regularity. That is, for $\sigma \in \R$, $s<0$, $u_0\in H^s(\T)$ and any $\varepsilon>0$, there exists a smooth solution $u_\varepsilon$ to \eqref{Gen-BOBBM-eq} and a time $t_\varepsilon = (0,\varepsilon)$ such that,
    \begin{align*}
        \left\|u_\varepsilon(0)-u_0\right\|_{H^s} < \varepsilon \quad \text{and} \quad \left\|u_\varepsilon(t_\varepsilon)\right\|_{H^{\sigma}} > \varepsilon^{-1}.
    \end{align*}
\end{theorem}
 To show this, we follow closely the method of Oh in \cite{oh2017}. The proof relies on a high to low frequency interaction, causing the smoothing of the $\phi_\beta(D_x)$ operator to weaken. The key estimate corresponding to this interaction is the lower bound on the first picard iterate of the solution, established in Lemma \ref{Lem: LOWERBD}. Much of the proof is similar to \cite{oh2017}, however extra combinatorial complexity is introduced from the arbitrary order nonlinearity $k$. To accomodate for this, we use the method developed by Forlano and Okamoto in \cite{forlano2020a}. Theorem \ref{InfiniteLossofRegs} also holds in $\R^d$ but we only present the proof for $\T^d$ for the sake of simplicity. We will be brief with the preliminary multilinear estimates, leaving references to the relevant results when the proofs are identical. For a detailed discussion of this method see \cite{forlano2020a,kishimoto2019,iwabuchi2015, oh2017,forlano2020}.

\section{Preliminaries}
\subsection{Notation}
  We write \(A\lesssim B\) if for some positive constant \(C\), \(A\leq CB\) and \(A\sim B\) if \(A\lesssim B\) and \(B\lesssim A\). For normed spaces $X$ and $Y$, by $X\hookrightarrow Y$, we mean $X$ is continuously embedded in $Y$. Furthermore, we denote the Fourier transform on \(\T^d\) by,
\begin{align*}
    \widehat{u}(n) \coloneqq \frac{1}{(2\pi)^d}\int_{\T^d} e^{-in\cdot x} u(x)\, dx
\end{align*}
for $n\in \Z^d$ so that \(\mathscr{F}\{\nabla u\}(n) = in\widehat{u}(n)\).

Let \(s\in\R\) and \(1\leq p\leq \infty\). We define the \emph{Bessel potential spaces} \(W^{s,p}(\T^d)\) by the norm,
\begin{align*}
    \left\|u\right\|_{W^{s,p}(\T^d)} := \left\|\mathscr{F}^{-1}\left\{\langle n \rangle^s \widehat{u}(n)\right\}\right\|_{L^{p}(\T^d)}
\end{align*}
where \(\langle n\rangle := \left(1+|n|^2\right)^{1/2}\) and  \(\mathscr{F}^{-1}\) denotes the inverse Fourier transform. Furthermore, we define the \emph{Riesz potential spaces} \(\dot{W}^{s,p}(\T^d)\) by the following norm,
\begin{align*}
    \left\|u\right\|_{\dot{W}^{s,p}(\T^d)} := \left\|\mathscr{F}^{-1}\left\{|n|^s \widehat{u}(n)\right\}\right\|_{L^p(\T^d)}.
\end{align*}
For \(p=2\) we write \(W^{s,2}(\T^d):=H^s(\T^d)\) and \(\dot{W}^{s,2}(\T^d) := \dot{H}^s(\T^d)\) for the homogeneous and inhomogeneous \(L^2\) based Sobolev spaces respectively. We also define the real \(L^2(\T^d)\) inner product,
\begin{align*}
    \left\langle u,v\right\rangle_{L^2(\T^d)} = \int_{\T^{d}} uv\,dx.
\end{align*}
Where it is clear from context, we omit the domain from the norm subscript. Furthermore, we define the Dirichlet projectors on $\T^d$ as the Fourier multipliers,
\begin{align*}
    \Pi_{\leq N} u(x) \coloneqq \sum_{|n|\leq N} \widehat{u}(n)e^{in\cdot x}, \quad 
    \Pi_{S} u(x) \coloneqq \sum_{n\in \Z^d}\widehat{u}(n)\mathbbm{1}_{S}(n)e^{in\cdot x}
\end{align*}
where $N \in \N$ and $S \subset \Z^d$. Lastly, we define the Littlewood Paley projectors on $\R^d$ and $\T^d$. Let $1 = \sum_{N} \phi_N(\xi)$ be a smooth, dyadic partition of unity for $\xi \in \R^d$ where $N \in 2^\N$. For $u\in \mathscr{S}'(\R^d)$, define the multiplier $\mathbf{P}_N$ such that $\widehat{\mathbf{P}_N u}(\xi) = \phi_N(\xi) \widehat{u}(\xi)$ for $\xi \in \R^d$. Restricting the symbol to $\Z^d$, we define $\mathbf{P}_N$  for $u\in (C^\infty(\T^d))'$ via  $\widehat{\mathbf{P}_Nu}(n) = \phi_N(n)\widehat{u}(n)$ where $n\in \Z^d$.
\subsection{Embedding theorems and interpolation}
It will be important for us to quantify the extent to which the embedding of $H^{1/2}(\T)$ into $L^{\infty}(\T)$ fails. The following result enables us to do so.
\begin{lemma}[Brezis-Gallouet type estimate {\cite[Lemma 3.2]{angulo2011}}]\label{lem:Brezis-Gallouet}
    For $u \in H^s(\T)$ with $s>1/2$, there exists a constant $C>0$ such that,
    \begin{align*}
        \left\|u\right\|_{L^\infty} \leq C\left(1+\sqrt{\log\left(1+\left\|u\right\|_{H^s}\right)}\left\|u\right\|_{H^{1/2}}\right).
    \end{align*}
\end{lemma}
Furthermore we would like to quantify the relationship between the integrability exponent $p$ of a Lebesgue space $L^p(\T)$ and the correseponding embedding constant of $H^{1/2}(\T)\hookrightarrow L^p(\T)$. As such, we will need a result obtained by G\'erard and Grellier in \cite{gerard2010}.
\begin{lemma}\label{Lem:Trudinger_type}
    There exists a constant $C>0$ such that for all $2<p<\infty$, we have
    \begin{align*}
        \left\|u\right\|_{L^p(\T)}\leq C \sqrt{p}\left\|u\right\|_{H^{1/2}(\T)}.
    \end{align*}
\end{lemma}
Thoughout this paper, the interpolation of $W^{s,p}$ spaces will be an important tool.
\begin{theorem}[{\cite[Section 3.6.1 Theorem 2]{schmeisser1987}}]\label{Interpolation:Brezis:THM}
     Let $0<\theta<1$, $1\leq p, p_0,p_1\leq \infty$, $s,s_1,s_2\in \R$ and $d\in \N$ so that,
    \begin{align*}
        s &= \theta s_2 + (1-\theta)s_2, &
        \frac{1}{p} &= \frac{\theta}{p_1} + \frac{1-\theta}{p_2}\label{Preliminaries:interp:cond}.
    \end{align*}
    Then we have the following,
    \begin{align*}
        \left\|u\right\|_{W^{s,p}(\T^d)} \lesssim \left\|u\right\|_{W^{s_1,p_1}(\T)}^{\theta} \left\|u\right\|^{1-\theta}_{W^{s_2,p_2}(\T^d)}.
    \end{align*}
\end{theorem}
We now present a corollary of the above result that will streamline the common process of first using the Sobolev inequality to embed the desired space on the interpolation line and then using Theorem \ref{Interpolation:Brezis:THM}.
\begin{corollary}\label{Interpolation:GNS:THM}
    Fix \(u\in W^{s_1,p_1}(\T^d)\cap W^{s_2,p_2}(\T^d)\). Assume that the real numbers \(0\leq s_1,s_2,\delta\), \(\theta\in (0,1)\) and \(1\leq p_1,p_2,p \leq \infty\) are such that the following is satisfied,
    \begin{align*}
        \theta\left(s_1-\frac{d}{p_1}\right) + (1-\theta) \left(s_2 - \frac{d}{p_2}\right) & \geq \delta -\frac{d}{p} \\
        \theta s_1 + (1-\theta)s_2                                                                             \geq \delta.
    \end{align*}
    Then the following holds,
    \begin{equation*}
        \left\|u\right\|_{W^{\delta,p}(\T^d)} \lesssim \left\|u\right\|^\theta_{W^{s_1,p_1}(\T^d)} \left\|u\right\|^{1-\theta}_{W^{s_2,p_2}(\T^d)}.
    \end{equation*}
    Furthermore for the case \(W^{s_1,p_1}(\T^d) = H^s(\T^d)\) and \(W^{s_2,p_2}(\T^d) = L^{k+1}(\T^d)\) we have that,
    \begin{equation*}
        \theta = \frac{\delta + \frac{d}{k+1} -\frac{d}{p}}{s-\frac{d}{2} + \frac{d}{k+1}}.
    \end{equation*}
\end{corollary}
\begin{proof}
    This follows from the Sobolev inequality and Theorem \ref{Interpolation:Brezis:THM}.
\end{proof}
We will make frequent use of Corollary \ref{Interpolation:GNS:THM} so we will simply refer to it as \emph{interpolation}. We will also use the following fractional Leibniz rule on $\T^d$.
\begin{lemma}[\cite{benyi2025}]\label{Lem:fractional_Liebniz_rule}
    Let \(1<p_j,q_j \leq \infty\), \(j=1,2\), \(1\leq r <  \infty\) and $s> 0$ such that \(\frac{1}{r} = \frac{1}{p_j} + \frac{1}{q_j}\). Then we have that,
    \begin{align*}
        \left\|uv\right\|_{W^{s,r}} \lesssim \left\|u\right\|_{W^{s,p_1}}\left\|v\right\|_{L^{q_1}} + \left\|u\right\|_{L^{p_1}}\left\|v\right\|_{W^{s,q_2}}.
    \end{align*}
\end{lemma}
We will simply refer to Lemma \ref{Lem:fractional_Liebniz_rule} as the ``fractional Leibniz rule'' as it is standard.
\subsection{Adapted Sobolev spaces and their multipliers}
In this section we will define a Sobolev space adapted to the dispersion of \eqref{Gen-BOBBM-eq} and then prove it's equivalence to the standard $H^s$ norm via a multiplier Lemma.
\begin{definition}
    For $s\in \R$ and $1\leq \beta\leq 2$, we define the spaces $H^s_\beta(\T)$ by the completion of $C^\infty(\T)$ under the norm,
    \begin{align*}
        \left\|u\right\|_{H^s_\beta} := \left\|\langle \pat_x\rangle^{s}_{\beta} u\right\|_{L^2}
    \end{align*}
    where $\langle \pat_x \rangle_\beta$ is a Fourier multiplier with symbol $\langle n\rangle_\beta := \left(1+|n|^\beta\right)^{\frac{1}{\beta}}$. 
\end{definition}
\begin{lemma}\label{Lem:multiplier_lemma}
    For a function $u$ sufficiently regular and integrable, $0<\theta \leq 1$, $1< \beta \leq 2$ and $1<p<\infty$ we have that,
    \begin{align*}
        \left\|\langle \pat_x\rangle^\theta_{\beta} u\right\|_{L^p}\sim \left\|\langle \pat_x\rangle^{\theta} u\right\|_{L^p} \quad \text{and} \quad   \left\|\pat_x u\right\|_{L^p} \lesssim \left\| \langle \pat_x\rangle_\beta u\right\|_{L^p}.\quad 
    \end{align*}
\end{lemma}
\begin{proof}
    This result follows by extending the relevent multilpiers to $\R$, using the Mikhlin-H\"ormander Theorem  (see {\cite[Theorem 6.2.7]{grafakos2014a}}) and then transferrence (see \cite[Theorem 4.3.7]{grafakos2014a}) to obtain the desired boundedness over $\T^d$.
\end{proof}

\section{Local theory}
In this section we prove the local well-posedness result, Theorem \ref{Local_theory:lwp_main}. As discussed in the introduction, we do so via a standard fixed point argument. For the purposes of analysis we write \eqref{Gen-BOBBM-eq} so that the regularisation of the nonlinearity is emphasised:
\begin{align}
     \pat_t u = -(1 + |D_x|^\beta)^{-1} \nabla \cdot \left( u + \mu\tfrac{1}{k} u^k \right).\label{Local:regularisation_form}
\end{align}
We denote the spatial differential operator on the right hand side of \eqref{Local:regularisation_form} by $\phi_\beta(D_x)$, a Fourier multiplier with symbol,
\begin{align*}
    \widehat{\phi_\beta(D_x) u}(n) = \frac{n}{1+|n|^\beta}\widehat{u}(n).
\end{align*}
With this, we may write \eqref{Local:regularisation_form} as,
\begin{align*}
    \pat_t u = -i\phi_\beta(D_x)\left(u +  \mu\tfrac{1}{k} u^k\right).
\end{align*}
This yields the Duhamel formulation from which we will take our solution,
\begin{align*}
    u(t) = S(t)u_0 - \frac{i\mu}{k}\int_0^t S(t-t')\mathcal{N}\left(u(t')\right)\, dt'
\end{align*}
where $S(t) = e^{-it \phi_\beta(D_x)}$ is the linear solution propagator and,
\begin{align*}
    \mathcal{N}\left(u\right) &= \phi_\beta(D_x)u^k.
\end{align*}
We remark that the unitarity of $S(t)$ is clear as it is on the unit circle. With the definitions set up, we may prove the product estimate.
\begin{lemma} \label{lem:multilinear_estimate}
    For \(1\leq \beta\leq 2\), $d\geq 1$ and \(k \ge 2\), with \(\phi_\beta(D_x)\) defined as the Fourier multiplier with symbol $ \phi(n) = \tfrac{n}{1+|n|^\beta}$. For $u_1,\dots,u_k\in H^s(\T^d)$ we have the following product estimate:
    \begin{equation*}
        \Big\|\phi_\beta(D_x) \prod_{i=1}^k u_i\Big\|_{H^s(\T^d)} \lesssim \prod_{i=1}^k\left\|u_i\right\|_{H^s(\T^d)}
    \end{equation*}
    provided:
    \begin{equation}
        \begin{cases}
            s > \frac{d}{2}, & \beta = 1\\
            s \geq \max\left(d\left(\frac{1}{2} - \frac{1}{k}\right), \frac{d}{2} - \frac{\beta-1}{k-1}\right), & \beta > 1.
        \end{cases}\label{Local:multilinear:s_cond}
    \end{equation}
\end{lemma}
\begin{proof}
    For the sake of readability, we prove the case when $u_1=\dots = u_k$, the general case follows in an identical fashion. Inspecting the Fourier multiplier \(\phi_\beta(D_x)\), we note that
    \begin{equation*}
        \widehat{\phi_\beta(D_x)u(n)} = \tfrac{n}{1+|n|^\beta} \widehat{u}(n) \lesssim |n|^{1-\beta} \widehat{u}(n).
    \end{equation*}
    Thus,
    \begin{equation*}
        \left\|{\phi_\beta(D_x) u^k}\right\|_{H^s} = \Big\|{\frac{n}{1+|n|^\beta} \widehat{u^k}(n)}\Big\|_{\ell^2_n} \lesssim \left\|u^k\right\|_{H^{s-\beta+1}}.
    \end{equation*}
    Here, we witness \(\phi_\beta(D_x)\) smoothing an input function by \(\beta-1\) derivatives. As such, it suffices to identify a suitable condition for
    \begin{equation*}
        \left\|u^k\right\|_{H^{s-\beta+1}} \lesssim \left\|u\right\|_{H^s}^k.
    \end{equation*}
    If \(\beta=1\), we observe no smoothing from \(\phi_\beta(D_x)\). This reduces to
    \begin{equation*}
        \left\|u^k\right\|_{H^s} \leq \left\|u\right\|_{H^s}^k.
    \end{equation*}
    This is the algebra property, implying \(H^s(\T^d)\) is a Banach algebra, which is sharply true for \(s > \frac{d}{2}\). Now, assuming that \(\beta>1\) we distinguish two possible cases:

    \underline{Case 1:} Assume that $s-\beta+1 \leq 0$. Here we invoke duality to yield,
    \begin{align}
        \left\|u^k\right\|_{H^{s-\beta+1}}  = \sup_{\left\|g\right\|_{H^{\beta-s-1}} \leq 1} \left|\int_\T u^k g \, dx\right| \leq \sup_{\left\|g\right\| \le 1} \left\|u^k g\right\|_{L^1} \label{Local:multilinear:Minkoski}.
    \end{align}
    Now, using Hölder's inequality:
    \begin{equation*}
        \eqref{Local:multilinear:Minkoski} \le \sup_{\left\|g\right\| \leq 1} \left\|u^k\right\|_{L^{\frac{p}{k}}} \left\|g\right\|_{L^q} = \sup_{\left\|g\right\| \leq 1}\left\|u\right\|_{L^{p}}^k \left\|g\right\|_{L^q}
    \end{equation*}
    where for \(L^{p/k}\) to be a well-defined Banach space we require that \(p\geq k\) and to satisfy Hölder conjugacy we enforce the condition:
    \begin{equation} \label{Local:multilinear:holder_conjugacy}
        \tfrac{k}{p} + \tfrac{1}{q} = 1.
    \end{equation}
    We now seek to verify that we have the following embeddings
    \begin{equation*}
        H^s(\T^d) \hookrightarrow L^{p}(\T^d) \quad \text{and} \quad H^{\beta-s-1}(\T^d) \hookrightarrow L^q(\T^d)
    \end{equation*}
    which would yield the desired estimate like so,
    \begin{align}
        \sup_{\left\|g\right\|_{H^{\beta-s-1}}\leq 1}\left\|u^k g\right\|_{L^1} & \lesssim \sup_{\left\|g\right\|_{H^{\beta-s-1}}\leq 1}\left\|u\right\|^k_{H^s}\left\|g\right\|_{H^{\beta-s-1}} \leq\left\|u\right\|^k_{H^s}\label{Local:multilinear:final_est}.
    \end{align}
    To do this, we use the Sobolev inequality to obtain the condition on \(p\) and \(q\) as such,
    \begin{align*}
        \tfrac{d}{p} \geq \tfrac{d}{2} - s \quad \text{and} \quad \tfrac{d}{q} \geq \tfrac{d}{2} - \beta +s+1.
    \end{align*}
    Simplifying these we get that,
    \begin{align}
        \tfrac{1}{p} \geq \tfrac{1}{2} - \tfrac{s}{d}, \quad \text{and} \quad \tfrac{1}{q} \geq \tfrac{1}{2} + \tfrac{s+1-\beta}{d}\label{Local:multilinear:embedding_cond}.
    \end{align}
    Inspecting the first inequality and using \(p\geq k\) we yield the condition
    \begin{align*}
        \tfrac{1}{2}-\tfrac{s}{d} \leq \tfrac{1}{k} \, \implies \, s\geq d\left(\tfrac{1}{2} - \tfrac{1}{k}\right)
    \end{align*}
    which is satisfied by assumption \eqref{Local:multilinear:s_cond}.
    Now, verifying the validity of our second embedding as well as our H\"older condition, we combine \eqref{Local:multilinear:embedding_cond} with  \eqref{Local:multilinear:holder_conjugacy} to get
    \begin{align*}
        1\geq k\left(\tfrac{1}{2}-\tfrac{s}{d}\right) + \tfrac{1}{2} + \tfrac{s+1-\beta}{d} \iff s\geq \tfrac{d}{2} - \tfrac{\beta-1}{k-1}
    \end{align*}
    which is again satisfied by assumption \eqref{Local:multilinear:s_cond}. As such, \eqref{Local:multilinear:final_est} is verified.

    \underline{Case 2:} Now assume that $s-\beta + 1 >0$. Here, we use the fractional Leibniz rule to obtain,
    \begin{align}
        \left\|u^k\right\|_{H^{s-\beta+1}}\lesssim \left\|u\right\|_{W^{s-\beta+1,p}}\left\|u\right\|^{k-1}_{L^{(k-1)q}}\label{Local:multilinear:case_2:fractional_Liebniz}
    \end{align}
    where we have the H\"older condition,
    \begin{align}
        \tfrac{1}{2}=\tfrac{1}{p} + \tfrac{1}{q}\label{Local:multilinear:case_2:Holder_cond}.
    \end{align}
    It remains to verify the following embeddings,
    \begin{align*}
        H^s(\T^d) \hookrightarrow W^{s-\beta+1,p}(\T^d) \quad \text{and} \quad H^s(\T^d)\hookrightarrow L^{(k-1)q}(\T^d)
    \end{align*}
    as with them, we yield the desired estimate upon returning to \eqref{Local:multilinear:case_2:fractional_Liebniz},
    \begin{align}
        \left\|u\right\|_{W^{s-\beta+1,p}(\T)}\left\|u\right\|^{k-1}_{L^{(k-1)q}} \lesssim \left\|u\right\|^k_{H^s}.\label{Local:multilinear:case_2:final_ineq}
    \end{align}
    Using the Sobolev inequality we get that,
    \begin{align*}
        s-\tfrac{d}{2} \geq s+1-\beta - \tfrac{d}{p} \iff \tfrac{1}{p} \geq  \tfrac{2-2\beta +d}{2d}
    \end{align*}
    and
    \begin{align*}
        \tfrac{d}{(k-1)q} \geq \tfrac{d}{2} -s \iff \tfrac{1}{q} \geq \tfrac{k-1}{d}\left(\tfrac{1}{2} - s \right).
    \end{align*}
    Given that $q\geq 2$ we require,
    \begin{align}
        \tfrac{1}{2} \geq \tfrac{k-1}{d}\left(\tfrac{1}{2} -s\right) \iff s\geq d\left(\tfrac{1}{2} - \tfrac{1}{2(k-1)}\right).\label{Local:multilinear:q_extra}
    \end{align}
    Combining our lower bounds on $1/p$ and $1/q$ in \eqref{Local:multilinear:case_2:Holder_cond} we get that,
    \begin{align*}
        \tfrac{1}{2} \geq \tfrac{2-2\beta+d}{2d} + \tfrac{k-1}{d}\left(\tfrac{d}{2} -s\right) \iff s\geq \tfrac{2-2\beta +d(k-1)}{2(k-1)}
    \end{align*}
    which is satisfied by assumption \eqref{Local:multilinear:s_cond}. It now remains to verify that \eqref{Local:multilinear:q_extra} is implied by our assumptions. Using our case assumption, $s-\beta+1>0$ in \eqref{Local:multilinear:s_cond} we get that,
    \begin{align*}
        s> \tfrac{2-2(s+1) +d(k-1)}{2(k-1)} \iff s > \tfrac{d(k-1)}{2k}.
    \end{align*}
    A quick calculation verifies that,
    \begin{align*}
        \tfrac{d(k-1)}{2k} >d\left(\tfrac{1}{2} - \tfrac{1}{2(k-1)}\right).
    \end{align*}
    Therefore, \eqref{Local:multilinear:q_extra} is satisfed, yielding our desired estimate, \eqref{Local:multilinear:case_2:final_ineq}.
\end{proof}
We may now prove the local result.
\begin{proof}[Proof of Theorem \eqref{Local_theory:lwp_main}]
   This follows simply from Lemma \ref{lem:multilinear_estimate}, the unitarity of $S_\beta(t)$ in $H^s(\T^d)$ and a standard fixed point argument.
\end{proof}
We now make a remark that will simplify the verification of the conditions for local well-posedness in the later sections.
\begin{remark}\label{Local:remark_verification}
    To make sense of the regularity condition \eqref{Local:multilinear:s_cond} we note that in dimension 1, the condition
    \begin{align*}
        s \geq \max\left(\tfrac{1}{2} - \tfrac{1}{k}, \tfrac{1}{2} - \tfrac{\beta-1}{k-1}\right)
    \end{align*}
    reduces to the two respective terms in the following cases
    \begin{align*}
        \beta \geq 1+\tfrac{k-1}{k} \implies s\geq \tfrac{1}{2}-\tfrac{1}{k} \geq  \tfrac{1}{2}-\tfrac{\beta-1}{k-1}
    \end{align*}
    and,
    \begin{align*}
        \beta < 1+\tfrac{k-1}{k} \implies s\geq \tfrac{1}{2}-\tfrac{\beta-1}{k-1} > \tfrac{1}{2}-\tfrac{1}{k}.
    \end{align*}
\end{remark}

\section{Global theory}
The goal of this section is to investigate the conditions under which we may extend our local solution obtained in Theorem \ref{Local_theory:lwp_main}. While much of the arguments in this paper are specialised to their specific dimension and regime there is some underlying architecture and tools that support multiple results. It is now that we develop this architecture which is based in conservation laws (Lemmas \ref{Global:lemma:Conservervation} and \ref{Global:critical:general:Lk1_control}) and the key energy identity, Lemma \ref{global:critical:general:lem_energy}. 
First, we derive the conservation laws of \eqref{Gen-BOBBM-eq}.
\begin{lemma}\label{Global:lemma:Conservervation}
    For $k\geq 2$, $d\geq 1$, $\mu \in \{-1,1\}$ and $1\leq \beta \leq 2$ the quantities \eqref{Conservation:mass}, \eqref{Conservation:H1_like} and \eqref{Conservation:Hamiltonian} are conserved under the flow of \eqref{Gen-BOBBM-eq}.
\end{lemma}
\begin{proof}
    The conservation of \eqref{Conservation:mass} follows by integrating \eqref{Gen-BOBBM-eq} over $\T^d$, integrating by parts and applying the Lebesgue dominated convergence theorem. The conservation of \eqref{Conservation:H1_like} follows from multiplying \eqref{Gen-BOBBM-eq} through by the solution \(u\), integrating by parts and using Lebesgue dominated convergence. For the conservation of the Hamiltonian, \eqref{Conservation:Hamiltonian}, we first set $d=1$ and rewrite \eqref{Gen-BOBBM-eq} as,
    \begin{align*}
        u_t = -\left(1+|D_x|^\beta\right)^{-1}\partial_{x}\left(u+\mu\tfrac{1}{k}u^k\right).
    \end{align*}
    Now, multiplying both sides by $u+ \frac{\mu}{k} u^k$ and integrating over $\T$ we obtain,
    \begin{align}
        \int_{\mathbb{T}} \left(u+\mu\frac{1}{k} u^k\right)u_t\, dx = -\int_{\mathbb{T}} \left(u+\mu\frac{1}{k}u^k\right)\left(1+|D_x|^\beta\right)^{-1} \partial_{x}\left(u+\mu\frac{1}{k}u^k\right)\, dx.\label{Conservation:Hamiltonian_energy_setup}
    \end{align}
    Applying the chain rule to the term on the left hand side we see,
    \begin{align}
        \int_{\mathbb{T}} \left(u+\mu\frac{1}{k} u^k\right)u_t\, dx &= \int_\T \frac{\pat}{\pat t}\left(\frac{1}{2}u^2 + \mu \frac{1}{k(k+1)}u^{k+1}\right) \, dx \notag\\
        &= \frac{d}{dt} \int_\T\frac{1}{2}u^2 + \mu \frac{1}{k(k+1)}u^{k+1} \, dx\label{Conservation:Hamiltonian:LHS}.
    \end{align}
    Now, inspecting the right hand side of \eqref{Conservation:Hamiltonian_energy_setup} we note that the differential operator $L$ defined,
    \begin{align*}
        L = (1+|D_x|^\beta)^{-1} \partial_{x}
    \end{align*}
    is skew-symmetric. Therefore, if we set $v = u+\mu\frac{1}{k}u^k$,
    \begin{align*}
        \int_{\mathbb{T}} \left(u+\mu\frac{1}{k}u^k\right)\left(1+|D_x|^\beta\right)^{-1} \partial_{x}\left(u+\mu\frac{1}{k}u^k\right)\, dx = \int_{\mathbb{T}} v L v \, dx = 0.
    \end{align*}
    Returning this and \eqref{Conservation:Hamiltonian:LHS} to \eqref{Conservation:Hamiltonian_energy_setup} we obtain \eqref{Conservation:Hamiltonian}. This argument can easily be expanded to the higher dimensions by expanding the sum of the $\pat_{x_i}$ operators and defining,
    \begin{align*}
        L_i = (1+|D_x|^\beta)^{-1} \partial_{x_{i}}
    \end{align*}
    which are also skew-symmetric for each $1\leq i\leq d$.
\end{proof}
\noindent From these conservation laws we get our first (albeit trivial) global result. Given that for any $t>0$, the quantity
\begin{align*}
    \mathscr{E}_{\beta,d}(u(t)) = \int_{\T^d} u(t)^2 + u |D_x|^\beta u(t) \, dx = \mathscr{E}_{\beta,d}(u_0)
\end{align*}
 is conserved, we may use the self adjointness of the \(|D_x|\) operator to obtain,
\begin{equation*}
    \left\|u(t)\right\|_{H^{\beta/2}(\T^d)} = \int_{\T^d} u(t)^2 + \left(|D_x|^{\beta/2}u(t)\right)^2\, dx = \int_{\T^d} u(t)^2 + u(t) |D_x|^\beta u(t) \, dx = \mathscr{E}_{\beta,d}(u_0).
\end{equation*}
Giving us global in time control on the \(H^s\) norm of the solution for \(s\geq \beta/2\). This allows us to extend our local solution from Theorem \ref{Local_theory:lwp_main} globally in time. While simple, we stress the logic of the below proof so reasoning of later proofs is more apparent.
\begin{corollary}\label{Global:supercritical:Global}
    For $k\geq 2$, $1< \beta\leq 2$, $\mu \in \{-1,1\}$ and $d \geq 1$ or  $k=3$, $\beta=1$, $\mu \in \{-1,1\}$ and $d=1$,   \eqref{Gen-BOBBM-eq} is Globally well posed in \(H^s(\T^d)\) for \(s\geq \beta/2\).
\end{corollary}
\begin{proof}
    By the above logic, if  \eqref{Gen-BOBBM-eq} is locally well posed in $H^{\beta/2}$ then by a standard iteraction argument, there exists a global solution. Now, we need only to verify the conditions of Theorem \ref{Local_theory:lwp_main}. For $\beta >1$, Remark \ref{Local:remark_verification} will help us in this process as we need to check only two cases for different ranges of $\beta$ and $k$.
    \begin{itemize}
        \item For $\beta \geq 1+\frac{k-1}{k}$ we require that $\frac{\beta}{2} \geq \frac{1}{2} - \frac{1}{k}$. This is clearly satisfied by the positivity of $k$ and the fact that $\beta\geq 1$.
        \item For $\beta < 1+\frac{k-1}{k}$ we require that $\frac{\beta}{2} \geq \frac{1}{2} - \frac{\beta-1}{k-1}$, which is again satisfied by the fact that $\beta\geq 1$ and that $k\geq 2$.
    \end{itemize}
    For the case of $k=3$, $\beta=1$, $\mu \in \{-1,1\}$ and $d=1$  we must use our local well-posedness result from Theorem \ref{THM:Local:endpoint} in $H^{1/2}(\T)$.

    Having checked these cases, we can say that a local solution exists and resultantly a global one. The result for $s>\beta/2$ follows by  the standard argument.
\end{proof}
\noindent Here, the natural energy conservation of the equation directly yields global control over the desired Sobolev norm. Below this regularity is where the analytical troubles begin. As such, pushing the regularity lower than \(s=\frac{\beta}{2}\) will be the primary goal of this section. We will now prove some useful lemmas that will help us in proving global well-posedness.
\begin{lemma}\label{global:critical:general:lem_energy}
    For a smooth solution \(u\) to \eqref{Gen-BOBBM-eq} with data $u_0$ and $d\geq 1$ we have the following,
    \begin{align*}
        \frac{1}{2}\frac{d}{dt} \left\|u \right\|^2_{H^s_\beta(\T^d)} =  -\frac{\mu}{k}\left\langle \langle D_x\rangle^{2s-\beta}_\beta u,\sum_{i=1}^d \pat_{x_i} \left((u)^k\right) \right\rangle_{L^2(\T^d)}
    \end{align*}
    where for $d=1$ the sum on the left hand is simply the first partial derivative in space, $\pat_x$.
\end{lemma}
\begin{proof}
    We multiply \eqref{Gen-BOBBM-eq} through by \(\left(1+|D_x|^{\beta}\right)^{\frac{2s-\beta}{\beta}}u\) to obtain,
    \begin{align*}
        \left(1+|D_x|^{\beta}\right)^{\frac{2s-\beta}{\beta}}u \pat_t(1+|D_x|^\beta) u + \left(1+ |D_x|^\beta\right)^{\frac{2s-\beta}{\beta}} u \sum_{i=1}^d \pat_{x_i}\left(u + \frac{\mu}{k}u^k\right) = 0.
    \end{align*}
    Integrating over \(\T\), we observe that the term containing $\pat_{x_i} u$ is zero for all $1\leq i \leq d$. This yields, 
    \begin{align}
        \int_{\T^d} \left(1+|D_x|^{\beta}\right)^{\frac{2s-\beta}{\beta}}u \pat_t(1+|D_x|^\beta) u \, dx = -\int_{\T^d}\frac{\mu}{k}\left(1+|D_x|^{\beta}\right)^{\frac{2s-\beta}{\beta}}u \sum_{i=1}^d \pat_{x_i} \left(u^k\right) \, dx\label{global:critical:general:strict_energy_equality}.
    \end{align}
    Using the self-adjointness of our Fourier multipliers we obtain,
    \begin{align}
        \left\langle \left(1+|D_x|^\beta\right)^{\frac{2s-\beta}{\beta}} u, \pat_t\left(1+|D_x|^\beta\right)u \right\rangle_{L^2} = \left\langle \langle D_x\rangle^{s}_\beta u, \pat_t \langle D_x\rangle^{s}_\beta u \right\rangle_{L^2}= \frac{1}{2}\frac{d}{dt} \left\|u\right\|_{H^s_\beta}\label{Global:General:critical:Energy_identity:LHS_lower_bound}.
    \end{align}
    Now writing the right hand side of \eqref{global:critical:general:strict_energy_equality} as an inner product,
    \begin{align*}
        -\frac{\mu}{k}\int_{\T^d} \left(1+|D_x|\right)^{\frac{2s-\beta}{\beta}} u\sum_{i=1}^d \pat_{x_i} \left(u^k\right)\, dx = -\frac{\mu}{k}\left\langle \langle D_x\rangle^{2s-\beta}_\beta u,\sum_{i=1}^d \pat_{x_i} \left(u^k\right) \right\rangle_{L^2}.
    \end{align*}
    Combining the above with \eqref{Global:General:critical:Energy_identity:LHS_lower_bound} in \eqref{global:critical:general:strict_energy_equality} we get the stated result.
\end{proof}
\begin{lemma}\label{Global:critical:general:Lk1_control}
    For \(k\) odd and \(1\leq \beta \leq 2\), a solution $u$ to \eqref{Gen-BOBBM-eq} with initial data $u_0\in L^{k+1}(\T^d)$  we have that for $t>0$,
    \begin{equation}
        \left\|u(t)\right\|_{L^{k+1}} \lesssim_{k,\mu,d}  \mathscr{H}_{k,1,d}(u_0)\label{Global:critical:general:Lk1_ineq}
    \end{equation}
    where $\mathscr{H}_{k,1,d}$ is the Hamiltonain of \eqref{Gen-BOBBM-eq} defined \eqref{Conservation:Hamiltonian}.
\end{lemma}
\begin{proof}
    If \(\mu =1\), our Hamiltonian conservation, \eqref{Conservation:Hamiltonian} gives us that,
    \begin{align*}
        \mathscr{H}_{k,1,d}(u(t)) = \int_{\T^d} \frac{1}{2}u(t)^2 + \frac{1}{k(k+1)}u(t)^{k+1}\, dx = \int_{\T^d} \frac{1}{2}u_0^2 + \frac{1}{k(k+1)}u_0^{k+1}\, dx = \mathscr{H}_{k,1,d}(u_0)
    \end{align*}
    for all $t \in [0,T]$.
    Given \(k\) is odd, $u(t)^{k+1}\geq 0$ for all $x\in \T^d$. As such, we may use the definition of the \(L^{k+1}\) norm to obtain \eqref{Global:critical:general:Lk1_ineq}.

    Now we consider the \(\mu =-1\) case. Multiplying \eqref{Conservation:Hamiltonian} by \(-1\) we get that,
    \begin{align}
        \int_{\T^d} \frac{1}{k(k+1)}u(t)^{k+1} \, dx -\int_{\T}\frac{1}{2}u(t)^2\, dx = \int_{\T^d} \frac{1}{k(k+1)}u_0^{k+1} \, dx -\int_{\T}\frac{1}{2}u_0^2\, dx\label{u_Lham_bounded:mu_minus:ident}
    \end{align}
    Now, using Young's inequality on our $u(t)^2$ integral term we obtain,
    \begin{align*}
        \frac{1}{2}\int_{\T^d} u(t)^2 \, dx \leq C_\varepsilon |\T^d| + \varepsilon \int_{\T^d} u^{k+1} \, dx
    \end{align*}
    for some \(\varepsilon >0\) and \(C_\varepsilon>0\) depending on \(\varepsilon\). Using this in \eqref{u_Lham_bounded:mu_minus:ident} we get that,
    \begin{align*}
        \left(\frac{1}{k(k+1)} - \varepsilon\right) \int_{\T^d}u^{k+1}\, dx \leq \int_{\T^d} \frac{1}{2}u_0^2 + \frac{1}{k(k+1)}u_0^{k+1}\, dx + C_\varepsilon|\T^d|.
    \end{align*}
    Setting \(\varepsilon < \frac{1}{k(k+1)}\) and rearanging, we obtain \eqref{Global:critical:general:Lk1_ineq}.
\end{proof}
In regimes where we use the energy method, namely Regime 1, Regime 3 and dimension 2, we will need to define a sequence of smooth solutions converging to our rough solution. That is, for $u_0\in H^s$ and $\varepsilon>0$, we define $u_0^\varepsilon = u_0 *\eta_\varepsilon$ where where $\eta_\varepsilon$ is the standard euclidean mollifier. Furthermore, we define the corresponding local evolution of $u_0^\varepsilon$ under \eqref{Gen-BOBBM-eq} as $u^\varepsilon$. By the standard theory we have that $u_0^\varepsilon \to u$ strongly in $H^s$ and $L^p$ for any $s\geq 0$ and $1\leq p < \infty$. Therefore, we have that,
\begin{align}
    \mathscr{E}_{\beta,d}(u_0) &= \lim_{\varepsilon \to 0}\mathscr{E}_{\beta,d}(u_0^\varepsilon)\label{global:energy_limit}\\
    \mathscr{H}_{k,\mu,d}(u_0) &= \lim_{\varepsilon \to 0} \mathscr{H}_{k,\mu,d}(u_0^\varepsilon)\label{global:Hamiltonian_limit}.
\end{align}
Furthemore, by the stability of the local solution map we get that $u^\varepsilon \to u$ strongly in $H^s$ as $\varepsilon \to 0$. With this we yield,
\begin{align}
    \left\|u\right\|_{H^s} = \lim_{\varepsilon \to 0}\left\|u^\varepsilon\right\|_{H^s}.\label{Global:solution_limit_Hs}
\end{align}

\subsection{Regime 1}
In this section we present the proof to Regime 1 of Theorem \ref{Global:thm}. We now state the main proof of this subsection.
\begin{proposition}[Regime 1]\label{Regime_1_THM}
    Let $d =1$, $\mu \in \{-1,1\}$, $k\geq 3$ be an odd integer and $2-\frac{2}{k+1} \leq \beta \leq 2$. Then \eqref{Gen-BOBBM-eq} is globally  wellposed in $H^s$ for,
    \begin{align}
        \tfrac{1}{2} - \tfrac{1}{k+1} \leq s < \tfrac{\beta}{2}.\label{Regime_1:reg_assumption}
    \end{align}
    Furthermore, we obtain polynomial growth bounds (\ref{Regime_1:poly_post_gron}, \ref{Regime_1:other_growth_bound}) for the $H^s$ norm of the solution $u$ if the corresponding initial condition $u_0\in H^s$.
\end{proposition}
\begin{proof}
    Clearly, \eqref{Regime_1:reg_assumption} satisfies both of the conditions stated in Remark \ref{Local:remark_verification}. Therefore, for $u_0 \in H^s$ there exists a unique corresonding local solution $u$ to \eqref{Gen-BOBBM-eq}. Given that \(s> \frac{1}{2} - \frac{1}{k+1}\), the Sobolev inequality guarantees that \(H^s \hookrightarrow L^{k+1}\). As such, Lemma \ref{Global:critical:general:Lk1_control} gives uniform control on the \(L^{k+1}\) norm of \(u\). Furthermore, we consider the mollified data $u_0^\varepsilon$ and $u^\varepsilon$, it's corresponding local evolution under \eqref{Gen-BOBBM-eq}.

    \underline{Case 1:} Assume that \(\frac{\beta-1}{2} <s < \frac{\beta}{2}\). Starting with Lemma \ref{global:critical:general:lem_energy} and Parseval's identity, we estimate
    \begin{align},
        \frac{1}{2}\frac{d}{dt} \left\|u^\varepsilon\right\|^2_{H^s_\beta} & \lesssim  \left|\left\langle \langle \pat_x\rangle_\beta^{2s-\beta}u^\varepsilon,\pat_x \left((u^\varepsilon)^k\right) \right\rangle_{L^2}\right|\notag                                                                                  \\ & = \left|\left\langle \langle\pat_x\rangle^\sigma_\beta u^\varepsilon, \langle \pat_x\rangle_\beta^{2s-\beta -\sigma}\pat_x\left((u^\varepsilon)^k\right)  \right\rangle_{L^2}\right|\notag                                                        \\
                                                             & \leq \left\|\langle \pat_x\rangle_\beta^\sigma u^\varepsilon\right\|_{L^{p}} \left\| \langle \pat_x \rangle_\beta^{2s - \beta - \sigma}\pat_x \left((u^\varepsilon) ^k\right)\right\|_{L^q}\notag
    \end{align}
    for some \(\sigma \in \R\) and \(1\leq p,q\leq \infty\) such that,
    \begin{equation}\label{Global:critical:general:first_holder_condition}
        1 = \tfrac{1}{p} + \tfrac{1}{q}.
    \end{equation}
    We now use Lemma \ref{Lem:multiplier_lemma} to obtain,
    \begin{align}
          \left\|\langle \pat_x\rangle_\beta^\sigma u^\varepsilon\right\|_{L^{p}} \left\| \langle \pat_x \rangle_\beta^{2s - \beta - \sigma}\pat_x \left((u^\varepsilon) ^k\right)\right\|_{L^q} \lesssim   \left\|\langle \pat_x\rangle^\sigma u^\varepsilon\right\|_{L^{p}} \left\| \langle \pat_x \rangle^{2s - \beta - \sigma+1} \left((u^\varepsilon) ^k\right)\right\|_{L^q}\label{Global:critical:general:case_1:first_holder}.
    \end{align}
    Focusing on the nonlinearity term in \eqref{Global:critical:general:case_1:first_holder} we assume that \(2s-\beta+1 - \sigma >0\) in order to use the fractional Leibniz rule like so,
    \begin{align}
        \left\| \langle \pat_x \rangle^{2s - \beta+1 - \sigma} \left((u^\varepsilon) ^k\right)\right\|_{L^q} & \lesssim \left\|(u^\varepsilon)^{k-1}\right\|_{L^{q_1}}\left\|\langle \pat_x\rangle^{2s-\beta+1-\sigma} u^\varepsilon\right\|_{L^{q_2}}\notag                                                   \\
                                                                                                   & = \left\|u^\varepsilon\right\|^{k-1}_{L^{q_1(k-1)}} \left\|\langle \pat_x\rangle^{2s-\beta+1-\sigma} u^\varepsilon\right\|_{L^{q_2}}\label{Global:critical:general:case_1:Fractional_liebniz}
    \end{align}
    giving us another H\"older integrability condition which, when combined with \eqref{Global:critical:general:first_holder_condition} gives us
    \begin{equation*}
        1 = \tfrac{1}{p} + \tfrac{1}{q_1}+ \tfrac{1}{q_2}.
    \end{equation*}
    Choosing these explicitly,
    \begin{align*}
        p & = k+1 & q_1 & = \tfrac{k+1}{k-1} & q_2 & = k+1.
    \end{align*}
    Now, combining \eqref{Global:critical:general:case_1:Fractional_liebniz} with \eqref{Global:critical:general:case_1:first_holder} we get,
    \begin{align}
        \frac{1}{2}\frac{d}{dt}\left\|u^\varepsilon\right\|_{H^s_\beta} & \lesssim \left\|\langle \pat_x\rangle^\sigma u^\varepsilon\right\|_{L^{k+1}}\left\|u^\varepsilon\right\|^{k-1}_{L^{k+1}}\left\|u^\varepsilon\right\|_{W^{2s-\beta+1-\sigma, k+1}}\notag\\
        &= \left(\text{Term 1}\right)\cdot \left(\text{Term 2}\right)\cdot \left(\text{Term 3}\right)\label{Global:critical:case_1:terms_expansion}.
    \end{align}
    \begin{itemize}
        \item \underline{Term 1:} Using Corollary \ref{Interpolation:GNS:THM} for Term 1 we have that,
    \begin{align}
        \left\|\langle \pat_x\rangle^\sigma u^\varepsilon\right\|_{L^p} \lesssim \left\|u^\varepsilon\right\|_{H^s}^{\theta_1}\left\|u^\varepsilon\right\|_{L^{k+1}}^{1-\theta}\label{Global:critical:general:case_1:theta_1_ineq}.
    \end{align}
    where,
    \begin{align*}
        \theta_1 = \frac{\sigma}{s-\frac{1}{2} + \frac{1}{k+1}}.
    \end{align*}
    Checking the validity of the interpolation we require,
    \doubleeq{ \theta_1 >0 \iff \sigma>0\label{Global:Regime_1:sigma_lb_1}}{ \theta_1 < 1 \iff \sigma < s-\tfrac{1}{2} + \tfrac{1}{k+1}\label{Global:Regime_1:sigma_ub_1}}
    There is room for such a $\sigma$ to exist given in this case $s > \tfrac{\beta-1}{2}$.
    \item \underline{Term 2:} Here, we already have control from Lemma \ref{Global:critical:general:Lk1_control}. Using this, we yield,
    \begin{align}
        \left\|u^\varepsilon\right\|_{L^{k+1}}^{k-1} &\lesssim \mathscr{H}_{k,1,1}(u_0^\varepsilon)^{k-1}\label{Global:general:case_1:term_2_ineq}.
    \end{align}
    \item \underline{Term 3:} Interpolating Term 3 we obtain,
    \begin{align}
        \left\|u^\varepsilon\right\|_{W^{2s-\beta+1-\sigma,q_2}}\lesssim \left\|u^\varepsilon\right\|_{H^s}^{\theta_3}\left\|u^\varepsilon\right\|^{1-\theta_3}_{L^{k+1}}\label{Global:critical:general:case_1:theta_3_ineq}.
    \end{align}
    with
    \begin{align*}
        \theta_3 = \frac{2s-\beta+1-\sigma}{s-\frac{1}{2}+\frac{1}{k+1}}.
    \end{align*}
    Again, checking the validity of this interpolation, we require that,
    \doubleeq{ \theta_3 >0 \iff \sigma < 2s-\beta+1 \label{Regime_1:sigma_ub_2}}{\theta_3 < 1 \iff \sigma > s-\beta + \tfrac{3}{2} - \tfrac{1}{k+1}\label{Regime_1:sigma_lb_2}}
    Using these upper and lower bounds to check that there is room for $\sigma$ to exist,
    \begin{align*}
        s-\beta + \tfrac{3}{2} - \tfrac{1}{k+1} < 2s-\beta+1 \iff s > \tfrac{1}{2} - \tfrac{1}{k+1}
    \end{align*}
    which is satisfied by assumption \eqref{Regime_1:reg_assumption}.
    \end{itemize}
     We must now also check that our upper bound from Term 1 \eqref{Global:Regime_1:sigma_ub_1} agrees with the lower bound from Term 3 \eqref{Regime_1:sigma_lb_2} and vice versa with \eqref{Global:Regime_1:sigma_lb_1} and \eqref{Regime_1:sigma_ub_2}
    \begin{align*}
        s-\beta + \tfrac{3}{2} -\tfrac{1}{k+1} <s-\tfrac{1}{2} + \tfrac{1}{k+1} &\iff \beta >2-\tfrac{2}{k+1}\\
        0 < 2s-\beta+1 &\iff s > \tfrac{\beta-1}{2}
    \end{align*}
    both of which are satisfied by assumption. As such, we may take,
    \begin{align*}
        \max\left(0,s-\beta+\tfrac{3}{2} - \tfrac{1}{k+1}\right) < \sigma < \min\left(s-\tfrac{1}{2} + \tfrac{1}{k+1},2s-\beta+1\right).
    \end{align*}
    Returning to \eqref{Global:critical:case_1:terms_expansion} with \eqref{Global:critical:general:case_1:theta_1_ineq}, \eqref{Global:general:case_1:term_2_ineq} and \eqref{Global:critical:general:case_1:theta_3_ineq} we get that,
    \begin{align*}
        \frac{1}{2}\frac{d}{dt}\left\|u^\varepsilon\right\|^2_{H^s_\beta} \lesssim \mathscr{H}_{k,1,1}(u_0)^{k-1} \left\|u^\varepsilon\right\|_{H^s}^{\theta_1 + \theta_3}\left\|u^\varepsilon\right\|_{L^{k+1}}^{(1-\theta_1)+(1-\theta_3)}
    \end{align*}
    where we have used \eqref{global:Hamiltonian_limit} to ensure the Hamiltonian constant is in terms of $u_0$. Using Lemma \ref{Global:critical:general:Lk1_control} on the remaining $L^{k+1}$ and invoking Lemma \ref{Lem:multiplier_lemma} to get a $H^s_\beta$ norm on the right hand side we obtain
    \begin{align*}
        \frac{1}{2}\frac{d}{dt}\left\|u^\varepsilon\right\|^2_{H^s_\beta} \lesssim \mathscr{H}_{k,1,1}(u_0^\varepsilon)^{1-\alpha +k}\left\|u^\varepsilon\right\|_{H^s_\beta}^{\alpha}
    \end{align*}
    where $\alpha \coloneqq \theta_1+\theta_2$. In order to use Gr\"onwall's inequality on this expression we require that $\alpha \leq 2$ which is equivalent to our assumption that $\beta \geq 2-\frac{2}{k+1}$. Therefore, if $\alpha =2$, meaning $\beta = 2- \frac{2}{k+1}$, we have the exponential growth bound for any $t\in [0,T]$ after using Grownall's inequality, \eqref{global:Hamiltonian_limit}, \eqref{Global:solution_limit_Hs} to take $\varepsilon\to 0$ and finally Lemma \ref{Lem:multiplier_lemma} to change the $H^s_\beta$ norm to $H^s$,
    \begin{align}
        \left\|u(t)\right\|_{H^{s}} \leq \left\|u_0\right\|_{H^s_\beta}\exp\left(C(s,k,\beta,\mu)\mathscr{H}_{k,1,1}(u_0)^{k-1}t\right)\label{Regime_1:exponential_post_gron}
    \end{align}
    where $C(s,k,\beta,\mu)>0$. If $\alpha <2$ we get a polynomial growth bound after integrating the inequality and again using \eqref{global:Hamiltonian_limit}, \eqref{Global:solution_limit_Hs} and Lemma \ref{Lem:multiplier_lemma},
    \begin{align}
        \left\|u(t)\right\|_{H^s} \leq \left(\left\|u_0\right\|_{H^s_\beta}^{1-\alpha} + C(s,k,\beta,\mu) \mathscr{H}_{k,1,1}(u_0)^{k-1}(1-\alpha)t\right)^{\frac{2}{2-\alpha}}\label{Regime_1:poly_post_gron}.
    \end{align}
    Applying Lemma \ref{Lem:multiplier_lemma} to \eqref{Regime_1:exponential_post_gron} and \eqref{Regime_1:poly_post_gron} we obtain the stated growth bounds and thus the global result.
    
    \underline{Case 2:} Now assume that  \(\frac{1}{2} - \frac{1}{k+1} \leq s\leq \beta - \frac{3(k+1)-2}{2(k+1)}\).  Using Lemma \ref{global:critical:general:lem_energy} and Parseval's identity, we estimate,
    \begin{align}
        \frac{1}{2}\frac{d}{dt} \left\|u^\varepsilon\right\|^2_{H^s_\beta} &\leq \left|\left\langle \pat_x \langle\pat_x\rangle_\beta^{2s-\beta} u^\varepsilon, (u^\varepsilon)^k \right\rangle_{L^2}\right|\leq \left\|\pat_x\langle\pat_x\rangle_\beta^{2s-\beta}u^\varepsilon\right\|_{L^{k+1}} \left\|(u^\varepsilon)^k\right\|_{L^{\frac{k+1}{k}}} \notag         \\
        & \sim \left\|u^\varepsilon\right\|_{W^{2s-\beta+1,k+1}}\left\|u^\varepsilon\right\|^k_{L^{k+1}}\label{Global:critical:general:case_two:first_ineq}
    \end{align}
    where in the second line we used H\"older's inequality and in the third we used Lemma \ref{Lem:multiplier_lemma} to exchange the $\langle \pat_x\rangle_\beta$ multiplier for a $\langle \pat_x \rangle$ multiplier.
    Now, we will use the Sobolev inequality to verify that \(H^s(\T)\hookrightarrow W^{2s-\beta+1,k+1}(\T)\) i.e.,
    \begin{equation}\label{Global:critical:general:case_two:sobolev}
        \left\|u\right\|_{W^{2s-\beta+1,k+1}} \lesssim \left\|u\right\|_{H^s}.
    \end{equation}
    This gives us the following condition,
    \begin{align*}
        s-\tfrac{1}{2} \geq 2s - \beta +1 - \tfrac{1}{k+1}
    \end{align*}
    which is equivalent to our assumption that \(s\leq \beta +\frac{2-3(k+1)}{2(k+1)}\). As such, using \eqref{Global:critical:general:case_two:sobolev} in \eqref{Global:critical:general:case_two:first_ineq} we get that,
    \begin{align*}
        \frac{1}{2}\frac{d}{dt} \left\|u^\varepsilon\right\|^2_{H^s_\beta} \lesssim \left\|u^\varepsilon\right\|_{H^s} \left\|u^\varepsilon\right\|^k_{L^{k+1}}.
    \end{align*}
    From this, we use our uniform control of the $L^{k+1}$ norm from Lemma \ref{Global:critical:general:Lk1_control} and Lemma \ref{Lem:multiplier_lemma} to obtain,
    \begin{equation*}
        \frac{1}{2} \frac{d}{dt} \left\|u^\varepsilon\right\|^2_{H^s_\beta} \lesssim  \mathscr{H}_{k,1,1}(u_0^\varepsilon)^k\left\|u\right\|_{H^s_\beta}.
    \end{equation*}
    As such, integrating the inequality and using \eqref{global:Hamiltonian_limit} and \eqref{Global:solution_limit_Hs} to take $\varepsilon\to 0$, we yield the  growth bound,
    \begin{equation}\label{Regime_1:other_growth_bound}
        \left\|u(t)\right\|_{H^s_\beta} \leq \left(\left\|u_0\right\|_{H^s_\beta}^{\frac{1}{2}} +  C(s,k,\beta,\mu)\mathscr{H}_{k,1,1}(u_0)^k t\right)^2
    \end{equation}
    where $C(s,k,\beta,\mu)>0$. Once again, we use Lemma \ref{lem:multilinear_estimate} interchange the $H^s_\beta$ norm with the $H^s$ norm yielding the stated growth bound and thereby global well-posedness in the stated regime.

    Now, to tie our cases together and obtain global wellposedness and polynomial growth for \(\frac{1}{2} - \frac{1}{k+1} \leq s<\frac{\beta}{2}\) we demand that,
    \begin{align*}
        \tfrac{\beta-1}{2} \leq \beta -\tfrac{3(k+1)-2}{2(k+1)}
    \end{align*}
    which is indeed satisfied due to our \(\beta \geq 2-\frac{2}{k+1}\) assumption, completing the proof.
\end{proof}
\begin{remark}
    We note that if $u(t)\in L^{k+1}$ for $t\in [0,T]$ where $T$ is the local existence time of $u$, the above proof forces no lower bound on $s$ other then the regularity needed to ensure local existence. This is because, the regularity condition in Theorem \ref{Regime_1_THM} that $s\geq \frac{1}{2} - \frac{1}{k+1}$ is simply to ensure that $H^s(\T)\hookrightarrow L^{k+1}(\T)$. If we begin with this assumption, we may weaken the regularity to the local condition. Although this result will not be directly used, it will be helpful to keep this in mind when we discuss the well-posedness of \eqref{Gen-BOBBM-eq} in $L^{k+1}$ itself.
\end{remark}

\subsection{Regime 2}
As mentioned in the introduction, for this regime, we use a high-low wavelength decomposition argument inspired by \cite{bonaSharpWellposednessResults2009} and \cite{kim2026}. Throughout, fix $d=1$, $\mu \in \{-1,1\}$ and $k=3$. Let \(s\) satisfy the regularity requirements of Theorem \ref{Local_theory:lwp_main} and take \(u_0\in H^s\).  Fixing a \(T>0\), we let \(N\gg 1\) such that,
\begin{align*}
    \left\|\Pi_{\geq N} u_0\right\|_{H^s} \lesssim T^{-\frac{1}{2}}.
\end{align*}
Evolving \(v_{0,N} = \Pi_{\geq N}u_0\) through \eqref{Gen-BOBBM-eq} we have by our local well-posedness result, Theorem \ref{Local_theory:lwp_main} that there exists a solution \(v_{0,N} \in H^s\) to \eqref{Gen-BOBBM-eq} with existence time $T_{v_N \text{existence}}$ satisfying,
\begin{align*}
    T_{v_N \, \text{existence}} \sim \left\|v_{0,N}\right\|^{-2}_{H^s} \gtrsim T.
\end{align*}
Furthermore, we define \(w_{0,N} = u_0- v_{0,N}\) as the low frequency component of \(u_0\), which we can also write as \(w_{0,N} = \Pi_{<N}u_0\). Given that \(w_{0,N}\) is bounded in frequency \(w_{0,N}\in H^{\beta/2}(\T)\). Therefore, formally, we consider the evolution of $w_{0,N}$, $w_N = u -v_N$ as satisfying the following difference equation
\begin{align}
    \begin{cases}
        \pat_t\left(1+|\pat_x|^\beta\right)w_N + \pat_x w_N + \frac{1}{3}\sum_{i=1}^3 \binom{3}{i} \pat_x \left(w^i_N v^{3-i}_N\right) = 0 \\
        w_N(x,0) = w_{0,N}(x).
    \end{cases}\label{Global:Frequency_cut:general:PDE}
\end{align}
The aim of this regime will be to first prove that we can obtain a local solution to \eqref{Global:Frequency_cut:general:PDE} in \(H^{\beta/2}(\T)\) and then a global one. This will allow us to combine our \(v\) and \(w\) solutions to form a global solution \(u\) to \eqref{Gen-BOBBM-eq}. Now, we note that \eqref{Global:Frequency_cut:general:PDE} has the Duhamel formulation,
\begin{align*}
    w_N(x,t) = S_\beta(t)w_{0,N} - i\int_0^t S_\beta(t-t') \mathcal{N}_{v_N}(w_N(t'))\, dt'
\end{align*}
where \(S_\beta(t)\) is the linear propagator of our solution and,
\begin{align*}
    \mathcal{N}_{v_N}(w_N(t)) := \frac{1}{3}\sum_{i=1}^3 \binom{3}{i} \phi_\beta(\pat_x)w_N^i v_N^{3-i}.
\end{align*}
We first prove a multilinear result that will be key to the local well-posedness of \eqref{Global:Frequency_cut:general:PDE}.
\begin{lemma}\label{Frequency_cut:product_estimate}
    Let \(1<\beta\leq 3/2\) and $1\leq i\leq 3$. For \(w_1,w_2,w_3 \in H^{\beta/2}\) and \(v_1,v_2,v_3 \in H^s\), there is the multilinear estimate
    \begin{align*}
        \Big\|\phi_\beta(\pat_x) \prod_{j=1}^i w_j \prod_{l=i}^3 v_l\Big\|_{H^{\beta/2}}\lesssim \prod_{j=1}^i\left\|w\right\|_{H^{\beta/2}} \prod_{l=i}^3\left\|v\right\|_{H^s}
    \end{align*}
    if the following condition on \(s\) and $\beta$ is satisfied,
    \begin{align*}
        s \geq \tfrac{3-\beta}{4}.
    \end{align*}
\end{lemma}
\begin{proof}
    We prove the case when $v_1=v_2=v_3$ and $w_1=w_2 = w_3$, the general case follows in an identical fashion. We begin by noticing that \(\phi_\beta(\pat_x)\) smooths by \(\beta-1\) derivatives,
    \begin{align*}
        \left\|\phi_\beta(\pat_x)w^i v^{3-i}\right\|_{H^{\beta/2}}\sim \left\|w^i v^{3-i}\right\|_{H^{1-\beta/2}}.
    \end{align*}
    Given that \(1-\beta/2>0\) we use the Fractional Libeniz rule to obtain,
    \begin{align}
        \left\|w^i v^{3-i}\right\|_{H^{1-\beta/2}} & \lesssim \left\|\langle \pat_x\rangle^{1-\beta/2} w^i\right\|_{L^{p_1}}\left\|v^{3-i}\right\|_{L^{q_1}} + \left\|w^{i}\right\|_{L^{\infty}}\left\|\langle \pat_x\rangle^{1-\beta/2}v^{3-i}\right\|_{L^2}\notag. \\
                                                       & = \text{Term }1 + \text{Term }2\label{Frequency_cut:Mulitilinear:Split}
    \end{align}
    where $p_1$ and $q_1$ satisfy the H\"older condition
    \begin{align}
         \tfrac{1}{2}=\tfrac{1}{p_1}+\tfrac{1}{q_1}\label{Global:frequency_cut:p_1_q_1_holder_cond}.
    \end{align}
    We inspect each of these terms separately and from them derive our necessary conditions on \(s\) and \(\beta\).
    \begin{itemize}
        \item \underline{Term 1:} Here, we use the fractional Leibniz rule again to obtain,
              \begin{align}
                  \left\|\langle \pat_x \rangle^{1-\beta/2}w^i\right\|_{L^{p_1}}\lesssim \left\|\langle\pat_x \rangle^{1-\beta/2}w\right\|_{L^{p_1}}\left\|w\right\|^{i-1}_{L^\infty}.\label{Frequency_cut_local:frac_lieb_key_step}
              \end{align}
              We demand the embedding \(H^{\beta/2}(\T)\hookrightarrow W^{1-\beta/2,p_1}(\T)\) which gives us the condition that,
              \begin{align}
                  \tfrac{\beta}{2} - \tfrac{1}{2} \geq 1-\tfrac{\beta}{2}-\tfrac{1}{p_1} \iff \tfrac{1}{p_1} \geq \tfrac{3}{2} - \beta. \label{Global:frequency_cut:multilinear:term1_p1_cond}
              \end{align}
              Here, \(p_1\) is well defined as $\beta>1$. When $i<3$, we also require that \(H^s(\T) \hookrightarrow L^{q_1(3-i)}(\T)\), which yields the condition,
              \begin{align}
                  \tfrac{1}{(3-i)q_1} \geq \tfrac{1}{2}-s \iff \tfrac{1}{q_1} \geq \left(3-i\right)\left(\tfrac{1}{2} - s\right).\label{Global:Frequency_cut:multilinear:term1_Lq1_cond}
              \end{align}
              For $L^{q_1}$ to be a well-defined Lebesgue space, we require that
              \begin{align}
                   \left(3-i\right)\left(\tfrac{1}{2} - s\right)\leq 1 \iff s\geq 0 \label{Global:Frequency_cut:mutlilinear:term1_well_def_leb}
              \end{align}
              which is satisfied by assumption. Therefore,
              \begin{align*}
                  \left\|v\right\|^{3-i}_{L^{q_1}}\lesssim \left\|v\right\|^{3-i}_{H^s}.
              \end{align*}
              Combining conditions \eqref{Global:frequency_cut:multilinear:term1_p1_cond} and \eqref{Global:Frequency_cut:multilinear:term1_Lq1_cond} in our H\"older condition \eqref{Global:frequency_cut:p_1_q_1_holder_cond} we obtain,
              \begin{align}
                  \tfrac{1}{2} \geq \tfrac{3}{2}- \beta + (3-i)\left(\tfrac{1}{2}-s\right)  \iff  s \geq \tfrac{2-\beta}{2} \label{Frequency_cut:local:term_1:fin_cond}
              \end{align}
              and thus,
              \begin{align}
                  \text{Term }1 & \lesssim \left\|w\right\|_{H^{\beta/2}}\left\|w\right\|^{i-1}_{L^\infty}\left\|v\right\|^{3-i}_{H^s}  \lesssim \left\|w\right\|^i_{H^{\beta/2}}\left\|v\right\|^{3-i}_{H^s}\label{Frequency_cut:local:term_1_ineq}
              \end{align}
              where the final inequality follows by the Sobolev embedding as $\beta>1$.
        \item \underline{Term 2:} We first use the Sobolev embedding to get that,
              \begin{align*}
                  \left\|w^i\right\|_{L^\infty} \lesssim \left\|w\right\|^i_{H^{\beta/2}}.
              \end{align*}
              If $i<3$, for the other norm in the product we use the Fractional Leibniz rule,
              \begin{align*}
                  \left\|\langle \pat_x\rangle^{1-\beta/2}v^{3-i}\right\|_{L^2} & \lesssim \left\|\langle \pat_x\rangle^{1-\beta/2} v \right\|_{L^{p_3}} \left\|v\right\|^{2-i}_{L^{q_3}}
              \end{align*}
              where we have the following H\"older condition,
              \begin{align}
                  \tfrac{1}{2} = \tfrac{1}{p_3} + \tfrac{2-i}{q_3}\label{Frequencycut:Multilinear:term_2:holder_cond}.
              \end{align}
              Now, we first require that \(H^s\hookrightarrow W^{1-\beta/2,p_3}\) which gives us the condition,
              \begin{align}
                  s-\tfrac{1}{2} = 1-\tfrac{\beta}{2} - \tfrac{1}{p_3} \iff \tfrac{1}{p_3} = \tfrac{3}{2} - \tfrac{\beta}{2} -s. \label{Frequencut:Multilinear:term_2:cond_1}
              \end{align}
              We know that $L^{p_3}$ is well defined as $\beta<3/2$. Now requiring that \(H^s \hookrightarrow L^{q_3}\) we yield,
              \begin{align}
                  \tfrac{1}{q_3} \geq \tfrac{1}{2} -s.\label{Frequencut:Multilinear:term_2:cond_2}
              \end{align}
              Combining conditions, \eqref{Frequencut:Multilinear:term_2:cond_1} and \eqref{Frequencut:Multilinear:term_2:cond_2} in \eqref{Frequencycut:Multilinear:term_2:holder_cond} we get that,
              \begin{align}
                  \tfrac{1}{2} \geq \tfrac{3}{2}-\tfrac{\beta}{2} - s+(2-i)(\tfrac{1}{2}-s) \iff s\geq \tfrac{3-\beta}{4}\label{Frequency_cut:local:term_2:fin_cond}
              \end{align}
              and resultingly,
              \begin{align}
                  \text{Term }2 \lesssim \left\|w\right\|^i_{H^{\beta/2}}\left\|v\right\|_{H^s}^{k-i}.\label{Frequency_cut:local:term2_ineq}
              \end{align}
    \end{itemize}
    Returning our findings \eqref{Frequency_cut:local:term_1_ineq} and  \eqref{Frequency_cut:local:term2_ineq} to \eqref{Frequency_cut:Mulitilinear:Split} we see that,
    \begin{align*}
        \left\|w^i v^{3-i}\right\|_{H^{1-\beta/2}} \lesssim \left\|w\right\|^i_{H^{\beta/2}}\left\|v\right\|^{3-i}_{H^s}
    \end{align*}
    which is what we wanted to show. We also note that our condition \eqref{Frequency_cut:local:term_2:fin_cond} is strictly stronger then our \eqref{Frequency_cut:local:term_1:fin_cond} and \eqref{Global:Frequency_cut:mutlilinear:term1_well_def_leb} conditions due to the \(1<\beta\leq 3/2\) assumption.
\end{proof}
\begin{remark}\label{Regime_2:restriction_remark}
    If we consider arbitrary $k\geq 2$, the proof of Lemma \eqref{Frequency_cut:product_estimate} is easily adapted to obtain the stated product estimate for
    \begin{align*}
        s \geq \tfrac{k-\beta}{2(k-1)}.
    \end{align*}
    Since we must specialise to the $k=3$ case in our global argument however (see Remark \ref{Regime_2:restriction_rem2}) the estimate in this more general form serves no use. 
\end{remark}
\begin{remark}
    Attempts to generalise Lemma \ref{Frequency_cut:product_estimate} to dimension 2 have been unsuccessful. The embedding of $H^{\beta/2} \hookrightarrow L^\infty$ fails in dimension 2 which is key to closing the estimate in the dimension 1 proof. If one instead attempts to avoid $L^\infty$ in \eqref{Frequency_cut_local:frac_lieb_key_step} by introducing exponents $p_3$ and $q_3$ such that,
    \begin{align*}
        \left\|\langle D_x \rangle^{1-\beta/2}w^i\right\|_{L^{p_1}}\lesssim \left\|\langle D_x \rangle^{1-\beta/2}w\right\|_{L^{p_3}}\left\|w\right\|^{i-1}_{L^{q_3}}
    \end{align*}
    the resulting embeddings in dimension 2 require,
    \begin{align*}
        s\geq 2- \tfrac{\beta-1}{2}.
    \end{align*}
    This leaves us above the low regularity threshold of $\beta/2$ meaning no progress can be made in the corresponding global argument.
\end{remark}
\begin{lemma}\label{thm:local_w_freq}
    Let $\mu \in \{-1,1\}$  and $1<\beta \leq 3/2$. \eqref{Global:Frequency_cut:general:PDE} is locally well-posed in $H^{\beta/2}$  if $v\in H^s$ such that,
    \begin{align*}
        s\geq \tfrac{3-\beta}{4}.
    \end{align*}
    Moreover, the corresponding solution $w_N(t)$ has existence time
    \begin{align*}
        T_{w_N \, \text{existence}} = \left\|w_{0,N}\right\|_{H^{\beta/2}}^{2} \wedge T_{v_N \, \text{existence}}
    \end{align*}
    where $T_{v_N \, \text{existence}}$ is the existence time of $v_N(t)$ from Theorem \ref{Local_theory:lwp_main}.
\end{lemma}
\begin{proof}
 The local well-posedness of \eqref{Global:Frequency_cut:general:PDE} in \(H^{\beta/2}\) follows from our product estimate Lemma \ref{Frequency_cut:product_estimate}. Using Remark \ref{Local:remark_verification} to verify that this condition exceeds the local existence conditions of $v_N$. We note that in Regime 2 we are solely in the case where, $\beta < \tfrac{5}{3}$. Therefore, we require that,
\begin{align*}
     \tfrac{3-\beta}{4} \geq \tfrac{2-\beta}{2}
\end{align*}
which is satisfied by our assumptions on $\beta$.
\end{proof}
We also show that we have uniform control over the $L^{4}$ norm of $w$.
\begin{remark}\label{Global:general:w_conservation}
    As seen in Lemma \ref{Global:critical:general:Lk1_control}, if we have initial data \(u_0\in L^{4}\) (i.e. \(u_0\in H^s\) for \(s\geq 1/4\)), we have uniform in time control of the \(L^{4}\) of the solution \(u\) to \eqref{Gen-BOBBM-eq}. Given that \(v_N\) is also a solution to
    \eqref{Gen-BOBBM-eq} we have uniform control of it's \(L^{4}\) norm, provided that \(v_{0,N}\in L^{4}\). Furthermore, using that \(w_N = u-v_N\) and the triangle inequality,
    \begin{align*}
        \left\|w_N(t)\right\|_{L^{4}} = \left\|u(t)-v_N(t)\right\|_{L^{k+1}} &\leq \left\|u(t)\right\|_{L^{4}}+\left\|v_N(t)\right\|_{L^{4}}\\
         &\lesssim  \mathscr{H}_{3,1,1}(u_0)\vee \mathscr{H}_{3,1,1}(v_{0,N})\\
        &\sim \mathscr{H}_{3,1,1}(u_0)
    \end{align*}
    where $t\in [0,T]$. Therefore, we obtain uniform control also  on the \(L^{4}\) norm of \(w_N\).
\end{remark}
 We now proceed to our global result below \(\beta = 3/2\) for $k=3$ and $d=1$.
\begin{proposition}\label{Regime_2:prop_result}
    Let $\mu\in \{-1,1\}$, $k=3$, $d=1$ and $5/4 < \beta \leq 3/2$. \eqref{Gen-BOBBM-eq} is globally well-posed in $H^s(\T)$ for,
    \begin{align*}
       \tfrac{9}{8} - \tfrac{\beta}{2} < s<\tfrac{\beta}{2}.
    \end{align*}
    Furthermore, we obtain a polynomial growth bound on the $H^s$ norm of our solution $u$.
\end{proposition}
\begin{proof}
    Fix \(T>0\) and take $N\gg 1$ such that,
    \begin{align*}
        \left\|v_{0,N}\right\|_{H^s} = \left\|\Pi_{\geq N}u_0\right\|_{H^s} \lesssim T^{-\frac{1}{2}}.
    \end{align*}
    Letting $T_{v_N \, \text{existence}}$ be the existence time of the local solution $v_N(t)\in H^s(\T)$ to \eqref{Gen-BOBBM-eq} from Theorem \ref{Local_theory:lwp_main}, we recall
    that,
    \begin{align*}
        T_{v_N \, \text{existence}} \lesssim \left\|v_{0,N}\right\|_{H^s}^{-2}.
    \end{align*}
    Therefore, $T_{v_N \, \text{existence}} \lesssim T$, meaning there is a solution $v_N(t)$ on $[0,T]$ such that,
    \begin{align}
       \sup_{t\in [0,T]}\left\|v_N(t)\right\|_{H^s} \lesssim T^{-\frac{1}{2}}.\label{Global:frequency_cut:v_T_bound}
    \end{align}
    Now that we have fixed $N$, we denote $w_N$ and $v_N$ as $w$ and $v$ respectively. Let $w$ be the local solution from Lemma \ref{thm:local_w_freq} with existence time $T_{w \text{existence}}$ where $v\in H^s$ such that, $ s \geq \tfrac{3-\beta}{4}$ which is satisfied by assumption. We now show that the solution $w$ of \eqref{Global:Frequency_cut:general:PDE} can be extended to time \(T\). After doing this, we may combine our solutions \(v\) and \(w\) to form the global solution \(u =v+w\) to \eqref{Gen-BOBBM-eq}. Before we begin, it is important to note that we require $s\geq 1/4$ in order to ensure the embedding \(H^s(\T)\hookrightarrow L^{4}(\T)\) by the Sobolev inequality. Given $\beta >1$, this is also satisfied by assumption. This means we have uniform in time control over the \(L^{4}\) norm of \(v\) by our Hamiltonian conservation \eqref{Conservation:Hamiltonian} and also uniform in time control over the \(L^{4}\) norm of \(w\) by Remark \ref{Global:general:w_conservation}.

    Multiplying \eqref{Global:Frequency_cut:general:PDE} by the solution \(w\), integrating over \(\T\) and integrating by parts we get,
    \begin{align*}
        \int_\T w \pat_t\left(1+|\pat_x|^\beta\right) w \, dx = \frac{1}{3}\sum_{i=1}^{3} \binom{3}{i} \int_{\T} w^{i}v^{3-i}\pat_x w\, dx.
    \end{align*}
    After an application of Parseval's identity on the left hand side we see that,
    \begin{align}
        \frac{1}{2}\frac{d}{dt}\left\|w\right\|^2_{H^{\beta/2}}\lesssim \frac{1}{3}\sum_{i=1}^{3} \binom{3}{i}\int_{\T}w^i v^{3-i}\pat_x w \, dx.\notag
    \end{align}
    We first note that the $i=3$ term vanishes by an integration by parts. Furthermore, in the sense of Fourier multipliers, \(\pat_x = - \mathcal{H} |\pat_x|\) where $\mathcal{H}$ is the Hilbert transform. Using this, we split the derivative in the $i=1,2$ case in the following way,
    \begin{align*}
        \int_{\T}w^i v^{3-i}\pat_x w \, dx & = -\int_{\T}\left(|\pat_x|^{1-\beta/2} w^i v^{3-i}\right)\mathcal{H}|\pat_x|^{\beta/2}w \, dx.
    \end{align*}
    Therefore we have,
    \begin{align}
        \frac{1}{2}\frac{d}{dt}\left\|w\right\|^2_{H^{\beta/2}} \lesssim \sum_{i=1}^2\left|\int_{\T}\left(|\pat_x|^{1-\beta/2} w^i v^{3-i}\right)\mathcal{H}|\pat_x|^{\beta/2}w \, dx\right|.\label{Global:lower:master_lesssim}
    \end{align}
    We set out to bound the right hand side integral by a power of the \(H^{\beta/2}\) norm of \(w\) less than $2$ up to a constant that is uniform in time. Fixing $i\in \{1,2\}$,
    \begin{align}
        \left|\int_{\T} \left(|\pat_x|^{1-\beta/2} w^i v^{3-i}\right)\mathcal{H} |\pat_x|^{\beta/2} w\, dx\right| & \lesssim \left\|w^i v^{3-i}\right\|_{\dot{H}^{1-\beta/2}} \left\|w\right\|_{\dot{H}^{\beta/2}} \lesssim \left\|w^i v^{3-i}\right\|_{H^{1-\beta/2}} \left\|w\right\|_{H^{\beta/2}}\label{Global:First_split}.
    \end{align}
    With a single power of our target space norm we now need to estimate the nonlinear term. To do so, we first use the fractional Leibniz rule (noting that \(1-\beta/2 >0\)) to obtain,
    \begin{align}
        \left\|w^i v^{3-i}\right\|_{H^{1-\beta/2}} & \lesssim \left\|w^i\right\|_{W^{1-\beta/2,p_1}}\left\|v^{3-i}\right\|_{L^{q_1}} + \left\|w^i\right\|_{L^{p_2}}\left\|v^{3-i}\right\|_{W^{1-\beta/2, q_2}}\notag \\
                                                       & = \text{Term 1} + \text{Term 2}\label{Global:term_A_term_B_split}
    \end{align}
    where we have the H\"older condition,
    \begin{align}
        \tfrac{1}{2} = \tfrac{1}{p_1} + \tfrac{1}{q_1} = \tfrac{1}{p_2} + \tfrac{1}{q_2}\label{Global:low:first_hold_cond}.
    \end{align}
    We will now treat these two terms separately,
    \begin{itemize}
        \item \underline{Term 1:} For this term we make further distinction,
              \begin{align*}
                  \text{Term 1} = \left\|w^i\right\|_{W^{1-\beta/2,p_1}}\left\|v^{3-i}\right\|_{L^{q_1}} = (\text{Term 1.1})\cdot(\text{Term 1.2}).
              \end{align*}
              \begin{itemize}
                  \item \underline{Term 1.1:} If \(i=1\), we interpolate to obtain,
                        \begin{align*}
                            \left\|w\right\|_{W^{1-\beta/2,p_1}} \lesssim \left\|w\right\|_{H^{\beta/2}}^{\theta^*_1}\left\|w\right\|_{L^{4}}^{1-\theta^*_1}
                        \end{align*}
                        where,
                        \begin{align*}
                            \theta^*_1 = \frac{1-\frac{\beta}{2} + \frac{1}{4} - \frac{1}{p_1}}{\frac{\beta}{2} - \frac{1}{2} + \frac{1}{4}}.
                        \end{align*}
                        Inorder for the interpolation to work we require that \(0<\theta^*_1<1\),
                        \doubleeqflex{.40}{.50}{\theta^*_1 >0 \iff \tfrac{1}{p_1}<\tfrac{5}{4}-\tfrac{\beta}{2}\label{Global:low:special_term_1.1_p_1_upper}}{\theta^*_1 < 1 \iff \tfrac{1}{p_1}> \tfrac{3}{2}-\beta\label{Global:low:special_term_1.1_p_1_lower}}
                        Therefore, for there to exist such a \(1/p_1\) we require that there is space between our upper and lower bounds which is satisfied by our $\beta>1$ assumption. Moving on to the \(i=2\) case, we encounter yet more terms after an application of the fractional Leibniz rule to Term 1.1,
                        \begin{align*}
                            \left\|w^2\right\|_{W^{1-\beta/2,p_1}} \lesssim  \left\|w\right\|_{W^{1-\beta/2,p_3}}\left\|w\right\|_{L^{q_3}}= (\text{Term 1.1.1})\cdot (\text{Term 1.1.2})
                        \end{align*}
                        where we have the H\"older condition,
                        \begin{align}
                            \tfrac{1}{p_1} = \tfrac{1}{p_3} + \tfrac{1}{q_3}.\label{Global:low:term1.1_holder}
                        \end{align}
                        \begin{itemize}
                            \item \underline{Term 1.1.1:} Here, we encounter an interpolation identical to our \(i=1\) case for Term 1.1 except with \(p_3\) instead of \(p_1\). That is, we have,
                                  \begin{align*}
                                      \left\|w\right\|_{W^{1-\beta/2,p_3}} \lesssim \left\|w\right\|^{\theta_1}_{H^{\beta/2}}\left\|w\right\|^{1-\theta_1}_{L^{k+1}}
                                  \end{align*}
                                  where,
                                  \begin{align*}
                                      \theta_1 = \frac{1-\frac{\beta}{2} + \frac{1}{4} - \frac{1}{p_3}}{\frac{\beta}{2} - \frac{1}{2} + \frac{1}{4}}.
                                  \end{align*}
                                  Furthermore, by \eqref{Global:low:special_term_1.1_p_1_lower} we get that,
                                  \doubleeqflex{0.5}{0.4}{ \tfrac{1}{p_3} < \tfrac{5}{4} - \tfrac{\beta}{2}\label{Global:low:term1.1.1_p_3_upper}}{\tfrac{1}{p_3} > \tfrac{3}{2}-\beta\label{Global:low:term1.1.1_p_3}}
                            \item \underline{ Term 1.1.2:} Interpolating we yield,
                                  \begin{align*}
                                      \left\|w\right\|_{L^{q_3}}\lesssim \left\|w\right\|_{H^{\beta/2}}^{\theta_2}\left\|w\right\|_{L^{k+1}}^{(1-\theta_2)}
                                  \end{align*}
                                  where,
                                  \begin{align*}
                                      \theta_2 = \frac{\frac{1}{4} - \frac{1}{q_3}}{\frac{\beta}{2} - \frac{1}{2} + \frac{1}{4}}.
                                  \end{align*}
                                  Checking the validity of our interpolation,
                                  \doubleeqflex{0.45}{0.45}{\theta_2>0 \iff \tfrac{1}{q_3} <\tfrac{1}{4}\label{Global:low:term1.1.2_q_3_upper}}{\theta_2 <1 \iff \tfrac{1}{q_3} > 0\label{Global:low:term1.1.2_q_3_lower}.}
                                  Clearly there is room for such a \(q_3\) to exist.
                        \end{itemize}
                        Combining our conditions \eqref{Global:low:term1.1.1_p_3} and \eqref{Global:low:term1.1.2_q_3_lower} in \eqref{Global:low:term1.1_holder} we obtain,
                        \begin{align*}
                            \tfrac{1}{p_1} > \tfrac{3}{2}-\beta
                        \end{align*}
                        and doing the same for \eqref{Global:low:term1.1.1_p_3_upper} and \eqref{Global:low:term1.1.2_q_3_upper},
                        \begin{align}
                            \tfrac{1}{p_1} < \tfrac{3}{2} -\tfrac{\beta}{2}\label{Global:lower:p_1upper}.
                        \end{align}
                        Verifying that there is enough room,
                        \begin{align*}
                            \tfrac{3}{2}-\beta < \tfrac{3}{2} -\tfrac{\beta}{2} \iff \tfrac{\beta}{2} > 0
                        \end{align*}
                        which is satisfied given \(\beta>1\).
                  \item \underline{Term 1.2:} Interpolating \(v\) between \(H^s\) and \(L^{4}\),
                        \begin{align*}
                            \left\|v^{3-i}\right\|_{L^{q_1}} = \left\|v\right\|^{3-i}_{L^{(3-i)q_1}} \lesssim \left\|v\right\|^{(3-i)\theta_3}_{H^s}\left\|v\right\|^{(3-i)(1-\theta_3)}_{L^{4}}
                        \end{align*}
                        where
                        \begin{align*}
                            \theta_3 = \frac{\frac{1}{4} - \frac{1}{q_1(3-i)}}{s-\frac{1}{2} + \frac{1}{4+1}}.
                        \end{align*}
                        Checking the interpolation's validity,
                        \doubleeqflex{0.3}{0.6}{ \theta_3 > 0 \iff \tfrac{1}{q_1} < \tfrac{3-i}{4}, \notag }{\theta_3<1 \iff \tfrac{1}{q_1} > (3-i)\left(\tfrac{1}{2} - s\right).\label{Global:low:term1.2:q_1_condition}}
                        Checking if there exists such a \(q_1\),
                        \begin{align*}
                            (3-i)\left(\tfrac{1}{2} - s\right) < \tfrac{3-i}{4} \iff s> \tfrac{1}{4}
                        \end{align*}
                        which is satisfied by assumption. Given that \(3-i\) is maximised for \(i=1\), we can write \eqref{Global:low:term1.2:q_1_condition} as,
                        \begin{align}
                            \tfrac{1}{q_1} > 2\left(\tfrac{1}{2} - s\right)\vee 0\label{Global:low:term1.2:q_1_lower_bound}.
                        \end{align}
              \end{itemize}
              Now, combining our \eqref{Global:low:special_term_1.1_p_1_lower} and \eqref{Global:low:term1.2:q_1_lower_bound} conditions in our H\"older condition \eqref{Global:low:first_hold_cond} we obtain,
              \begin{align*}
                  \tfrac{1}{2} > \tfrac{3}{2}-\beta + 2\left(\tfrac{1}{2} -s\right)\vee 0.
              \end{align*}
              For \(s>1/2\),
              \begin{align*}
                  \tfrac{1}{2} > \tfrac{3}{2}-\beta \iff \beta> 1
              \end{align*}
              which is satisfied by assumption. For \(s<1/2\),
              \begin{align*}
                  \tfrac{1}{2} > \tfrac{3}{2} - \beta + 2\left(\tfrac{1}{2} -s\right) \iff s> 1 - \tfrac{\beta}{2}
              \end{align*}
              which is satisfied by our assumption that $s>9/8-\beta/2$. Now, writing our expanded form of Term 1 we see that for \(i=1\),
              \begin{align}
                  (\text{Term 1})\cdot\left\|w\right\|_{H^{\beta/2}}  &\lesssim \left\|w\right\|_{H^{\beta/2}}^{\theta_1^*+1}\left\|w\right\|^{(1-\theta^*_1)}_{L^{4}}\left\|v\right\|^{2\theta_3}_{H^s}\left\|v\right\|^{2(1-\theta_3)}_{L^{4}}\notag \\
                                                                         & \lesssim T^{-\theta_3}\mathscr{H}_{3,1,1}(u_0)^{(1-\theta_1^*)} \mathscr{H}_{3,1,1}(v_0)^{2(1-\theta_3)}\left\|w\right\|_{H^{\beta/2}}^{\theta_1^* +1}\label{Global:lower:expansion_term1_one}
              \end{align}
              where we have used Remark \ref{Global:general:w_conservation}, Lemma \ref{Global:critical:general:Lk1_control}, and \eqref{Global:frequency_cut:v_T_bound} to control the supremum of the $H^s$ norm of $v$. Doing the same for \(i=2\),
              \begin{align}
                  (\text{Term 1})\cdot\left\|w\right\|_{H^{\beta/2}}  &\lesssim \left\|w\right\|_{H^{\beta/2}}^{\theta_1 + (i-1)\theta_2 + 1}\left\|w\right\|_{L^{4}}^{(1-\theta_1) + (i-1)(1-\theta_2)}\left\|v\right\|_{H^s}^{(3-i)\theta_3}\left\|v\right\|_{L^{4}}^{(3-i)(1-\theta_3)}\notag \\
                                                                         & \lesssim  T^{-\frac{(3-i)\theta_3}{2}}\mathscr{H}_{3,1,1}(u_0)^{1-\theta_1 + (i-1)(1-\theta_2)}\notag\\
                                                                         &\hphantom{XXXXXXXXXXX}\times \mathscr{H}_{3,1,1}(v_0)^{(3-i)(1-\theta_3)} \left\|w\right\|_{H^{\beta/2}}^{\theta_1 + (i-1)\theta_2 +1}\label{Global:lower:expansion_term1_i}.
              \end{align}

        \item \underline{Term 2:} We now turn our attention to Term 2 from \eqref{Global:term_A_term_B_split} where we estimate the two norms in the product seperately,
              \begin{align*}
                  \text{Term 2} =\left\|w^i\right\|_{L^{p_2}}\left\|v^{3-i}\right\|_{W^{1-\beta/2,q_2}} \eqqcolon (\text{Term 2.1})\cdot(\text{Term 2.2}).
              \end{align*}
              \begin{itemize}
                  \item \underline{Term 2.1:} Interpolating we get,
                        \begin{align*}
                            \left\|w^i\right\|_{L^{p_2}} = \left\|w\right\|^i_{L^{ip_2}} \lesssim \left\|w\right\|^{i\theta_4}_{H^{\beta/2}}\left\|w\right\|^{i(1-\theta_4)}_{L^{4}}
                        \end{align*}
                        where,
                        \begin{align*}
                            \theta_4 & = \frac{\frac{1}{4} - \frac{1}{ip_2}}{\frac{\beta}{2} -\frac{1}{2}+\frac{1}{4}}.
                        \end{align*}
                        Verifying the interpolation's validity,
                        \doubleeqflex{0.45}{0.45}{ \theta_4>0 \iff \tfrac{1}{p_2} < \tfrac{i}{4}\label{Global:lower:Term2.1:p_2_upper}}{\theta_4<1 \iff \tfrac{1}{p_2} > 0.\label{Global:p_2_lower:Regime_2}}
                        There is clearly room for there to exist such a \(p_2\).
                  \item \underline{Term 2.2:} For the \(i=2\) case we interpolate to get,
                        \begin{align*}
                            \left\|v\right\|_{W^{1-\beta/2,q_2}} \lesssim \left\|v\right\|^{\theta_5^*}_{H^s}\left\|v\right\|_{L^{4}}^{1-\theta_5^*}
                        \end{align*}
                        where,
                        \begin{align*}
                            \theta_5^* = \frac{1-\frac{\beta}{2} + \frac{1}{4} - \frac{1}{q_2}}{s-\frac{1}{2} + \frac{1}{4}}.
                        \end{align*}
                        Requiring validity,
                        \doubleeqflex{0.45}{0.45}{ \theta_5^* >0 \iff \tfrac{1}{q_2} < 1-\tfrac{\beta}{2} + \tfrac{1}{4},\notag}{ \theta_5^* < 1 \iff \tfrac{1}{q_2}>\tfrac{3}{2} - \tfrac{\beta}{2}-s.\notag}
                        Checking that such a \(q_2\) can exist,
                        \begin{align*}
                            \tfrac{3}{2} -\tfrac{\beta}{2} -s < 1-\tfrac{\beta}{2} + \tfrac{1}{4} \iff s> \tfrac{1}{4}
                        \end{align*}
                        which is satisfied by assumption. If \(i=1\) we distinguish further via another application of the fractional Leibniz rule,
                        \begin{align*}
                            \left\|v^{2}\right\|_{W^{1-\beta/2, q_2}} & \lesssim \left\|v\right\|_{W^{1-\beta/2,p_4}}\left\|v\right\|_{L^{q_4}} \\
                                                                            & = (\text{Term 2.2.1})\cdot (\text{Term 2.2.2}).\end{align*}
where we have the H\"older condition,
\begin{align}
    \tfrac{1}{q_2} = \tfrac{1}{p_4} + \tfrac{1}{q_4}\label{Regime_2:q_2:holder_con}
\end{align}
                        \begin{itemize}
                            \item \underline{Term 2.2.1:}  Interpolating much as we did for Term 2.2 in the \(i=2\) case,
                                  \begin{align*}
                                      \left\|v\right\|_{W^{1-\beta/2, p_4}}\lesssim \left\|v\right\|_{H^s}^{\theta_5} \left\|v\right\|_{L^{k+1}}^{1-\theta_5}
                                  \end{align*}
                                  where,
                                  \begin{align*}
                                      \theta_5 = \frac{1-\frac{\beta}{2} + \frac{1}{4} - \frac{1}{p_4}}{s-\frac{1}{2} + \frac{1}{4}}
                                  \end{align*}
                                  where validity follows identically from the \(i=2\) case from Term 2.2. Note that we  have the condition,
                                  \begin{align}
                                      \tfrac{1}{p_4}>\tfrac{3}{2} - \tfrac{\beta}{2}-s.\label{Global:low:term2.2.1:p_4_lower}
                                  \end{align}
                            \item \underline{Term 2.2.2:} Again interpolating we yield,
                                  \begin{align*}
                                      \left\|v\right\|_{L^{q_4}} \lesssim \left\|v\right\|_{H^s}^{\theta_6}\left\|v\right\|_{L^{4}}^{(1-\theta_6)}
                                  \end{align*}
                                  where,
                                  \begin{align*}
                                      \theta_6 = \frac{\frac{1}{4}-\frac{1}{q_4}}{s-\frac{1}{2} + \frac{1}{4}}.
                                  \end{align*}
                                  Checking for validity,
                                  \doubleeqflex{0.3}{0.6}{\theta_6>0 \iff \tfrac{1}{q_4} < \tfrac{1}{4},\notag}{\theta_6 <1\iff \tfrac{1}{q_4} > \tfrac{1}{2}-s\label{Global:low:q_4_lower}}
                                  Checking that there is space for such a \(q_4\) to exist,
                                  \begin{align*}
                                      \tfrac{1}{2}-s<\tfrac{1}{4}\iff s> \tfrac{1}{4}.
                                  \end{align*}
                                  which is satisfied by assumption.
                        \end{itemize}
                        Using \eqref{Global:low:term2.2.1:p_4_lower} and \eqref{Global:low:q_4_lower} in our H\"older condition \eqref{Regime_2:q_2:holder_con} we get that,
                        \begin{align*}
                            \tfrac{1}{q_2} > \tfrac{3}{2}-\tfrac{\beta}{2}-s + (\tfrac{1}{2} -s) \vee 0.
                        \end{align*}
                        Now using this and \eqref{Global:p_2_lower:Regime_2} in our H\"older condition \eqref{Global:low:first_hold_cond} we obtain,
                        \begin{align*}
                            \tfrac{1}{2} > \tfrac{3}{2}-\tfrac{\beta}{2}-s + (2-i)(\tfrac{1}{2} -s)\vee 0 \iff s > \tfrac{3-\beta}{4}
                        \end{align*}
                        which follows from the assumption that our local result is satisfied.
              \end{itemize}
              Now we can write our expanded estimate for Term 2, for $i=2$,
              \begin{align}
                  \left(\text{Term 2}\right)\cdot \left\|w\right\|_{H^{\beta/2}} & \lesssim \left\|w\right\|_{H^{\beta/2}}^{2\theta_4+1}\left\|w\right\|^{2(1-\theta_4)}_{L^{4}}\left\|v\right\|_{H^s}^{\theta_5^*}\left\|v\right\|_{L^{4}}^{1-\theta_5^*}\notag \\
                                                                                     & \lesssim  T^{-\frac{\theta_5^*}{2}}\mathscr{H}_{3,1,1}(u_0)^{2(1-\theta_4)}\mathscr{H}_{3,1,1}(v_0)^{1-\theta_5^*}\left\|w\right\|_{H^{\beta/2}}^{(2)\theta_4 +1}.\label{Global:lower:expansion_term2_one}
              \end{align}
              For \(i=1\),
              \begin{align}
                  (\text{Term 2})\cdot \left\|w\right\|_{H^{\beta/2}} & \lesssim \left\|w\right\|_{H^{\beta/2}}^{\theta_4 + 1}\left\|w\right\|_{L^{4}}^{1-\theta_4}\left\|v\right\|_{H^s}^{\theta_5 + \theta_6}\left\|v\right\|_{L^{4}}^{\theta_5 + \theta_6}\notag \\
                                                                          & \lesssim T^{-\frac{\theta_5 + \theta_6}{2}}\mathscr{H}_{3,1,1}(u_0)^{1-\theta_4} \mathscr{H}_{3,1,1}(v_0)^{\theta_5 + \theta_6}\left\|w\right\|^{\theta_4 +1}_{H^{\beta/2}}\label{Global:lower:expansion_term2_i}
              \end{align}
              where again we have used Remark \eqref{Global:general:w_conservation}, Lemma \ref{Global:critical:general:Lk1_control} and \eqref{Global:frequency_cut:v_T_bound}.

    \end{itemize}
    Returning to \eqref{Global:First_split} with \eqref{Global:lower:expansion_term1_one}, \eqref{Global:lower:expansion_term1_i}, \eqref{Global:lower:expansion_term2_one} and \eqref{Global:lower:expansion_term2_i} we get that,
    \begin{align*}
        &\int_\T \left|\left(|\pat_x|^{1-\beta/2} w^i v^{3-i}\right)\mathcal{H} |\pat_x|^{\beta/2} w\right|\, dx
         \lesssim \begin{cases}
          C^{\text{Term 1}}_{1}\left\|w\right\|_{H^{\beta/2}}^{\theta_1^* +1} + C^{\text{Term 2}}_{1}\left\|w\right\|_{H^{\beta/2}}^{\theta_4 +1}                & i=1  \\
          C^{\text{Term 1}}_{i}\left\|w\right\|_{H^{\beta/2}}^{\theta_1 + (i-1)\theta_2 + 1} + C^{\text{Term 1}}_{i}\left\|w\right\|_{H^{\beta/2}}^{i\theta_4+1} & i=2
                                                               \end{cases}
    \end{align*}
    where,
    \begin{align*}
        &C_i^{\text{Term 1}} \coloneqq C_i^{\text{Term 1}}(s,\beta,u_0,v_0)\\
          &\hphantom{X}= \begin{cases}
            T^{-\frac{1}{2}}\mathscr{H}_{3,1,1}(u_0)^{(1-\theta_1^*)}\mathscr{H}_{3,1,1}(v_0)^{2(1-\theta_3)}, & i=1\\
            T^{-\theta_3}\mathscr{H}_{3,1,1}(u_0)^{1-\theta_1 + (i-1)(1-\theta_2)}\mathscr{H}_{3,1,1}(v_0)^{(3-i)(1-\theta_3)} & i=2
        \end{cases}\\\\
        &C_i^{\text{Term 2}} \coloneqq C_i^{\text{Term 2}}(s,\beta,u_0,v_0)\\
        &\hphantom{X}=\begin{cases}
            T^{-\frac{\theta_5 + \theta_6}{2}}\mathscr{H}_{3,1,1}(u_0)^{i(1-\theta_4)}\mathscr{H}_{3,1,1}(v_0)^{\theta_5 + \theta_6}, & i=1\\
            T^{-\frac{\theta_5^*}{2}}\mathscr{H}_{3,1,1}(u_0)^{2(1-\theta_4)}\mathscr{H}_{3,1,1}(v_0)^{1-\theta_5^*}, & i=2
        \end{cases}
    \end{align*}
    Using this in \eqref{Global:lower:master_lesssim} we obtain,
      \begin{align}
        \frac{1}{2}\frac{d}{dt}\left\|w\right\|_{H^{\beta/2}}^2 &\lesssim C_1^{\text{Term 1}}\left\|w\right\|_{H^{\beta/2}}^{\theta_1^* +1} + C_1^{\text{Term 2}}\left\|w\right\|_{H^{\beta/2}}^{\theta_4 +1}\notag\\
         &\hphantom{X}+ C_{2}^{\text{Term 1}}\left\|w\right\|_{H^{\beta/2}}^{\theta_1+ \theta_2+1} + C_2^{\text{Term 2}}\left\|w\right\|_{H^{\beta/2}}^{2\theta_4 + 1}\label{Global:freq_cut:cubic_pregron}
    \end{align}
     In order for this inequality to be amenable to integration, we require from the first term that \(\theta_1^* +1 < 2\) which we verified by checking \(\theta_1^* < 1\). Next, we require that \(i\theta_4 +1 <2\). That is, we need
      \begin{align*}
        \frac{\frac{i}{4} - \frac{1}{p_2}}{\frac{\beta}{2} - \frac{1}{2} + \frac{1}{4}} < 1 \iff \frac{1}{p_2} > \frac{1+i}{4} - \frac{\beta}{2}.
    \end{align*}
    To ensure we have enough space for \(p_2\) to exist we must return to \eqref{Global:lower:Term2.1:p_2_upper}. We require that,
    \begin{align*}
         \tfrac{1+i}{4} - \tfrac{\beta}{2} < \tfrac{i}{4}
        \iff \beta > \tfrac{1}{2}
    \end{align*}
    which is satisfied by our \(\beta > 1\) assumption.
    Lastly, we need that \(\theta_1 + \theta_2+1<2\),
    \begin{align}
        \theta_1 + \theta_2 + 1\leq 2 & \iff \frac{1-\frac{\beta}{2} + \frac{1}{4} - \frac{1}{p_3}}{\frac{\beta}{2} - \frac{1}{2} + \frac{1}{4}} + \frac{\frac{1}{4} - \frac{1}{q_3}}{\frac{\beta}{2} - \frac{1}{2} + \frac{1}{4}} < 1 \notag  \\
                                           & \iff \frac{1-\frac{\beta}{2} + \frac{1}{2} - \frac{1}{p_1}}{\frac{\beta}{2} -\frac{1}{2} + \frac{1}{4}} < 1        \notag \\
                                           & \iff \tfrac{1}{p_1} > \tfrac{7}{4} - \beta \label{Global:low:p_1_gronwall_condnow}
    \end{align}
    where in the second line we used \eqref{Global:low:term1.1_holder}. Given that $5/4<\beta < 3/2$ it is clear that $p_1$ is well defined. Furthermore, we need to check that there is space for \(p_1\) to exist within our existing upper bounds. Note that of \eqref{Global:lower:p_1upper} and \eqref{Global:low:special_term_1.1_p_1_upper}, \eqref{Global:low:special_term_1.1_p_1_upper} is the stricter condition. Therefore, we require that,
    \begin{align*}
        \tfrac{7}{4} - \beta < \tfrac{5}{4} - \tfrac{\beta}{2} \iff \beta>1 
    \end{align*}
    which is satisfied by assumption.  We must now re-check our H\"older condition \eqref{Global:low:first_hold_cond} with our new condition \eqref{Global:low:p_1_gronwall_condnow} and \eqref{Global:low:term1.2:q_1_lower_bound} (notice that this condition is strictly stronger then our initial \(p_1\) condition \eqref{Global:low:special_term_1.1_p_1_lower}),
    \begin{align*}
        \tfrac{1}{2} > \tfrac{7}{4} - \beta + 2\left(\tfrac{1}{2} -s\right)\vee 0.
    \end{align*}
    For \(s\geq 1/2\) we verify that,
    \begin{align*}
        \tfrac{1}{2} > \tfrac{7}{4} - \beta \iff \beta > \tfrac{5}{4}
    \end{align*}
    which is satisfied in our range of $\beta$. For $s < 1/2$,
    \begin{align*}
        \tfrac{1}{2} > \tfrac{7}{4} - \beta + 2\left(\tfrac{1}{2} - s\right) \iff s> \tfrac{9}{8} - \tfrac{\beta}{2}
    \end{align*}
    wich is satisfied by assumption. Setting,
    \begin{align*}
        \alpha &\coloneqq \max\left(\theta_1^* + 1, \theta_4 + 1, \theta_1+\theta_2 + 1,2\theta_4+1\right)
    \end{align*}
    and returning to \eqref{Global:freq_cut:cubic_pregron} we obtain,
    \begin{align*}
        \frac{1}{2}\frac{d}{dt}\left\|w\right\|_{H^{\beta/2}}^2 &\leq  C(s,\beta,u_0,v_0)\max_{i\in \{1,2\}}\left(C_i^{\text{Term 1}}, C_i^{\text{Term 2}}\right)\left\|w\right\|_{H^{\beta/2}}^\alpha.
    \end{align*}
    for some $C(s,\beta,u_0,v_0)>0$. Given we have conditioned that each term in the maximum be strictly less than 2, we know that $\alpha<2$. Therefore, integrating this inquality yields the following growth on $w(t)$,
    \begin{align}
        \left\|w(t)\right\|_{H^{\beta/2}} \leq \left(\left\|w_0\right\|_{H^{\beta/2}}^{1-\alpha} + C(s,\beta,u_0,v_0)\max_{i\in \{1,2\}}\left(C_i^{\text{Term 1}}, C_i^{\text{Term 2}}\right)(1-\alpha)t \right)^{\frac{2}{2-\alpha}}.\label{Global:frequency_cut:w_growth_poly}
    \end{align}
    Using this, can iterate in steps $T_{w \, \text{existence}}$ to yield a solution $w(t)$ on $[0,T]$. Combining this with our solution $v(t)$, we yield our desired solution $u(t) = v(t) + w(t)$ to \eqref{Gen-BOBBM-eq} on $[0,T]$. Furthermore, using \eqref{Global:frequency_cut:w_growth_poly} and \eqref{Global:frequency_cut:v_T_bound} we obtain the growth bound for $t\in [0,T]$,
    \begin{equation}
          \left\|u(t)\right\|_{H^s} \leq 2T^{-\frac{1}{2}}\vee \left(\left\|w_0\right\|_{H^{\beta/2}}^{1-\alpha} + C(k,s,\beta,u_0,v_0)\max_{i\in \{1,2\}}\left(C_i^{\text{Term 1}}, C_i^{\text{Term 2}}\right)(1-\alpha)t \right)^{\frac{2}{2-\alpha}}\label{Reg_2:growth_bound}.
    \end{equation}
\end{proof}
\begin{remark}\label{Regime_2:restriction_rem2}
    Within the stated regime, this argument fails for $k>3$. If we extend our $\beta$ to $\beta \leq 2$ however, the numerology allows for $k<7$. This mirrors the $s>1/2$ global well-posedness for the quintic BBM obtained in \cite{kim2026} using the same method. Regime 1 improves this result to $s\geq 1/3$ so there is no need to use the above regime for the higher $\beta$ values.
\end{remark}

\subsection{Regime 3}
For this regime, we return to general odd $k$ and the energy method.

\begin{proposition}\label{Regime_3:proposition_result}
     Let $\mu \in \{-1,1\}$, $d=1$, $k\geq 3$ be an odd integer and $\frac{2k-2}{k+1} \leq \beta \leq 2$. Then \eqref{Gen-BOBBM-eq} is globally well-posed in $H^s$ for $s>1/2$.
\end{proposition}
\begin{proof}
    For $u_0\in H^s$, we again define the corresponding mollified data $u_0^\varepsilon = u_0*\eta_\varepsilon$ and its local evolution $u^\varepsilon$. Using Lemma \ref{global:critical:general:lem_energy} we get that,
    \begin{align*}
        \frac{1}{2}\frac{d}{dt}\left\|u^\varepsilon\right\|^2_{H^s_\beta }&= -\frac{1}{k}\left\langle \langle \pat_x \rangle_\beta^{2s -\beta} u^\varepsilon, \pat_x\left((u^\varepsilon)^k\right) \right\rangle_{L^2}\\
        &=  -\frac{1}{k}\left\langle \langle \pat_x \rangle_\beta^{s} u^\varepsilon, \langle \pat_x\rangle^{s-\beta}_\beta \pat_x\left((u^\varepsilon)^k\right) \right\rangle_{L^2}.
    \end{align*}
    Estimating and then using Lemma \ref{Lem:multiplier_lemma} to change into our standard $H^s$ multipliers we obtain
    \begin{align}
        \frac{1}{2}\frac{d}{dt} \left\|u^\varepsilon\right\|^2_{H_\beta ^s} &\lesssim \left\|u^\varepsilon\right\|_{H^s_\beta}\left\|\langle \pat_x\rangle^{s-\beta}_\beta \pat_x \left((u^\varepsilon)^k\right)\right\|_{L^2}\notag\\
        &\lesssim \left\|u^\varepsilon\right\|_{H^s_\beta} \left\| \langle \pat_x \rangle^{s-\beta+1} (u^\varepsilon)^k\right\|_{L^2} = \left\|u^\varepsilon\right\|_{H^s_\beta} \left\|(u^\varepsilon)^k\right\|_{H^{s-\beta+1}}.\label{Regime3:initial_bound}
    \end{align}
     We now inspect the norm of the nonlinear term. If $\beta>1$, we apply the Sobolev inequality \footnote{Here, by the Sobolev inequality we mean that the embedding $F^s_{\frac{2}{2\beta-1},2}(\T)\hookrightarrow F^{s-\beta+1}_{2,2}(\T) = H^{s-\beta+1}(\T)$ holds. See \cite[Theorem 3.5.5 (ii)]{schmeisser1987} for this result.} and then our fractional Leibniz rule for the $F^s_{p,q}(\T)$ spaces, Proposition \ref{Frac_F_rule} to obtain
    \begin{align}
        \left\|(u^\varepsilon)^k\right\|_{H^{s-\beta +1}} \lesssim \left\|(u^\varepsilon)^k\right\|_{F^{s}_{\frac{2}{2\beta-1},2}(\T)} &\lesssim \left\|u^\varepsilon\right\|_{F^s_{2,2}}\left\|(u^\varepsilon)^{k-1}\right\|_{h^{\frac{1}{\beta-1}}}\notag\\
        &\sim  \left\|u^\varepsilon\right\|_{H^s}\left\|(u^\varepsilon)^{k-1}\right\|_{L^{\frac{1}{\beta-1}}}\label{Regime3:2nd_split}.
    \end{align}
    In second inequality we note that we have required that $s > \tfrac{2\beta-1}{2} -1$, but
    \begin{align*}
        \tfrac{2\beta-1}{2} -1 < \tfrac{1}{2} \iff \beta <2.
    \end{align*}
    Therefore, $s>\tfrac{2\beta-1}{2}-1$ is satisfied by our $s>\tfrac{1}{2}$ assumption. In the last line, \eqref{Regime3:2nd_split} we have used the fact that  $\frac{1}{\beta-1} >1$ given $\beta <2$, meaning the Hardy space $h^{\frac{1}{\beta-1}}(\T)$ (defined in Definition \ref{Def:F_space}) is equivalent in norm to $L^{\frac{1}{\beta-1}}(\T)$.

    If $\beta=1$, we simply apply the standard fractional Leibniz rule like so,
    \begin{align}
        \left\|(u^\varepsilon)^k\right\|_{H^s} \lesssim \left\|(u^\varepsilon)\right\|_{H^s}\left\|(u)^\varepsilon\right\|_{L^\infty}^{k-1}.\label{Regime_3:beta1_case}
    \end{align}
    
    We now distinguish two cases:

    \underline{Case 1:} If $\beta>1$ and $\frac{k-1}{\beta-1}\leq k+1$, then by the finite measure of $\T$, we have that
    \begin{align*}
        \left\|(u^\varepsilon)^k\right\|_{H^{s-\beta+1}} \lesssim \left\|u^{\varepsilon}\right\|_{H^s}\left\|u^{\varepsilon}\right\|_{L^{k+1}}^{k-1}.
    \end{align*}
    Using Lemma \ref{Global:critical:general:Lk1_control} on the $L^{k+1}$ norm and returning to \eqref{Regime3:initial_bound} we obtain
    \begin{align*}
        \frac{1}{2}\frac{d}{dt}\left\|u^\varepsilon\right\|_{H^s_\beta}^2 \lesssim \mathscr{H}_{k,1,1}(u_0^\varepsilon)^{k-1}\left\|u^\varepsilon\right\|_{H^s}^2.
    \end{align*}
    using Lemma \ref{Lem:multiplier_lemma} to exchange the $H^s$ norm for a $H^s_\beta$ norm, using Gr\"onwall's inequality and finally taking $\varepsilon \to 0$ with \eqref{global:Hamiltonian_limit} and \eqref{Global:solution_limit_Hs} we get that,
    \begin{align}
        \left\|u(t)\right\|_{H^s}\leq \left\|u_0\right\|_{H^s}\exp\left(C(k,s,\beta,u_0)\mathscr{H}_{k,1,1}(u_0^\varepsilon)^{k-1}t\right)\label{Regime_3:Plain_exp}
    \end{align}
    where $C(k,s,\beta,u_0)>0$.

    \underline{Case 2:} Now assume that $\frac{k-1}{\beta-1}>k+1$. If $\beta>1$ we interpolate between $L^{k+1}$ and $L^{\infty}$.
    \begin{align*}
        \left\|u^\varepsilon\right\|_{L^{\frac{k-1}{\beta-1}}}\lesssim \left\|u^\varepsilon\right\|_{L^{k+1}}^{\frac{k+1}{k-1}(\beta-1)}\left\|u^\varepsilon\right\|_{L^\infty}^{1-\frac{k+1}{k-1}(\beta-1)}.
    \end{align*}
    Returing this to \eqref{Regime3:initial_bound} with \eqref{Regime3:2nd_split} as well as using Lemma \ref{Global:critical:general:Lk1_control} to deal with the $L^{k+1}$ norm we yield and Lemma \ref{Lem:multiplier_lemma} to exchange the $H^s$ norm for a $H^s_\beta$ norm
    \begin{align*}
        \frac{d}{dt}\left\|u^\varepsilon\right\|_{H^s_\beta}^2 \lesssim \mathscr{H}_{k,1,1}(u_0^\varepsilon)^{(k+1)(\beta-1)}\left\|u^\varepsilon\right\|_{H^s_\beta}^2\left\|u^\varepsilon\right\|_{L^\infty}^{k-1-(k+1)(\beta-1)}.
    \end{align*}
    The above inequality also holds for $\beta=1$ as we simply return \eqref{Regime_3:beta1_case} to \eqref{Regime3:initial_bound} and apply Lemma \ref{Lem:multiplier_lemma}. Setting $\alpha = k-1-(k+1)(\beta-1)$ and $\gamma = (k+1)(\beta)$ we note that $\alpha \leq 2$ by our assumption that $\beta \geq \frac{2k-2}{k+1}$. Now, we invoke the Brezis Gallou\"et inequality, (Lemma \ref{lem:Brezis-Gallouet}) to obtain
    \begin{align*}
        \frac{1}{2}\frac{d}{dt}\left\|u^\varepsilon\right\|^2_{H^s_\beta}&\lesssim \mathscr{H}_{k,1,1}(u^\varepsilon_0)^{\gamma}\left\|u^\varepsilon\right\|_{H^s_\beta}^2\left(1+ \left\|u^\varepsilon\right\|_{H^{1/2}} \sqrt{\log\left(\left\|u^\varepsilon\right\|_{H^s}+1\right)}\right)^{\alpha}\\
        &\lesssim \mathscr{H}_{k,1,1}(u^\varepsilon_0)^{\gamma}\left\|u^\varepsilon\right\|_{H^s_\beta}^2\left(1+ \left\|u^\varepsilon\right\|_{H^{\beta/2}} \sqrt{\log\left(\left\|u^\varepsilon\right\|_{H^s}+1\right)}\right)^{\alpha}.
    \end{align*}
    Noting that $\left\|u\right\|_{H^{\beta/2}}$ is conserved by Lemma \ref{Global:lemma:Conservervation}, we get that,
    \begin{align*}
        \frac{1}{2}\frac{d}{dt}\left\|u^\varepsilon\right\|^2_{H^s_\beta} \lesssim \mathscr{H}_{k,1,1}(u^\varepsilon_0)^{\gamma}\left\|u^\varepsilon\right\|_{H^s_\beta}^2\left(1+ \mathscr{E}_{\beta,1}(u^\varepsilon_0) \sqrt{\log\left(\left\|u^\varepsilon\right\|_{H^s}+1\right)}\right)^{\alpha}.
    \end{align*}
    Now simplifying we obtain,
    \begin{align*}
        \frac{1}{2}\frac{d}{dt}\left\|u^\varepsilon\right\|_{H^s_\beta} \lesssim \mathscr{H}_{k,1,1}(u^\varepsilon_0)^{\gamma}\left\|u^\varepsilon\right\|_{H^s_\beta}\left(1+ \mathscr{E}_{\beta,1}(u^\varepsilon_0) \sqrt{\log\left(\left\|u^\varepsilon\right\|_{H^s}+1\right)}\right)^{\alpha}.
    \end{align*}
    Letting $y(t) = \left\|u^\varepsilon(t)\right\|_{H^s_\beta}$ and change variables to $y(t)= e^{z^\varepsilon(t)}$ to get,
    \begin{align*}
        \frac{d}{dt}z^\varepsilon(t) \leq C(k,s,\beta) \mathscr{H}_{k,1,1}(u^\varepsilon_0)^{\gamma}\mathscr{E}_{\beta,1}(u^\varepsilon_0)^{\alpha}\left(z^\varepsilon(t) + 1\right)^{\frac{\alpha}{2}}
    \end{align*}
    for some $C(k,s,\beta)>0$. In order to integrate this inequality we require that $\alpha/2 \leq 1$ which is satisfied by the fact that $\alpha \leq 2$. Moreover, if  $\alpha < 2$  we obtain a polynomial growth bound on $z^\varepsilon(t)$,
    \begin{align*}
        z^\varepsilon(t) \leq \left((z(0)+1)^{1-\frac{\alpha}{2}} + C(k,s,\beta,u_0) \mathscr{H}_{k,1,1}(u_0)^{\gamma}\mathscr{E}_{\beta,1}(u_0)^{\alpha} (1-\alpha)t\right)^{\frac{1}{1-\alpha/2}}.
    \end{align*}
    Recalling the definition of $z^\varepsilon(t)$, applying Lemma \ref{Lem:multiplier_lemma} to switch to a $H^s$ norm as well as using \eqref{global:energy_limit}, \eqref{global:Hamiltonian_limit} and \eqref{Global:solution_limit_Hs} to take $\varepsilon\to 0$. we yield the following growth bound
    \begin{align}
        \left\|u(t)\right\|_{H^s}&\leq \exp\left(\left(\left(\log \left\|u_0\right\|_{H^s_\beta}+1\right)^{1-\frac{\alpha}{2}}+ C(k,s,\beta,u_0) \mathscr{H}_{k,1,1}(u_0^\varepsilon)^{\gamma}\mathscr{E}_{\beta,1}(u_0^\varepsilon)^{\alpha} (1-\tfrac{\alpha}{2})t\right)^{\frac{1}{1-\alpha/2}}\right).\label{Regime_3_growth_bound}
    \end{align}
    If $\alpha=2$, we obtain an exponential growth bound on $z^\varepsilon(t)$,
    \begin{align*}
        z^\varepsilon(t) \leq (z(0)+1)\exp\left(C(k,s,\beta) \mathscr{H}_{k,1,1}(u_0^\varepsilon)^{\gamma}\mathscr{E}_{\beta,1}(u_0^\varepsilon)^2t\right).
    \end{align*}
    Again recalling the definition of $z^\varepsilon(t)$, using Lemma \ref{Lem:multiplier_lemma}, \eqref{global:energy_limit}, \eqref{global:Hamiltonian_limit} and \eqref{global:energy_limit} we get that
    \begin{align}
         \left\|u(t)\right\|_{H^s} \leq \exp\left(\left(\log \left\|u_0\right\|_{H^s_\beta} +1\right)\exp\left(C(k,s,\beta) \mathscr{H}_{k,1,1}(u_0^\varepsilon)^{\gamma}\mathscr{E}_{\beta,1}(u_0^\varepsilon)^2t\right)\right).\label{Regime_3_growth_bound2}
    \end{align}
\end{proof}
\begin{remark}\label{Regime_3:tech_remark}
    Simpler techniques for proving this regime were ineffective. One possible avenue is introducing a case distinction on the sign of $s-\beta+1$ at \eqref{Regime3:2nd_split} instead of using the Sobolev inequality into $F^s_{\frac{2}{2\beta-1}}(\T)$. This distinction would involve using the fractional Leibniz rule when $s-\beta+1$ is positive and duality when it is not. This introduces several restrictions including $k\leq 7$, making the approach not viable. As such, we instead embed directly into a space of regularity $s$, introducing the quasi-Banach complications.
\end{remark}
We now tie together our results from each of the regimes.
\begin{proof}[Proof of Theorem \ref{Global:thm}]
    Combine Propositions \ref{Regime_1_THM}, \ref{Regime_2:prop_result} and \ref{Regime_3:proposition_result}.
\end{proof}
\subsection{Dimension 2}
We now discuss the case when $d=2$. We specialise further to the two dimensional cubic BBM equation ($d=2$, $k=3$, $\beta=2$ and $\mu = \{-1,1\}$),
\begin{align*}
    \begin{cases}
        \pat_t u + |D_x|^2 \pat_t u + \left(\pat_x + \pat_y\right)\left(u + \mu \tfrac{1}{3} u^3\right) =0\\
        u(x,0) = u_0(x)  
    \end{cases} x\in \T^2
\end{align*}
where $x$ and $y$ are coordinates on $\T^2$. Our local result, Theorem \ref{Local_theory:lwp_main} tells us that \eqref{dim_2:cubic_BBM} is locally well-posed in $H^s(\T^2)$ for $s\geq 1/2$. Using this we now prove the global result for $s\geq 1/2$.
\begin{proof}[Proof of Theorem \ref{THM:dim_2}]
    Let $u_0\in H^s(\T^2)$ and define the it's mollification $u_0^\varepsilon = u_0 *\eta_\varepsilon$ and it's corresponding local evolution under \eqref{Gen-BOBBM-eq}, $u_\varepsilon$. Using Lemma \ref{global:critical:general:lem_energy} and the self adjointness of $\langle D_x\rangle$ we get that,
    \begin{align*}
        \frac{1}{2}\frac{d}{dt}\left\|u^\varepsilon\right\|_{H^s}^2 &= -\frac{1}{3}\left\langle \langle D_{x}\rangle^{2s-2} u^\varepsilon, \left(\pat_x + \pat_y\right)\left((u^\varepsilon)^3\right)\right\rangle_{L^2}\\
         &=-\frac{1}{3}\left\langle \langle D_{x}\rangle^{2s-1} u^\varepsilon, \langle D_{x}\rangle^{-1}\left(\pat_x + \pat_y\right) \left((u^\varepsilon)^3\right)\right\rangle_{L^2}.
    \end{align*}
    Now, given that $H^s(\T^2)\hookrightarrow W^{2s-1,2/s}(\T^2)$ we use H\"older's inequality to obtain,
    \begin{align*}
        \frac{1}{2}\frac{d}{dt}\left\|u^\varepsilon\right\|_{H^s}^2 &\lesssim \left\|\langle D_{x}\rangle^{2s-1}u^\varepsilon\right\|_{{L^{2/s}}}\big\|\langle D_{x}\rangle^{-1}\big(\pat_x + \pat_y\big) \left((u^\varepsilon)^3\right)\big\|_{L^{\frac{2}{2-s}}}.
    \end{align*}
    Using our embedding as well as the $L^p$ boundedness of the operator $\langle D_x\rangle^{-1}(\pat_x + \pat_y)$ we get that,
    \begin{align}
        \frac{1}{2}\frac{d}{dt}\left\|u^\varepsilon\right\|^2_{H^s}\lesssim \left\|u^\varepsilon\right\|_{H^s}\big\|(u^\varepsilon)^3\big\|_{L^{\frac{2}{2-s}}}\lesssim \left\|u^\varepsilon\right\|_{H^s}\big\|u^\varepsilon\big\|^3_{L^{\frac{6}{2-s}}}.\label{Dimension_2:global:pre_interp}
    \end{align}
    Now, interpolating the $L^{\frac{6}{2-s}}(\T^2)$ norm we obtain,
    \begin{align*}
        \left\|u^\varepsilon\right\|_{L^{\frac{6}{2-s}}} \lesssim \left\|u^\varepsilon\right\|^{1/3}_{H^s}\left\|u^\varepsilon\right\|^{2/3}_{L^{4}}.
    \end{align*}
    Returning this to \eqref{Dimension_2:global:pre_interp} and using Lemma \ref{Global:critical:general:Lk1_control}  we yield,
    \begin{align*}
        \frac{1}{2}\frac{d}{dt}\left\|u^\varepsilon\right\|_{H^s}^2 \leq  C(s)\mathscr{H}_{3,1,2}(u^\varepsilon_0)^2\left\|u^\varepsilon\right\|_{H^s}^2
    \end{align*}
    for some $C(s)>0$. Now, after an application of Gr\"onwall's inequality and using \eqref{global:Hamiltonian_limit} and \eqref{Global:solution_limit_Hs} to take $\varepsilon\to 0$, we obtain the following growth bound,
    \begin{align*}
        \left\|u(t)\right\|_{H^s} \leq \left\|u_0\right\|_{H^s}\exp\left( C(s)\mathscr{H}_{3,1,2}(u_0)^2 t\right)
    \end{align*}
    thereby proving global well-posedness for $1/2\leq s <2$. The peristence of regularity of the solution map yields the global result for $s \geq 2$.
\end{proof}
\begin{remark}\label{Dim_two_restriction_remark}
    It is not clear to the author how to extend this result to a lower range of $\beta$. If we were to adapt our derivative splitting energy method from Regime 1 to the two dimensional setting, we would demand the following upper and lower bounds on our shared derivative, $\sigma$,
    \begin{align*}
        s-\beta + 2 - \tfrac{2}{k+1} <\sigma < s-1 + \tfrac{2}{k+1}.
    \end{align*}
    For there to be room for such a $\sigma$ to exist, we would require that,
    \begin{align*}
        \beta > 3- \tfrac{4}{k+1}.
    \end{align*}
    For any $k\geq 3$, this forces $\beta >2$, leaving us outside the bounds of the model. Using a high-low wavelength decomposition like the one done by Kim and Kwak in \cite{kim2026} for the cubic in dimension 1 was also ineffective, after completing the calculation for general $k$ one is left with the restriction that $k<3$, excluding the cubic equation.
\end{remark}

\subsection{Well-posedness in $L^{k+1}$}

For this subsection, we are interested in the values of $1\leq\beta\leq 2$ such that corresponding instances of \eqref{Gen-BOBBM-eq} do not require $H^s$ regularity assumptions to be well-posed. Instead, we make integrability assumptions, namely that $u_0\in L^{k+1}$.
\begin{lemma}\label{Intersection:Local:multilinear_lem}
    For $u_1,\dots,u_k \in L^{k+1}(\mathbb{T}^d)$, $d\geq 1$, $k\geq 2$ and, 
\begin{align}
    \beta \geq 1 + d\left(\tfrac{k-1}{k+1}\right)\label{Intersection:Lk1:multilinear:beta_assumption}
\end{align}
there is the following estimate, 
\begin{align*}
    \Big\|\phi_\beta(D_x)\prod_{i=1}^k u_i \Big\|_{L^{k+1}} \lesssim \prod_{i=1}^k\left\|u_i\right\|_{L^{k+1}}.
\end{align*}
\end{lemma}
\begin{proof}
    Again, for readability, we prove the case for $u_1=\dots=u_k$. As we did in the proof of Lemma \ref{lem:multilinear_estimate} we note that  $\phi_\beta(D_x)$ smooths $\beta-1$ derivatives
\begin{align}
    \left\|\phi_\beta(D_x)u^k\right\|_{L^{k+1}} \sim \left\|u^k\right\|_{W^{1-\beta,k+1}}.\label{Intersection:Lk1:multilinear:smoothing}
\end{align}
Now, by the Sobolev embedding theorem, we see that $L^{\frac{k+1}{k}}(\mathbb{T}^d) \hookrightarrow W^{1-\beta,k+1}(\mathbb{T}^d)$ as,
\begin{align*}
    \beta - 1 - \tfrac{dk}{k+1} \geq -\tfrac{d}{k+1} \iff \beta \geq 1+ d\left(\tfrac{k-1}{k+1}\right)
\end{align*}
which is satisfied by assumption \eqref{Intersection:Lk1:multilinear:beta_assumption}. Therefore, returning to \eqref{Intersection:Lk1:multilinear:smoothing},
\begin{align*}
    \left\|\phi_\beta(D_x)u^k\right\|_{L^{k+1}} \lesssim \left\|u^k\right\|_{L^{\frac{k+1}{k}}} = \left\|u\right\|^k_{L^{k+1}}.
\end{align*}
\end{proof}
Using this estimate, we obtain the following local result.
\begin{lemma}\label{Lk1_local_lemma}
    If $d\geq 1$, $k\geq 2$,  $u_0\in L^{k+1}(\T^d)$ and,
    \begin{align*}
         1+ d\left(\tfrac{k-1}{k+1}\right) \leq \beta \leq 2
    \end{align*}
    then \eqref{Gen-BOBBM-eq} is locally well-posed in $L^{k+1}(\T^d)$.
\end{lemma}
\begin{proof}
    As noted by Genovese \textit{et al.} in the work on the transport of Gaussian measures for the fractional BBM equation \cite{genovese2023}, one may Taylor expand $S(t) = e^{t\phi_\beta(D_x)}$ to obtain
    \begin{align*}
        \left\|S(t) u\right\|_{L^{k+1}} \lesssim e^t \left\|u\right\|_{L^{k+1}}.
    \end{align*}
    Using this and Lemma \ref{Intersection:Local:multilinear_lem}, one obtains the result via the standard fixed point procedure.
\end{proof}
Note that if $d=1$, Lemma \ref{Lk1_local_lemma} yields local well-posedness for all odd $k\geq 3$. If $d=2$, the only admissible case is $k=3$ and $\beta=2$. If $d=3$, the condition on $\beta$ fails entirely. With this in mind, can prove the global result stated in the introduction.
\begin{proof}[Proof of Theorem \ref{Intersection:Lk1:local:thm}]
    Using Lemma \ref{Lk1_local_lemma} and the uniform boundedness of the $L^{k+1}(\T^d)$ norm of the solution as provided by Lemma \ref{Global:critical:general:Lk1_control}, we obtain our global result by the standard iteration argument.
\end{proof}
Before we apply this result to the context of almost sure well-posedness we first recall the following probabilistic estimate.
\begin{lemma}\label{Appendix:lem:strichartz}
    Let $d\geq 1$, $p\geq 2$ and $u_0\in L^2(\T^d)$. Then, there exists $C>0$ such that
    \begin{align*}
        \left\|\left\|u_0^\omega \right\|_{L^p(\T^d)}\right\|_{L^p(\Omega)} \leq C \sqrt{p}|\T^d|^{1/p} \left\|u_0\right\|_{L^2(\T^d)}.
    \end{align*}
    Furthermore, this implies that there exists $\Sigma \subset \Omega$ such that $\Prb(\Sigma) = 1$ and that for all $\omega \in \Sigma$, $u_0^\omega \in L^p(\T^d)$.\footnote{This is simply the definition of the almost sure membership $u_0^\omega \in L^p(\T^d)$, but we explicitly include it for the proof of Corollary \ref{cor:almost_sure_global}}
\end{lemma}
\begin{proof}
    See \cite[Section 4.2.2]{gubinelli2019} for a proof.
\end{proof}
With this, we may prove the almost sure global result.
\begin{proof}[Proof of Corollary \ref{cor:almost_sure_global}]
    Fixing $u_0\in L^2(\T^d)$, we randomise according to \eqref{Intro:data_randomisation} to obtain $u_0^\omega$. Now, using Lemma \ref{Appendix:lem:strichartz}, we ensure that there exists a set $\Sigma \subset \Omega$ such that $\Prb(\Sigma) = 1$ and $u_0^\omega \in L^{k+1}(\T^d)$, for all $\omega \in \Sigma$. For these values of $\omega$, we may then apply Theorem \ref{Intersection:Lk1:local:thm} to obtain a unique, stable, global solution $u^\omega(t)$ to \eqref{Gen-BOBBM-eq} with initial data $u_0^\omega$. With this, we yield the almost sure global well-posedness of \eqref{Gen-BOBBM-eq} in $L^2(\Omega)$. By the persistence of regularity of the solution map, as afforded by our local well-posedness on the $H^s$ scale (Theorem \ref{Local_theory:lwp_main}), we extend the almost sure global result to $H^s$ for any $s\geq 0$.
\end{proof}

\section{Endpoint well-posedness for the cubic regularised BO equation}
In the context of our current local theory, Theorem \ref{Local_theory:lwp_main}, when $\beta=1$ we have local well posedness of \eqref{Gen-BOBBM-eq} in $H^s(\T)$ for $s>\frac{1}{2}$. Here, we will work to recover a local result at the $s=\frac{1}{2}$ endpoint using a compactness method. We will take inspiration from He and Mammeri's method in \cite{he2018} where they proved global well-posedness for the regularised BO equation in $H^{1/2}(\R)$. While there is significant overlap in the local well-posedness arguments for the quadratic and cubic nonlinearities, the argument for global well-posedness used for the quadratic fails for the cubic, as we will remark in the proof.

For convience, we will restate the cubic regularised BO equation, \eqref{Intro:cRBO} as an initial value problem
\begin{align}
    \begin{cases}
        \pat _t u + |D_x| \pat_t u + \pat_x u +\mu\tfrac{1}{3}\pat_x\left(u^3\right) = 0\\
        u(x,0) = u_0(x)
    \end{cases}x \in \T\label{cubic_rbo_PDE}.
\end{align}
This is simply \eqref{Gen-BOBBM-eq} with $d=1$, $\mu\in\{-1,1\}$, $k=3$ and $\beta=1$. The goal of this section is to prove the following result.
\begin{proposition}\label{THM:Local:endpoint}
    Let $\mu \in \{-1,1\}$. Then \eqref{cubic_rbo_PDE} is locally well-posed in $H^{1/2}(\T)$.
\end{proposition}
To prove Proposition \ref{THM:Local:endpoint}, we will begin by proving a technical Lemma and then proceed to proving the existence, uniqueness and continuity of the weak solution sequentially. Until we prove local uniqueness, we fix an arbitrary $T>0$.
\begin{lemma}\label{LEM:Local:endpoint:subsequence}
    Let $(u_{n,0})_{n\in \N}\subset H^1(\T)$ such that $u_{n,0}\to u_0$ in $H^{1/2}(\T)$ and denote~$u_n \in C([0,T];H^1(\T))$ to be the global solution of \eqref{cubic_rbo_PDE}. There exists a subsequence $u_{n_m}$ and a function $u \in C_{w}([0,T];H^{1/2}(\T))\cap C\left([0,T];L^2(\T)\right)$ such that,
    \begin{align}
        &u_{n_{m}} \longrightarrow u \text{\, weakly-$\star$ in \,} L^\infty\left([0,T];H^{1/2}(\T)\right)\label{Local:endpoint:main_weak_star}\\
         &u_{n_m}\longrightarrow u \text{\, strongly in \,} C([0,T];L^2(\T))\label{Local:endpoint:main_strong}.
    \end{align}
\end{lemma}
\begin{proof}
    We first note that \eqref{cubic_rbo_PDE} is globally well-posed in $H^1(\T)$ by Proposition \ref{Regime_3:proposition_result}. Now, Let $(u_{n,0})_{n\in \N} \subset H^1(\T)$ such that $u_{n,0}\to u_0$ in $H^{1/2}(\T)$. Moreover, for each $n\in \N$ denote $u_n(t) \in C([0,T] ; H^1(\R))$ to be the evolution of $u_{0,n}$ under \eqref{cubic_rbo_PDE}  as obtained in Proposition \ref{Regime_3:proposition_result}. By Lemma  \ref{Global:lemma:Conservervation} the $H^{1/2}(\T)$ norm of $u_n$ is conserved, i.e.
    \begin{align}
        \left\|u_n\right\|_{H^{1/2}(\T)} =  \mathscr{E}_{1,1}(u_{0,n})\label{Local:endpoint:conservation_claim}
    \end{align}
    for any $t\in [0,T]$. Therefore, $u_{n}$ is bounded in $C([0,T] ; H^{1/2}(\T))$. Seeking to obtain uniform control of $\pat_t u_n$ in $L^2$ we first multiply \eqref{cubic_rbo_PDE} by $\pat_t u_n$ and integrate yielding,
    \begin{align*}
        \int_\T |\pat_t u_n|^2\, dx = -\int_\T (\pat_t u_n) \pat_x \left(1+D_x\right)^{-1} \left(u_n + \mu\frac{u_n^3}{3}\right)\, dx.
    \end{align*}
    Using Young's product inequality, we see for all $\varepsilon >0$,
    \begin{align*}
        \left\|\pat_t u_n\right\|^2_{L^2} &\leq \varepsilon\int_\T \left|\pat_t u_n\right|^2 \, dx + \frac{1}{2\varepsilon} \left\|\phi_\beta(D_x) u_n\right\|^2_{L^2} + \frac{1}{2\varepsilon} \big\|\phi_\beta(D_x) \frac{u_n^3}{3}\big\|_{L^2}.
    \end{align*}
    Setting $\varepsilon=1/2$, we get that,
    \begin{align*}
        \left\|\pat_t u_n\right\|^2_{L^2} \lesssim \left\|\phi_\beta(D_x)u_n\right\|_{L^2}^2 + \left\|\phi_\beta(D_x) u_n^3\right\|_{L^2}^2\lesssim \left\|u_n\right\|^2_{L^2} + \left\|u_n^3\right\|^2_{L^2}
    \end{align*}
    where in the final line we have used the $L^2$ boundedness of $\phi_\beta(D_x)$. Now, using the Sobolev embedding we find a uniform bound,
    \begin{align*}
         \left\|\pat_t u_n\right\|^2_{L^2} = \left\|u_n\right\|^2_{L^2} + \left\|u_n\right\|^6_{L^6}\lesssim \left\|u_n\right\|^2_{H^{1/2}} + \left\|u_n\right\|^6_{H^{1/2}} \leq  \mathscr{E}_{1,1}(u_{0,n}) + \mathscr{E}_{1,1}(u_{0,n})^3.
    \end{align*}
    By our above estimates, we see that $u_n$ is bounded in $L^\infty([0,T];H^{1/2}(\T))$ and that $\pat_t u_n$ is bounded in $L^\infty([0,T];L^2(\T))$. Using the boundedness of $u_n$, Banach-Alaoglue guarantees that there exists a subsequence $u_{n_k}$ and a function $u \in L^{\infty}\left([0,T];H^{1/2}(\T)\right)$ such that,
    \begin{align}
        u_{n_{k}} \longrightarrow u \text{\, weakly-$\star$ in \,} L^\infty\left([0,T];H^{1/2}(\T)\right)\label{Local:endpoint:weak_star_condition}.
    \end{align}
    Furthermore, by Rellich-Kondrakov, we know that $H^{1/2}(\T)$ is compactly embedded in $L^2(\T)$. Therefore, we have the following chain of embeddings,
    \begin{align*}
        H^{1/2}(\T)  \xhookrightarrow{\text{compact}} L^2(\T) \xhookrightarrow{\text{continuous}} L^2(\T).
    \end{align*}
    Applying the Aubin-Lion Lemma (see \cite[Theorem 2.5.16]{boyer2013}), the space $W$ defined,
    \begin{align*}
        W = \left\{u \in L^{\infty}\left([0,T]; H^{1/2}(\T)\right): \pat_t u \in L^\infty\left([0,T] ;L^2(\T)\right)\right\}
    \end{align*} 
    is compactly embedded in $C\left([0,T];L^2(\T)\right)$. Clearly $\left(u_n\right)_{n\in \N}\subset W$ meaning that $u_n$ is relatively compact in $C([0,T];L^2(\T))$. Therefore, we can extract a subsequence $u_{n_j}$ and a limit $v\in C([0,T];L^2(\T))$ such that,
    \begin{align}
        u_{n_j}\longrightarrow v \text{\, strongly in \,} C([0,T];L^2(\T)).\label{Local:endpoint:strong_condition}
    \end{align}
    Now, to prove that $u\equiv v$ we take the subsequence $u_{n_k}$ satisfying \eqref{Local:endpoint:weak_star_condition} and index it via $j$ where $u_{n_j}$ satisfies \eqref{Local:endpoint:strong_condition}. The sequence $u_{n_{k_j}}$ will satisfy both of these properties and since strong convergence implies weak-$\star$ convergence, and $L^2(\T)$ is Hausdorff, the weak-$\star$ limit must be unique, meaning $u\equiv v$. Defining a new index $m$ such that $u_{n_m} = u_{n_{k_j}}$ we get that there exists a function $u\in L^{\infty}\left([0,T];H^{1/2}(\T)\right)\cap C([0,T];L^2(\T))$ such that, \eqref{Local:endpoint:main_weak_star} and \eqref{Local:endpoint:strong_condition} are satisfied. Since $H^{1/2}(\T)$ is seperable, reflexive and embedded continuously into $L^2(\T)$, it follows from standard regularity results that $u \in C_w\left([0,T];H^{1/2}(\T)\right)\cap C([0,T];L^2(\T))$ (see, \cite[Lemma 2.5.9]{boyer2013}).
\end{proof}
 Using this Lemma we can prove existence of a local solution.
\begin{lemma}\label{Local:endpoint:Lem:existence}
    For initial data $u_0\in H^{1/2}(\T)$, there exists a local weak solution $u$ to \eqref{cubic_rbo_PDE} such that $u\in C_w([0,T];H^{1/2}(\T))\cap C([0,T];L^2(\T))$.
\end{lemma}
\begin{proof}
    Let $u_{n_m}$ be the subsequence obtained in \ref{LEM:Local:endpoint:subsequence}. We have that for all $k\in \N$ and all $v\in C^\infty_c(\T)$,
    \begin{align*}
        \int_0^T \int_\T v \pat_t u_{n_m} + v\pat_t D_x u_{n_m} + v\pat_x \left(u_{n_m} + \mu\frac{u_{n_m}^3}{3}\right)\, dx\, dt=0.
    \end{align*}
    Integrating by parts we obtain,
    \begin{align}
        \int_0^T \int_{\T} u_{n_m} \pat_t v + u_{n_m}\pat_t D_x v + \left(u_{n_m}+\mu\frac{u_{n_m}^3}{3}\right)\pat_x v \, dx\, dt =0\label{Local:endpoint:integration_by_parts}
    \end{align}
    Now, by the weak-$\star$ convergence of $u_{n_m}$ we obtain a limit for our linear terms,
   \begin{align*}
    \lim_{m\to \infty}\int_0^T \int_{\T} u_{n_m} \pat_t v + u_{n_m}\pat_t D_x v + u_{n_m}\pat_x v \, dx\, dt = \int_0^T \int_{\T} u \pat_t v + u\pat_t D_x v + u\pat_x v \, dx\, dt.
   \end{align*}
   For the nonlinear term, we require the strong convergence due to the product giving us,
   \begin{align*}
    \lim_{m\to \infty}\int_0^T \int_{\T} \mu\frac{u_{n_m}^3}{3} \pat_x v\, dx = \int_0^T \int_{\T} \mu\frac{u^3}{3} \pat_x v\, dx.
   \end{align*}
   Returning these to \eqref{Local:endpoint:integration_by_parts} we get that,
   \begin{align*}
    \int_0^T \int_{\T} u \pat_t v + u\pat_t D_x v + \left(u+\mu\frac{u^3}{3}\right)\pat_x v \, dx\, dt =0
   \end{align*}
   for all $v\in C^\infty(\T)$, meaning that $u$ is a weak solution to \eqref{cubic_rbo_PDE}.
\end{proof}
 We now recall the definition of the operator $\phi_\beta(D_x) := \pat_x (1+D_x)^{-1}$. To show uniqueness, we first require a basic commutator estimate on this operator.
\begin{lemma}\label{Local:endpoint:commutator_estimate}
    for $1\leq p<\infty$ and $1/q + 1/p = 1/r$ with $p$ and $r$ valid Lebesgue exponents, there is the following commutator estimate,
    \begin{align*}
        \left\|\phi_\beta(D_x)(uv) - u\phi_\beta(D_x)v - v\phi_\beta(D_x)u\right\|_{L^p} \lesssim \left\|u\right\|_{L^q}\left\|v\right\|_{L^r}.
    \end{align*}
\end{lemma}
\begin{proof}
    See \cite[Lemma 1.3]{he2018}.
\end{proof}
     We may now proceed with the uniqueness proof. As mentioned, we follow closely the method of He and Mammeri in \cite{he2018}. Note now that $T$ is no longer arbitrary and will henceforth depend on $u_0$.
    \begin{lemma}\label{Lem:Local:endpoint:uniqueness}
        The weak solution obtained to \eqref{cubic_rbo_PDE} in Lemma \ref{Local:endpoint:Lem:existence} is unique in $C_w([0,T];H^{1/2}(\T))\cap C([0,T];L^2(\T))$ for fixed initial data $u_0\in H^{1/2}(\T)$ and some $T := T(u_0) >0$.
    \end{lemma}
    \begin{proof}
        Assume that there exists weak solutions $u,v \in C_w([0,T];H^{1/2}(\T))\cap C([0,T];L^2(\T))$ to \eqref{cubic_rbo_PDE} with the same initial data $u_0\in H^{1/2}(\T)$. We first note that we can obtain uniform in time bounds on the $H^{1/2}$ norms of $u$ and $v$. For $u$ this can be done via the weak approximating sequence of $H^1(\T)$ solutions $u_{n_m}$ and the weak lower semi-continuity of the norm,
        \begin{align}
            \sup_{t\in [0,T]}\left\|u(t)\right\|_{H^{1/2}} &\leq \liminf_{m\to \infty} \left\|u_{n_m}\right\|_{H^{1/2}} = \liminf_{m\to \infty} \left\|u_{n_m,0}\right\|_{H^{1/2}} = \left\|u_0\right\|_{H^{1/2}} = \mathscr{E}_{1,1}(u_0)\label{Local:endpoint:unif_bound_u}
        \end{align}
        where we used the conservation \eqref{Local:endpoint:conservation_claim}. Since we lack an approximating sequence for $v$, we instead use the uniform boundedness of weakly continuous functions on compact intervals to obtain,
        \begin{align}
            \sup_{t\in[0,T]}\left\|v(t)\right\|_{H^{1/2}}&\leq C_{v_0,T}\label{Local:endpoint:unif_bound_v}.
        \end{align}
        for some constant $C_{v_0,T}>0$. Now, define
        \begin{align*}
            g(t) \coloneqq \left\|u(t)-v(t)\right\|^2_{L^2}
        \end{align*}
        and $w\coloneqq u-v$. Taking the time derivative of $g$, aiming to use Gr\"onwall's inequality we see that,
        \begin{align}
            g'(t) = 2\int_\T (u-v)(\pat_t u - \pat_t v)\, dx &= -2\int_{\T}(u-v)\phi_\beta(D_x)\left(u+\mu\frac{u^3}{3} - v - \mu\frac{v^3}{3}\right)\, dx\notag\\
            &= -2\int_\T w\phi_\beta(D_x)w + \mu w\phi_\beta(D_x)\left(\frac{u^3}{3} - \frac{v^3}{3}\right)\, dx\notag\\
            &= -\mu\frac{2}{3} \int_\T w \phi_\beta(D_x)\left(w\left(u^2 + uv + v^2\right)\right)\, dx\label{Local:endpoint:uniqueness:ipb}
        \end{align}
        where in the last inequality we used integration by parts. Now, defining $A := A(u,v) = u^2 +uv + v^2$ and $R := R(u,v) = \phi_\beta(D_x)(wA) -A\phi_\beta(D_x)w - w\phi_\beta(D_x) A$ we can split \eqref{Local:endpoint:uniqueness:ipb} like so,
        \begin{align}
            g'(t) &= -\mu \frac{2}{3}\int_\T w^2\phi_\beta(D_x)A \, dx -\mu \frac{2}{3}\int_\T wA \phi_\beta(D_x) w\, dx -\mu \frac{2}{3}\int_\T w R\, dx\notag\\
            &= \text{Term 1} + \text{Term 2} + \text{Term 3}\label{Local:endpoint:term_split}.
        \end{align}
        Before bounding each of these terms seperately it will be useful to have a bound on the $L^p(\T)$ norm of $A$ for some $1\leq p<\infty$. Using H\"older's and Young's product inequality we get that
        \begin{align*}
            \left\|A\right\|_{L^p} &\lesssim \left\|u\right\|^2_{L^{2p}} + \left\|v\right\|^2_{L^{2p}}.
        \end{align*}
        Now using Lemma \ref{Lem:Trudinger_type} we obtain,
        \begin{align*}
            \left\|A\right\|_{L^p}\lesssim p\left(\left\|u\right\|^2_{H^{1/2}} + \left\|v\right\|^2_{H^{1/2}}\right).
        \end{align*}
        Finally, \ref{Local:endpoint:unif_bound_u} and \eqref{Local:endpoint:unif_bound_v} we obtain the uniform bound,
        \begin{align}
            \left\|A\right\|_{L^p} \lesssim p\max(\mathscr{E}_{1,1}(u_0)^2, C_{v_0,T}^2)\label{Local:endpoint:A_unif_bound}.
        \end{align}
        We now go through the terms in \eqref{Local:endpoint:term_split}, obtaining our upper bounds.
        \begin{itemize}
            \item \underline{Term 1:} Note that  $w^2 = w^{2(1-1/p)}w^{2/p}$.
            \begin{align*}
                \left|\text{Term 1}\right| &\lesssim \big\| w^{2(1-1/p)}\big\|_{L^{\frac{p}{p-1}}}\big\|w^{2/p}\phi_\beta(D_x)A\big\|_{L^p}\\
                &\leq \left\|w\right\|^{2(1-1/p)}_{L^2}\left\|w\right\|^{2/p}_{L^4} \left\|\phi_\beta(D_x)A\right\|_{L^{2p}}\\
                &\leq g(t)^{1-1/p}\left(\left\|u\right\|^{2/p}_{L^4} + \left\|v\right\|_{L^4}^{2/p}\right)\left\|\phi_\beta(D_x)A\right\|_{L^{2p}}.
            \end{align*}
            Using the $L^p$ boundedness of $\phi_\beta(D_x)$ and \eqref{Local:endpoint:A_unif_bound} we obtain,
            \begin{align*}
                \left|\text{Term 1}\right| &\lesssim p\max(\mathscr{E}_{1,1}(u_0)^2, C_{v_0,T}^2) g(t)^{1-1/p} \left(\left\|u\right\|^{2/p}_{L^4} + \left\|v\right\|_{L^4}^{2/p}\right).
            \end{align*}
            Now Using Lemma \ref{Lem:Trudinger_type} with our uniform bounds \eqref{Local:endpoint:unif_bound_u} and \eqref{Local:endpoint:unif_bound_v},
            \begin{align}
                \left|\text{Term 1}\right| &\lesssim p\max(\mathscr{E}_{1,1}(u_0), C_{v_0,T})^{2(1+1/p)} g(t)^{1-1/p}\label{Local:endpoint:Term_1_bound}.
            \end{align}
            \item \underline{Term 2:} Here we note that $w = w^{1-2/p}w^{2/p}$
            \begin{align*}
                \big|\text{Term 2}\big| &\lesssim \big\|w^{1-2/p} \phi_\beta(D_x)w\big\|_{L^{\frac{p}{p-1}}}\big\|w^{2/p}A\big\|_{L^p}\\
                &\lesssim \big\|w^{1-2/p}\big\|_{L^{\frac{2p}{p-2}}}\left\|\phi_\beta(D_x)w\right\|_{L^2}\left(\left\|u\right\|^{2/p}_{L^4} + \left\|v\right\|_{L^4}^{2/p}\right)\left\|A\right\|_{L^{2p}}\\
                &\lesssim g(t)^{1-1/p} \left(\left\|u\right\|^{2/p}_{L^4} + \left\|v\right\|_{L^4}^{2/p}\right)\left\|A\right\|_{L^{2p}}.
            \end{align*}
            Using \eqref{Local:endpoint:A_unif_bound}, Lemma \ref{Lem:Trudinger_type} and our uniform bounds \eqref{Local:endpoint:unif_bound_u} and \eqref{Local:endpoint:unif_bound_v} we get,
            \begin{align}
                \left|\text{Term 2}\right| \lesssim p\max(\mathscr{E}_{1,1}(u_0), C_{v_0,T})^{2(1+1/p)} g(t)^{1-1/p}\label{Local:endpoint:Term_2_bound}.
            \end{align}
            \item \underline{Term 3:} Using the commutator estimate, Lemma \ref{Local:endpoint:commutator_estimate} we get that for $\frac{1}{q} = \frac{1}{p}(1-\frac{2}{p})$ and $\frac{1}{r} = \frac{1}{2} - \frac{1}{q}$,
            \begin{align*}
                \left|\text{Term 3}\right| &\lesssim \left\|w\right\|_{L^2}\left\|A\right\|_{L^q}\left\|w\right\|_{L^r}\\
                &\lesssim \left\|w\right\|_{L^2}\left\|A\right\|_{L^q}\left\|w\right\|^{1-2/p}_{L^2}\left\|w\right\|_{L^p}^{\frac{p}{2}}\\
                &\leq \left\|w\right\|_{L^2}\left\|A\right\|_{L^q}\left\|w\right\|^{1-2/p}_{L^2}\left(\left\|u\right\|_{L^p}^{\frac{p}{2}}+ \left\|v\right\|_{L^p}^{\frac{p}{2}}\right)\\
                &\lesssim qp^{1/p}\max\left(\mathscr{E}_{1,1}(u_0),C_{v_0,T}\right)^{2(1+1/p)}g(t)^{1-1/p}.
            \end{align*}
            Now taking $p>2$ large enough such that,
            \begin{align*}
                 qp^{1/p} &= p^{1/p}\frac{p}{1-\frac{2}{p}}\lesssim p
            \end{align*}
            we obtain,
            \begin{align}
                \left|\text{Term 3}\right| &\lesssim p \max\left(\mathscr{E}_{1,1}(u_0),C_{v_0,T}\right)^{2(1+1/p)}g(t)^{1-1/p}.\label{Local:endpoint:Term_3_bound}
            \end{align}
        \end{itemize}
        Now, combining \eqref{Local:endpoint:Term_1_bound}, \eqref{Local:endpoint:Term_2_bound} and \eqref{Local:endpoint:Term_3_bound} in \eqref{Local:endpoint:term_split} we get that for some constant $C>0$,
        \begin{align*}
            g'(t) &\leq  p C\max\left(\mathscr{E}_{1,1}(u_0),C_{v_0,T}\right)^{2(1+1/p)}g(t)^{1-1/p}.
        \end{align*}
        Integrating this we yield,
        \begin{align*}
            g(t) \leq \left(C\max\left(\mathscr{E}_{1,1}(u_0),C_{v_0,T}\right)^2\right)^p.
        \end{align*}
        For this upper bound to go to zero as $p\to \infty$ we require that $C\max\left(\mathscr{E}_{1,1}(u_0),C_{v_0,T}\right)^2 <1$. Since these constants depend only on our initial data however, for any $u_0$, there exists a constant $T>0$ such that on $[0,T]$, $g(t) = 0$. Therefore, we obtain uniqueness of our weak solution $u\in C_w([0,T];H^{1/2}(\T))\cap C([0,T];L^2(\T))$.
    \end{proof}
    \begin{remark}\label{Local:endpoint:failure_remark}
        The above method fails for the quintic and higher powers of the regurarised BO equation. One obtains a multiplicative factor of $p$ in the bound for $g(t)$, forcing the constant $C$ to depend on $p$ for the limit to evaluate to zero. This destroys the argument, indicating that new ideas are needed.
    \end{remark}
     Now finally we treat continuity of the solution map.
    \begin{lemma}\label{lem:Local:endpoint:continuity}
        The unique weak solution $u$ obtained in Lemma \ref{Lem:Local:endpoint:uniqueness} is in $C([0,T];H^{1/2}(\T))$ and is continuous with respect to the initial data $u_0$.
    \end{lemma}
    \begin{proof}
        By the weak lower semi-continuity of the norm  and the conservation \eqref{Local:endpoint:conservation_claim} of the $H^1$ initial data we obtain,
        \begin{align*}
            \left\|u(t)\right\|_{H^{1/2}} \leq \liminf_{m\to \infty}\left\|u_{n_m}(t)\right\|_{H^{1/2}} \leq \left\|u_0\right\|_{H^{1/2}}.
        \end{align*}
        Furthermore, by time reversal of \eqref{cubic_rbo_PDE} we get that,
        \begin{align*}
            \left\|u(0)\right\|_{H^{1/2}} \leq \left\|u(t)\right\|_{H^{1/2}}.
        \end{align*}
        Therefore, by the uniqueness of $u$, we know that $\left\|u(t)\right\|_{H^{1/2}} = \left\|u_0\right\|_{H^{1/2}}$. From this we can see that $u(t)$ is continuous in norm as $\left\|u(t)\right\|_{H^{1/2}} = \left\|u(s)\right\|_{H^{1/2}}$ for any $s<t<T$. Therefore, $u(t)$ is continuous in norm and weakly continuous, meaning it must be strongly continuous. The same reasoning can be used to infer that $u_0\mapsto u(t)$ is strongly continuous as $t\to 0$.
    \end{proof}
     Combinining these results we can prove Proposition \ref{THM:Local:endpoint}.
    \begin{proof}[Proof of Proposition \ref{THM:Local:endpoint}]
        Combine Lemmas \ref{Local:endpoint:Lem:existence}, \ref{Lem:Local:endpoint:uniqueness} and \ref{lem:Local:endpoint:continuity}.
    \end{proof}
    Now, as discussed in the introduction we may use the conservation of the $H^{1/2}$ norm to obtain the corresponding global result.
    \begin{proof}[Proof of Theorem \ref{thm:global:endpoint}]
        By Proposition \ref{THM:Local:endpoint} and the conservation of the $H^{1/2}$ norm, \eqref{Conservation:H1_like} we obtain a global solution via the standard iteration argument.
    \end{proof}


\section{Ill-posedness}
We will first prove that the solution map undergoes norm inflation. To do so, we will study an equivalent formulation that is easier to work with. 
\begin{proposition}\label{Prop_norm_integer}
    Let $d\geq 1$, $\mu \in \{-1,1\}$, $1\leq \beta \leq 2$, $k\geq 3$ and $s<0$. The following formulations are equivalent:
    \begin{itemize}
    \item \underline{Norm inflation everywhere:} For any $u_0\in H^s$ and $\varepsilon>0$, there exists a solution $u_\varepsilon$ to \eqref{Gen-BOBBM-eq} on $t\in (0,\varepsilon)$ such that,
    \begin{align*}
        \left\|u_\varepsilon(0) - u_0\right\|_{H^s} < \varepsilon \quad \text{and} \quad \left\|u_\varepsilon(t)\right\|_{H^s} > \varepsilon^{-1}.
    \end{align*}
    \item \underline{Quantitative norm inflation everywhere:} For any $u_0\in H^s$ and $n\in \N$, there exists a solution $u_n$ to \eqref{Gen-BOBBM-eq} and $t_n \in (0,1/n)$ such that,
    \begin{align}
        \left\|u_n(0) - u_0\right\|_{H^s} < \tfrac{1}{n} \quad \text{and} \quad \|u_n(t_n)\|_{H^s}>n \label{integer_norm_inflation}
    \end{align}
    \end{itemize} 
\end{proposition}
\begin{proof}
    See \cite[Proposition 3.1]{oh2017}
\end{proof}
Therefore, we will first work to establish the following.
\begin{proposition}\label{Thm:quant_norm_inflation}
    Let $k\geq 2 $, $1\leq \beta\leq 2$, $\mu \in \{-1,1\}$, $s<0$ and $d\geq 1$. Then \eqref{Gen-BOBBM-eq} undergoes quantitative norm inflation everywhere.
\end{proposition}
We will now introduce some notation.
\subsection{Tree formalism and Picard expansion}
It is important to note that \eqref{Gen-BOBBM-eq} is locally well-posed in the Fourier Lebesgue space $\mathscr{F}L^1(\T^d)$.
\begin{theorem}\label{Illposdness:local_FL}
    For data $u_0\in \mathscr{F}L^1(\T^d)$, there exists $T \lesssim_\beta k^{-1}\left(2\left\|u_0\right\|_{\mathscr{F}L^1}\right)^{1-k}$ and a unique solution $u\in C\left([0,T];\mathscr{F}L^1\right)$ to \eqref{Gen-BOBBM-eq}.
\end{theorem}
\begin{proof}
    The proof follows by the algebra property of $\mathscr{F}L^1(\T^d)$ and the unitarity of the propagator $S_\beta(t)$ in $\mathscr{F}L^1$.
\end{proof}
We define the Duhamel integral operator for \eqref{Gen-BOBBM-eq} as,
\begin{align*}
        \mathscr{I}_{\beta,k}[u_1,\dots,u_k](t) \coloneqq \int_0^t S_\beta(t-t')\phi_\beta(D_x)\prod_{i=1}^ku_i(t') \, dt'.
    \end{align*}
    We will often write $\mathscr{I}_{\beta,k}^k[u]$ to denote the function $\mathscr{I}_{\beta,k}[u,\dots,u]$. Using this we can define a Picard scheme for \eqref{Gen-BOBBM-eq},
    \begin{align}
        P_0(\psi) = S_\beta(t)\psi \quad \text{and} \quad P_j(\psi) = S_\beta(t) + \mathscr{I}_{\beta,k}^k\left[P_{j-1}(\psi)\right], \text{ for } j\in \N\label{eq: Picard_scheme}
    \end{align}
    where $\psi\in H^\sigma$ for some $\sigma \in \R$. Taking $j\to \infty$, Theorem \ref{Illposdness:local_FL} ensures that $P_j(u_0)$ converges to the solution to \eqref{Gen-BOBBM-eq} in $C_T\mathscr{F}L^1$ for short $T$.

 As done in \cite{forlano2020} and \cite{oh2017}, we make use of trees to track the Duhamel terms. We present the definition of $k$-ary rooted trees as used in these works.
 \begin{definition}
  (i) Given a partially ordered set $\mathscr{T}$ with partial order $\leq$, we say that $b\in \mathscr{T}$ with $b\leq a$ and $b\neq a$ is a \emph{child} of $a\in \mathscr{T}$, if $b\leq c \leq a$ implies $c=a$ or $c=b$. If the latter condition holds we also say that $a$ is a \emph{child} of $b$.
   
  (ii) A tree $\mathscr{T}$ is a finite partially ordered set, satisfying the following properties:
  \begin{itemize}
    \item Let $a_1,a_2,a_3,a_4\in \mathscr{T}$. If $a_4 \leq a_2\leq a_1$ and $a_4 \leq a_3 \leq a_1$, then we have $a_2\leq a_3$ or $a_3\leq a_2$,
    \item A node $a\in \mathscr{T}$ is called \emph{terminal}, if it has no child. A \emph{non-terminal} node $a\in \mathscr{T}$ is a node with exactly $k$ children.
    \item There exists a maximimal element $r\in \mathscr{T}$, the rood node of the tree, such that $a\leq r$ for all $a\in \mathscr{T}$,
    \item $\mathscr{T}$ consists of the disjoint union of $\mathscr{T}^0$ and $\mathscr{T}^\infty$, where $\mathscr{T}^0$ and $\mathscr{T}^\infty$ denote the collections of non-terminal nodes and terminal nodes, respectively.
  \end{itemize}
  (iii) Let $T_k(j)$ denote the set of all trees with $j$ non-terminal nodes.
 \end{definition}

 We now provide a bijection between $k$-ary trees and the iterative scheme shown in \eqref{eq: Picard_scheme}. 
 \begin{lemma}\label{lem: T_k(j) bound}
        Let $T_k(j)$ be the set of $k$-ary trees on $j$ generations. There exists $C_0 > 0$ such that,
        \[\# T_k(j) \leq C^j_0.\]
        for all $j \in \N$. 
    \end{lemma}
    \begin{proof}
       See \cite[Lemma 2.2]{forlano2020a} or \cite[Lemma 2.3]{oh2017} for the $k=3$ case.
    \end{proof}
For a fixed $\phi \in \mathscr{F}L^1(\T^d)$, we associate to a given tree $\mathscr{T}\in T_k(j)$, a space-time distribution $\Psi_\phi(\mathscr{T}) \in \mathcal{D}'([0,T]\times \T^d)$ as follows: we replace a non-terminal node by the Duhamel integral operator $\mathscr{I}_{k,\beta}$ with it's $k$ arguments as children and we replace all terminal nodes by the linear propagator $S(t)\phi$. Now we define
\begin{align*}
    \Xi_j(\phi) = \sum_{\mathscr{T}\in T_k(j)}\Phi_\phi(\mathscr{T}).
\end{align*}
    \subsection{Multilinear estimates}
    We will now define some variables that we will fix at the end of the proof of Theorem \ref{Thm:quant_norm_inflation}.
    \begin{definition}[Core variables and notation]\label{Illposedness_definitions}
    We denote,
    \begin{align*}
         e_{1,d} &\coloneqq  (1,0,\dots,0) \in \Z^d & b_d &\coloneqq \left\{\xi \in \Z^d: |\xi| \leq 1\right\}
    \end{align*}
    and define the following variables where a subscript denotes dependance,
     \begin{align*}
        N &\coloneqq N_n \gg 1 &
        A &\coloneqq A_N  \gg 1 &
        R &\coloneqq R_N &
        Q_{A,d} &\coloneqq \left[-\tfrac{A}{2},\tfrac{A}{2}\right)^d \subset \Z^d
        \end{align*}
        defined such that
\begin{align}
    RA \gg \|u_0\|_{\mathscr{F}L^1} \quad \text{and} \quad A \ll N \label{A,d,n,R_rel}.
\end{align}
Lastly, define we define,
\begin{align*}
    f_{s,d}(A) \coloneqq \begin{cases}
1, & \text{if } s < -\tfrac{d}{2}, \\[4pt]
(\log A)^{\tfrac{1}{2}}, & \text{if } s = -\tfrac{d}{2}, \\[4pt]
A^{\tfrac{d}{2} + s}, & \text{if } s > -\tfrac{d}{2}.
\end{cases}
\end{align*}
\end{definition}
With these variables set up we may define a the function we will use at every initial data to prompt an infinite loss of regularity. Define $\psi_n$ so that for any $\xi \in \Z^d$,
\begin{align}
        \widehat{\psi_n}(\xi) &= R\left\{\mathbbm{1}_{-2Ne_{1,d}+Q_{A,d}}(\xi) + \mathbbm{1}_{-Ne_{1,d}+Q_{A,d}}(\xi) + \mathbbm{1}_{Ne_{1,d}+Q_{A,d}}(\xi)+ \mathbbm{1}_{2Ne_{1,d}+Q_{A,d}}(\xi)\right\}\label{phin-bbm-def}.
    \end{align}
With these definitions we have the followign preliminary estimates.
\begin{lemma}
There exists $C>0$ such that  for all $\psi \in \mathscr{F}L^p(\T^d) \cap \mathscr{F}L^1(\T^d)$ with $1\leq p\leq \infty$ and $j\in \N$,
\begin{align}
    \| \Xi_j(\psi)(t)\|_{\mathscr{F}L^1} &\leq C^j t^j\|\psi\|^{(k-1)j+1}_{\mathscr{F}L^1}\label{Xibbm-Fl1-est}\\
    \| \Xi_j(\psi)(t)\|_{\mathscr{F}L^\infty} &\leq C^j t^j \|\psi\|^{(k-1)j-1}_{\mathscr{F}L^1}\|\psi\|^2_{L^2}\label{Xibbm-flinf}\\
    \left\| \Xi_j(u_0 + \psi)(t) - \Xi_j(\psi)(t)\right\|_{\mathscr{F}L^p} &\leq C^j t^j\|u_0\|_{\mathscr{F}L^p}\left(\|u_0\|_{\mathscr{F}L^1}^{(k-1)j+1} + \left\|\psi\right\|^{(k-1)j+1}_{\mathscr{F}L^1}\right) \label{Multi-Xi-est-BBm}.
\end{align}
Furthermore, for $\psi_n$ as defined in \eqref{phin-bbm-def} and $s<0$ we have that,
\begin{align}
         \|\psi_n\|_{H^s} \sim RA^{d/2}N^s \quad \text{and} \quad \|\psi_n\|_{\mathscr{F}L^1}\sim RA^d \label{phi_growth}.
    \end{align}
With  $u_{0,n} = u_0 + \psi_n$ it also holds that
    \begin{align}
         \|u_{0,n} - u_0\|_{H^s} &\lesssim RA^{\tfrac{d}{2}}N^s\label{est: BBM u_0,n - u_0}\\
        \left\|\Xi_0(u_{0,n})\right\|_{H^s}&\lesssim 1+ RA^{\tfrac{d}{2}}N^s\label{est: BBM Xi_0}\\
         \left\|\Xi_1(u_{0,n}) - \Xi_1(\psi_n)(t)\right\|_{H^s} &\lesssim t\|u_0\|_{L^2}R^{k-1} A^{(k-1)d}\label{est: BBM Xi_1 diff}\\
         \|\Xi_j(u_{0,n}^{})(t)\|_{H^s} &\leq  C^jt^j(RA^d)^{(k-1)j}\left(\|u_0\|_{L^2} + Rf_s(A)\right)\label{Xi-Hs-up-bd}.
    \end{align}
Lastly for some fixed constants $c>0$ and $M\geq 1$,
\begin{align}
    \|\langle\xi\rangle^s\|_{\ell^2(B_{cMA}(\xi))} &\lesssim Mf_s(A)\label{large_factor_xi_control}\\
    \|\langle\xi\rangle^s\|_{\ell^2(B_{cMA}(\xi))} &\sim \|\langle\xi\rangle^s\|_{L^2(B_{cMA}(\xi)\setminus b)}\label{eq: annulus_control_xi}.
\end{align}
\end{lemma}
\begin{proof}
    The proof of \eqref{Xibbm-Fl1-est} and \eqref{Xibbm-flinf} follow almost identically from \cite[Lemma 2.3]{forlano2020a}. \eqref{Multi-Xi-est-BBm} follows from \cite[Lemma 2.6]{oh2017} and \eqref{phi_growth} follows simply from expanding the definition of $\psi_n$. \eqref{est: BBM u_0,n - u_0}, \eqref{est: BBM Xi_0}, \eqref{est: BBM Xi_1 diff} and \eqref{Xi-Hs-up-bd} follow almost identically from \cite[Lemma 3.2]{forlano2020a}. Lastly, \eqref{large_factor_xi_control} and \eqref{eq: annulus_control_xi} follow from evaluating the corresponding sums.
\end{proof}
Using this, we can construct our local solution corresponding to the data $u_{0,n}$ that will exhibit norm inflation.
\begin{lemma}\label{Power-series-convergence}
    Let $u_{0,n}^{} = u_0 + \psi_n$ with corresponding solution to \eqref{Gen-BOBBM-eq} $u_n$. The following power series converges absolutely.
    \begin{align*}
        u_n = \sum_{j=0}^\infty \Xi_j(u_{0,n}^{})
    \end{align*}
    on $[0,T]$ for,
    \begin{align*}
        T \lesssim \left(\|u_0\|_{\mathscr{F}L^1} + RA\right)^{k-1} \sim \left(RA\right)^{1-k}.
    \end{align*}
\end{lemma}
\begin{proof}
    See \cite[Lemma 2.3]{forlano2020a} for a proof of convergence of a similar power series.
\end{proof}
We now recall a standard inequality for the convolution of cubes.
\begin{lemma}\label{Lem: cube conv ineq}
    Let $d\geq 1$. Then there exists $c_d$ such that
    \begin{align*}
        \mathbbm{1}_{a+Q_{A,d}} * \mathbbm{1}_{b+Q_{A,d}}(\xi) \geq c_d A^d \mathbbm{1}_{a+b+Q_{A,d}}(\xi)
    \end{align*}
    for all $a,b,\xi \in \Z^d$ and $A\geq 1$.
\end{lemma}
Using this, we may prove the key lower bound. It follows closely the method of Kishimoto in \cite{kishimoto2019} and Oh in \cite{oh2017}.
\begin{lemma}\label{Lem: LOWERBD}
    Let $\psi_n$ be defined as in \eqref{phin-bbm-def} and $s<0$. For $0<t\ll 1$, we have,
    \begin{align*}
        \left\|\Pi_{Q_{A,d}}\Xi_1^{}(\psi_n)(t)\right\|_{H^s} &\gtrsim  tR^k A^{d(k-1)} f_s(A).
    \end{align*}   
\end{lemma}
\begin{proof}
    Let $\omega(\xi) = \widehat{\phi_\beta}(\xi)$. We recall,
    \begin{align*}
        \Xi_1(\psi_n)(t) = \int_0^t S_\beta(t-t')\psi_\beta(D_x)\left(S_\beta(t') \psi_n(t')\right)^k \, dt'.
    \end{align*}
    Taking the Fourier transform and bounding it's modulus below by it's real component we get that,
    \begin{align}
         &\left|\mathscr{F}\left\{ \Xi_1^{}(\psi_n)(t)\right\}(\xi)\right|\notag\\
          &\hphantom{X}\geq \Big|\omega(\xi)\Big|\sum_{\substack{\xi = \xi_1+\cdots+\xi_k \\ \xi \in \Z^d}} \prod_{i=1}^k \widehat{\psi_n}(\xi_i)\Rea \int_0^t \exp\Big(-it' \Big(\omega(\xi) - \sum_{i=1}^k \omega(\xi_i)\Big)\Big)\, dt'\label{eq:FT_for_BBM_Xi_1}
    \end{align}
   where we have used that $e^{t\omega(\xi)}$ is on the unit circle and each of the $\widehat{\psi_n}$ functions are real valued. We note that given $\xi_i \in \supp \widehat{\psi_n}$, $|\xi_i|\sim N$. Therefore, we have that $|\xi|= \left|\xi_1+\dots+ \xi_k\right| \sim N$ and resultantly for $0\leq t' \ll 1$
    \begin{align*}
        \Big|t' \Big[\omega(\xi) - \sum_{i=1}^{k} \omega(\xi_i)\Big]\Big| \ll 1.
    \end{align*}
    Now, note that when $|x|\leq \tfrac{\pi}{3}$, $\cos(x) \geq \tfrac{1}{2}$. As such, for $0\leq t' \ll 1$,
    \begin{align*}
         \Rea \int_0^t \exp\Big(-it'\Big[\omega(\xi) - \sum_{i=1}^{k} \omega(\xi_i)\Big]\Big) \, dt' &=\int_0^t \cos\Big(-t'\Big[\omega(\xi) - \sum_{i=1}^{k} \omega(\xi_i)\Big]\Big) \, dt'\geq \frac{1}{2}t.
    \end{align*}
    Using this in \eqref{eq:FT_for_BBM_Xi_1},
    \begin{align*}
         \left|\mathscr{F}\left\{ \Xi_1(\psi_n)(t)\right\}(\xi)\right| &\geq t\frac{\left|\omega(\xi)\right|}{2}\sum_{\substack{\xi = \xi_1+\cdots+\xi_k \\ \xi \in \Z^d}} \prod_{i=1}^k \widehat{\psi_n}(\xi_i)
         =   t\tfrac{\left|\omega(\xi)\right|}{2}\big(\widehat{\psi_n}*\dots* \widehat{\psi_n}\big)(\xi).
    \end{align*}
    Given that a convolution of indicator functions will be positive, we may indicate on a nonempty subset of $\Z^d$ to lower bound this value further. As such, we will indicate on $Q_{A,d} \setminus b$ where $b$ is the set defined in Definition \ref{Illposedness_definitions},
    \begin{align*}
        \left|\mathscr{F}\left\{ \Xi_1(\psi_n)(t)\right\}(\xi)\right|
        &\geq \mathbbm{1}_{Q_{A,d}\setminus b}(\xi)t \tfrac{|\omega(\xi)|}{2}\big(\widehat{\psi_n}*\dots* \widehat{\psi_n}\big)(\xi)\\
        &\geq\mathbbm{1}_{Q_{A,d}\setminus b}(\xi)t\big(\widehat{\psi_n}*\dots* \widehat{\psi_n}\big)(\xi).
    \end{align*} 
     Where in the last inequality we recall that $\xi \in Q_{A,d}\setminus b$ implies that $|\xi| \geq 1$. Expanding our definition of $\psi_n$, \eqref{phin-bbm-def} we get a sum over all possible combinations of $k$-convolutions of the indicator functions in $\widehat{\psi_n}$,
    \begin{align*}
         \left|\mathscr{F}\left\{ \Xi_1^{}(\psi_n)(t)\right\}(\xi)\right| \geq \mathbbm{1}_{Q_{A,d}\setminus b}(\xi) tR^k\sum_{j_1, \dots,j_k \in \left\{-2,-1,1,2\right\}}\mathbbm{1}_{j_1Ne_{1,d} + Q_{A,d}}*\dots * \mathbbm{1}_{j_kNe_{1,d} +Q_{A,d}}(\xi).
    \end{align*}
    We must now strategically pick a term from this sum such that we can use Lemma \eqref{Lem: cube conv ineq} to end up with a $\mathbbm{1}_{Q_{A,d}}$ term in a lower bound. For $k$ even, we choose the term that alternates between $-Ne_{1,d}+Q_{A,d}$ and $Ne_{1,d} + Q_{A,d}$ like so,
    \begin{align*}
        |\mathscr{F}\left\{ \Xi_1^{}(\psi_n)(t)\right\}(\xi)| & \geq\mathbbm{1}_{Q_{A,d} \setminus b}(\xi) t R^k\left(\mathbbm{1}_{-Ne_{1,d} + Q_{A,d}}*\dots * \mathbbm{1}_{Ne_{1,d} +Q_{A,d}}\right)(\xi)\\
        &\gtrsim tR^k A^{d(k-1)}\mathbbm{1}_{Q_{A,d} \setminus b}(\xi) 
    \end{align*}
    where we have applied Lemma \ref{Lem: cube conv ineq} $k-1$ times. If $k=3$ we choose the following,
    \begin{align*}
        |\mathscr{F}\left\{ \Xi_1^{}(\psi_n)(t)\right\}|(\xi) &\geq\mathbbm{1}_{Q_{A,d} \setminus b}(\xi) t R^k\left(\mathbbm{1}_{-2Ne_{1,d} + Q_{A,d}}* \mathbbm{1}_{Ne_{1,d} + Q_{A,d}} * \mathbbm{1}_{Ne_{1,d} +Q_{A,d}}\right)\\
        &\gtrsim tR^k A^{d(k-1)}\mathbbm{1}_{Q_{A,d} \setminus b}(\xi) .
    \end{align*}
    Where again, we have applied Lemma \ref{Lem: cube conv ineq} $k-1$ times. We can extend this to the more general $k$ odd case by alternating between $-Ne_{1,d} + Q_{A,d}$ and $Ne_{1,d} + Q_{A,d}$ terms until the $k-2$th term, where we choose as we have in the $k=3$ case,
    \begin{align*}
         \left|\mathscr{F}\left\{ \Xi_1^{}(\psi_n)(t)\right\}(\xi)\right|
         &\gtrsim \mathbbm{1}_{Q_{A,d} \setminus b} t R^kA^{2d}\big(\overbrace{\mathbbm{1}_{-Ne_{1,d} + Q_{A,d}}*\dots * \mathbbm{1}_{Ne_{1,d} +Q_{A,d}}}^{k-3 \text{ terms}} * \mathbbm{1}_{Q_{A,d}}\big)(\xi)\\
        &\gtrsim t R^k A^{d(k-1)}\mathbbm{1}_{Q_{A,d} \setminus b}.
    \end{align*}   
    Thus, for any $k$, we have,
    \begin{align}
         \left|\mathscr{F}\left\{ \Xi_1^{}(\psi_n)(t)\right\}(\xi)\right| \gtrsim t R^k A^{d(k-1)}\mathbbm{1}_{Q_{A,d} \setminus b}.\label{Illposedness:f_t_lower_bound_xi}
    \end{align}
Taking the $H^s$ norm of $\Pi_{Q_{A,d}}\Xi_1^{}(\psi_n)$,
\begin{align*}
    \left\|\Pi_{Q_{A,d}}\Xi_1^{}(\psi_n)(t)\right\|_{H^s} = \bigg(\sum_{\xi\in Q_{A,d}} \langle \xi\rangle^{2s}\mathscr{F}\left\{\Xi_1^{}(\psi_n)(t)\right\}^2 \bigg)^{\tfrac{1}{2}}&\gtrsim  tR^k A^{d(k-1)} \left\|\langle \xi \rangle^{s}\right\|_{\ell^2_{\xi}(Q_{A,d} \setminus b)}
\end{align*}
 where in the second inequality we used our lower bound \eqref{Illposedness:f_t_lower_bound_xi}. Estimating $\|\langle\xi\rangle^s\|_{\ell^2_\xi(Q_{A,d})}$, we find that for any $0<\varepsilon<A/2$,
    \begin{align*}
        \|\langle\xi\rangle^s\|_{\ell^2_\xi(B_{(1-\varepsilon)A/2\setminus b})}\leq\|\langle\xi\rangle^s\|_{\ell^2_\xi(Q_{A,d}\setminus b)}\leq \|\langle\xi\rangle^s\|_{\ell^2_\xi(B_{(1+\varepsilon)A/2}\setminus b)}
    \end{align*}
    where using \eqref{large_factor_xi_control} and \eqref{eq: annulus_control_xi} gives us that $\|\langle\xi\rangle^s\|_{\ell^2_\xi(Q_{A,d} \setminus b)} \sim f_{s,d}(A)$. This yields the desired result.
\end{proof}
\subsection{Proof of infinite loss}
\begin{proposition}\label{THM:Norm_inf_to_prop}
    Proposition \ref{Thm:quant_norm_inflation} is equivalent to the following properties holding for $n\in \N$.
    \begin{enumerate}
        \item[(i)] $RA^{\tfrac{d}{2}}N^s \ll \tfrac{1}{n}$,
        \item[(ii)] $R^{k-1}TA^{d(k-1)} \ll 1$,
        \item[(iii)] $TR^kA^{d(k-1)}f_{s,d}(A) \gg n$,
        \item[(iv)] $T \ll 1$
        \item[(v)]\eqref{phi_growth} and $Rf_{s,d}(A)\gg \left\|u_0\right\|_{L^2}$.
    \end{enumerate}
\end{proposition}
\begin{proof}
Assuming (i) and using \eqref{est: BBM u_0,n - u_0}  we get,
\begin{align*}
    \big\|u_{0,n}^{} - u_0\big\|_{H^s} \lesssim RA^{\tfrac{d}{2}}N^s \ll \tfrac{1}{n}
\end{align*}
which gives us the first estimate in \eqref{integer_norm_inflation}. Using \eqref{A,d,n,R_rel} and \eqref{phi_growth} we get,
\begin{align*}
    &RA \sim  \|\psi_n\|_{\mathscr{F}L^1}- \|u_{0}\|_{\mathscr{F}L^1} \leq \|u_{0,n}^{}\|_{\mathscr{F}L^1} \leq \|\psi_n\|_{\mathscr{F}L^1}+ \|u_{0}\|_{\mathscr{F}L^1} \sim RA^d\\
    &\hphantom{}\implies \|u_{0,n}^{}\|_{\mathscr{F}L^1}  \sim RA^d.
\end{align*}
Thus, by property (ii) we have that,
\begin{align*}
    T\big\|u_{0,n}^{}\big\|^{k-1}_{\mathscr{F}L^1} \sim TR^{k-1}A^{d(k-1)} \ll 1
\end{align*}
this implies that, $T\lesssim \big\|u_{0,n}^{}\big\|^{1-k}$. Thus, by Theorem \ref{Illposdness:local_FL} we get a local solution $u_n$ to \eqref{Gen-BOBBM-eq} from $[0,T]$ with initial data $u_{0,n}^{}$. Furthermore, by Lemma \ref{Power-series-convergence} we obtain the convergence of the power series on $[0,T]$. By \eqref{est: BBM Xi_0} and property (i) we have that,
\begin{align}
    \big\|\Xi_{0}^{}(u_{0,n}^{})(T)\big\|_{H^s} \lesssim 1+RA^{\tfrac{d}{2}}N^s \sim 1.\label{Illposedness:Xi_0:lower_final}
\end{align}
and by Lemma \ref{Lem: LOWERBD} and property (iv) (iii) we get,
\begin{align}
     \left\|\Xi_1^{}\left(\psi_n\right)(T)\right\|_{H^s} \gtrsim TR^kA^{d(k-1)}f_s(A) \gg n\label{proof of prop: Xi_1}.
\end{align}
We can begin to estimate the rest of the power series using \eqref{Xi-Hs-up-bd} and property (v),
\begin{align*}
    \Big\|\sum_{j=2}^\infty \Xi_j(u_{0,n})(T)\Big\|_{H^s} &\leq \sum_{j=2}^\infty \left\|\Xi_j(u_{0,n})(T)\right\|_{H^s}\lesssim \sum_{j=2}^\infty T^jR^{j(k-1)}A^{dj(k-1)}\left(\left\|u_0\right\|_{L^2} + Rf_{s,d}(A)\right)\\
    &\lesssim \sum_{j=2}^\infty T^jR^{j(k-1)+1}A^{dj(k-1)}f_{s,d}(A).
\end{align*}
Now using property (ii),
\begin{align}
    \Big\|\sum_{j=2}^\infty \Xi_j(u_{0,n})(T)\Big\|_{H^s} &\lesssim TR^{k}A^{d(k-1)}f_{s,d}(A)\sum_{j=1}^\infty T^jR^{j(k-1)}A^{jd(k-1)}\notag\\
    &\lesssim  TR^{k}A^{k-1}f_{s,d}(A).\label{Illposedness:Xi_power_upper:final}
\end{align}
Now, tying these estimates together, we use the reverse triangle inequality on the power series expansion for our solution $u(T)$ who's convergence is guaranteed by Lemma \ref{Power-series-convergence}. Then, we make use of our estimates \eqref{Illposedness:Xi_0:lower_final}, \eqref{proof of prop: Xi_1} and \eqref{Illposedness:Xi_power_upper:final} and \eqref{est: BBM Xi_1 diff} ,
\begin{align*}
    &\left\|u(T)\right\|_{H^s}\\
     &\hphantom{X}\geq \left\|\Xi_1(\psi_n)(T)\right\|_{H^s} - \left\|\Xi_0(u_{0,n})\right\|_{H^s} - \left\|\Xi_1(u_{0,n}(T)) - \Xi_1(\psi_n)(T)\right\|_{H^s}- \Big\|\sum_{j=2}^\infty \Xi_j(u_{0,n}(T))\Big\|_{H^s}\\
    &\hphantom{X}\gtrsim TR^kA^{d(k-1)}f_s(A) - 1+RA^{\tfrac{d}{2}}N^s - T\left\|u_0\right\|_{L^2}R^{k-1}A^{d(k-1)}  \\
    &\hphantom{XXXXXXXXXXXXXXXXXXXXXXXX}- TR^kA^{d(k-1)}f_{s,d}(A)R^{k-1}A^{d(k-1)}.
\end{align*}
With (i) we get, $ RA^{\tfrac{d}{2}}N^s \ll \tfrac{1}{n}$. (ii) and (v) tell us that $T\left\|u_0\right\|_{L^2}R^{k-1}A^{d(k-1)} \ll 1$. Furthermore, (ii) with (iii) yields
\begin{align*}
    TR^kA^{d(k-1)}f_{s,d}(A)R^{k-1}TA^{d(k-1)} \ll TR^kA^{d(k-1)}f_{s,d}(A).
\end{align*}
Thus,
\begin{align*}
    1+RA^{\tfrac{d}{2}}N^s + T\left\|u_0\right\|_{L^2}R^{k-1}A^{d(k-1)}  + TR^kA^{d(k-1)}f_{s,d}(A)R^{k-1}A^{d(k-1)} \ll TR^kA^{d(k-1)}f_{s,d}(A).
\end{align*}
Therefore, by property (iv),
\begin{align*}
    \left\|u(T)\right\|_{H^s} &\gtrsim  TR^k A^{d(k-1)}f_{s,d}(A) \gg n.
\end{align*}
\end{proof}
\begin{proof}[Proof of Proposition \ref{Thm:quant_norm_inflation}]
To prove Proposition \ref{Thm:quant_norm_inflation}, it remains to verify each of the conditions in Proposition \ref{THM:Norm_inf_to_prop} by choosing an $A$, $R$ and $T$ in terms of $N$. We make the following choices,
\begin{equation}
    R = N^{-\tfrac{s}{2}}, \quad T = N^{(k-\tfrac{1}{2})\tfrac{s}{2}}, \quad \text{and} \quad A = 10\notag.
\end{equation}
We now verify the required conditions. For (i) we get,
\begin{align*}
    RA^{\tfrac{d}{2}}N^s \sim RN^s = N^{-\tfrac{s}{2}}N^s = N^{\tfrac{s}{2}} \ll \tfrac{1}{n}.
\end{align*}
For (ii),
\begin{align*}
    R^{k-1}TA^{d(k-1)} \sim R^{k-1}T = N^{-(k-\tfrac{1}{2})\tfrac{s}{2}}N^{(k-1)\tfrac{s}{2}} = N^{\tfrac{s}{4}} \ll 1.
\end{align*}
For (iii),
\begin{align*}
    R^{k}A^{d(k-1)}f_{s,d}(A) \sim R^{k}A^{d(k-1)} = N^{k\tfrac{s}{2}-\tfrac{s}{4} - k\tfrac{s}{2}} = N^{-\tfrac{s}{4}} \gg n.
\end{align*}
Here, (v) and (iv) are satisfied trivially given that $R\gg 1$. 
\end{proof}

\begin{proof}[Proof of Theorem \ref{InfiniteLossofRegs}]
    Pick $R,T$ and $A$ as in the proof of Proposition \ref{Thm:quant_norm_inflation}. First, we make the key observation that for $A \lesssim 1$ and some $f\in H^s\cap H^\sigma$, 
    \begin{align}
        \left\|\langle\xi\rangle^s \mathbbm{1}_{Q_{A,d}}(\xi)\right\|_{\ell^2_\xi} \sim_{s,\sigma} \left\|\langle\xi\rangle^{\sigma} \mathbbm{1}_{Q_{A,d}}(\xi)\right\|_{\ell^2_\xi}\notag.
    \end{align}
    Using this we draw the following bounds,
    \begin{align}
        \left\|u_n(T)\right\|_{H^\sigma} &\geq \left\|\Pi_{Q_{A,d}} u_n(T)\right\|_{H^\sigma}\notag\\
        &\geq C_{s,\sigma}\big\|\mathbbm{1}_{Q_{A,d}}(\xi)\langle\xi\rangle^s\widehat{u_n(T)}(\xi)\big\|_{\ell^2_\xi}= C_{s,\sigma}\left\| \Pi_{Q_{A,d}} u_n(T) \right\|_{H^s}\notag.
    \end{align}
    As in Proposition \ref{THM:Norm_inf_to_prop} we use the reverse triangle inequality to draw a lower bound in terms of the dominating Picard iterate,
    \begin{align*}
        \left\|u_n(T)\right\|_{H^\sigma} &\geq C_{s,\sigma}\bigg(\left\|\Pi_{Q_{A,d}}\Xi_1(\psi_n)(T)\right\|_{H^s} - \left\|\Pi_{Q_{A,d}}\Xi_0(u_{0,n})\right\|_{H^s} \bigg. \\
         &\hphantom{XXXXXX}\bigg.- \Big\|\Pi_{Q_{A,d}}\left(\Xi_1(u_{0,n}(T)) - \Xi_1(\psi_n)(T)\right)\Big\|_{H^s} - \Big\|\Pi_{Q_{A,d}}\sum_{j=2}^\infty \Xi_j(u_{0,n}(T))\Big\|_{H^s}\bigg).
    \end{align*}
    Using Lemma \ref{Lem: LOWERBD}, the fact that $\Pi_{Q_{a,d}}$ is a bounded operator on $H^s$ and Proposition \ref{Thm:quant_norm_inflation} we yield,
    \begin{align*}
        \left\|u_n(T)\right\|_{H^\sigma} \gg n
    \end{align*}
    thereby proving Theorem \ref{InfiniteLossofRegs}.
\end{proof}

\appendix



\section{A fractional Leibniz rule in $F^s_{p,q}(\T)$}
For $\mathcal{M}=\R,\T$, we begin by defining the Triebel-Lizorkin spaces.
\begin{definition}\label{Def:F_space}
    Let $s\in \R$, $0< q\leq\infty$, $0<p<\infty$ and $m:\mathcal{M} \to [0,\infty)$ be a measurable function. We define the \emph{weighted Triebel-Lizorkin spaces} by,
    \begin{align*}
        F^s_{p,q}(\mathcal{M},w) = \left\{f\in D'(\mathcal{M}): \left\|f\right\|_{F_{p,q}^s(\mathcal{M},w)}<\infty\right\}
    \end{align*}
    where $D'(\R) = \mathscr{S}'(\R)$, $D'(\T) = (C^\infty(\T))'$ and\footnote{This definition depends on the choice of the projectors $\mathbf{P}_N$. We will soon assume that $m$ belongs to a space $W$ where the defined norms are equivalent under such a choice. See \cite[Sections 5.1.1, 5.1.2]{triebel1983} for an explaination.}
    \begin{align*}
        \left\|f\right\|_{F_{p,q}^s(\mathcal{M},w)} \coloneqq \Big\|\Big(\sum_{N \text{ dyadic}} N^{qs}|\mathbf{P}_N f|^q\Big)^{\frac{1}{q}}\Big\|_{L^p(\mathcal{M},w)}.
    \end{align*}
    We write $F^s_{p,q}(\mathcal{M}) = F^s_{p,q}(\mathcal{M},\mathbbm{1}_\T)$ for the unweighted Triebel-Lizorkin spaces. We also define the Hardy spaces $h^p(\mathcal{M}) \coloneqq F_{p,2}^0(\mathcal{M})$.
\end{definition}
It is worth noting that for $p>1$, the Hardy spaces are equivalent in norm to the Lebesgue spaces, that is $h^p(\mathcal{M}) = L^p(\mathcal{M})$. We would like to prove the following result:
\begin{proposition}\label{Frac_F_rule}
    Let $0<p<1$, $0<p_1,p_1<\infty$, $s> \min(p,q,1)^{-1}-1$ and $0<q\leq \infty$ such that $\frac{1}{p} = \frac{1}{p_1}+\frac{1}{p_2}$. Then for $u,v \in (C^\infty(\T))'$ there is the estimate,
    \begin{align*}
        \left\|uv\right\|_{F_{p,q}^s(\T)}\lesssim \left\|u\right\|_{F_{p_1,q}^s(\T)}\left\|v\right\|_{h^{p_2}(\T)} + \left\|u\right\|_{h^{p_1}(\T)}\left\|v\right\|_{F^s_{p_2,q}(\T)}.
    \end{align*}
\end{proposition}
To prove this, we will make use of a result of Naibo and Thompson in \cite{naibo2019} of which the following theorem is a special case of.
\begin{theorem}\cite[Theorem 2.5]{naibo2019}\label{Thm_naibo}
    Let $0<p,p_1,p_2<\infty$, $0<q\leq \infty$ and $s> \min(p,q,1)^{-1}-1$ such that $\frac{1}{p} = \frac{1}{p_1}+ \frac{1}{p_2}$ Then it holds for $u,v\in \mathscr{S}'(\R)$
    \begin{align*}
         \left\|uv\right\|_{F_{p,q}^s(\R)}\lesssim \left\|u\right\|_{F_{p_1,q}^s(\R)}\left\|v\right\|_{h^{p_2}(\R)} + \left\|u\right\|_{h^{p_1}(\R)}\left\|v\right\|_{F^s_{p_2,q}(\R)}.
    \end{align*}
\end{theorem}
Our goal now is to transfer this result to the statement of Proposition \ref{Frac_F_rule}. To do this we will use the localisation method outlined by Triebel in \cite[Chapter 3]{triebel1983} and \cite[Chapter 5]{triebel1992}. This method requires one to localise the $F^s_{p,q}(\T)$ norm in $F^s_{p,q}(\R)$. In this direction, we recall the following result.
\begin{lemma}\cite[Theorem 9.2.1]{triebel1992}\label{FTFRident_lem}
    Let $s\in \R$, $0\leq q <\infty$ and $0<p<\infty$ and $w(x) \coloneqq (1+|x|)^{-\kappa}$ with $\kappa >0$ such that $\kappa p >1$. Then, 
    \begin{align*}
        F_{p,q}^s(\T) = F^s_{p,q,\pi}(\R,w)
    \end{align*}
    where $F^s_{p,q,\pi}(\R,w)$ is the space of periodic distributions on $\R$ restricted by the $F^s_{p,q}(\R,w)$ norm.
\end{lemma}
In this lemma, a distribution $f\in F_{p,q}^s(\T)$ is identified with a distribution $\tilde{f}\in F_{p,q,\pi}^s(\R,w)$ via the natural periodic extension (here we abuse notation and write $\tilde{f} =f$). It will be useful to replace $w$ from Lemma \ref{FTFRident_lem} with the smooth weight $\tilde{w} = (1+x^2)^{-\kappa/2}$. We note that on $\R$, $w \sim \tilde{w}$. Therefore, expanding the definition of the $F^s_{p,q}(\R,w)$ norm it is clear that in the sense of equivalence of norms on $\mathscr{S}_\pi'(\R)$ (the periodic Schwartz distributions on $\R$),
\begin{align}
    F^s_{p,q,\pi}(\R,w) = F^s_{p,q,\pi}(\R,\tilde{w}).\label{norm_equiv_smooth_weight}
\end{align}   
We now define a class of weights $W$ that will be needed to state an important lemma and would like $\tilde{w}$ to be a member of.
\begin{definition}\label{def_admissible_w}
    The class $W$ of admissible weight functions is the collection of all possible $C^\infty(\R)$ functions $w:\R\to [0,\infty)$ satifying the following properties:
    \begin{itemize}
        \item[(i)] For all $\alpha \in \N$, there exists $C_\alpha >0$ such that for all $x\in \R$,
        \begin{align*}
            |\pat_x^\alpha w(x)| \leq C_\alpha w(x).
        \end{align*}
        \item[(ii)] There exists $C>0$ and $\gamma \geq 0$ such that for all $x,y\in \R$,
        \begin{align*}
            0<w(x) \leq Cw(y)(1+|x-y|^2)^{\frac{\gamma}{2}}.
        \end{align*}
    \end{itemize}
\end{definition}
A short computation verifies that $\tilde{w}$ satisfies condition (i) of Definition \ref{def_admissible_w}. To show that $\tilde{w}$ also satisfies (ii) we recall Peetre's inequality for real numbers $x,y\in \R$
\begin{align*}
    1+x^2 \leq 2(1+(x-y)^2)(1+y^2).
\end{align*}
With this we get that
\begin{align}
    \frac{\tilde{w}(x)}{\tilde{w}(y)} \leq 2^{\frac{\kappa}{2}}(1+(x-y)^2)^{\frac{\kappa}{2}}.\label{w_ineqa}
\end{align}
 Rearranging \eqref{w_ineqa}, we see that property (ii) is satisfied. Therefore $\tilde{w}\in W$. The reason why this is desirable is the following lemma.
 \begin{proposition}\cite[Proposition 6.13]{triebel2006}\label{A:local_prop}
    Let $s\in \R$, $0<p<\infty$, $0<q\leq \infty$, $w\in W$ and $\psi\in C^\infty(\R)$ be a non-negative, compactly supported function such that for all $x\in \R$
    \begin{align*}
        \sum_{m\in \Z} \psi_m(x) \sim 1 \quad \psi_m(x) \coloneqq \psi(x-2\pi m) .
    \end{align*}
    Then, the right hand side of the equivalence
    \begin{align*}
        \left\|f\right\|_{F^s_{p,q}(\R,w)}^p \sim \sum_{m \in \Z}w^p(m)\left\|\psi_m f\right\|_{F_{p,q}^s(\R)}^p
    \end{align*}
    is finite.
 \end{proposition}
    Using this, we prove the following lemma.
    \begin{lemma}\label{A:Lem:T_R_loc}
        Let $0<p<\infty$, $0< q\leq \infty$ and $s\in \R$. If $\psi\in C^\infty(\R)$ such that $\psi \equiv 1$ on $[0,2\pi]$ and $\supp \psi  = [-0.1,2\pi +0.1]$. Then
        \begin{align*}
            \left\|f\right\|^p_{F_{p,q}^s(\T)} \sim \left\|\psi f\right\|^p_{F_{p,q}^s(\R)}
        \end{align*}
    \end{lemma}
    \begin{proof}
        Using Lemma \ref{FTFRident_lem} and \eqref{norm_equiv_smooth_weight} we get that
        \begin{align*}
            \left\|f\right\|_{F^s_{p,q}(\T)}\sim \left\|f\right\|_{F^s_{p,q}(\R ,w)}.
        \end{align*}
        Now using Proposition \ref{A:local_prop} and the fact that $\tilde{w}\in W$ we see that,
        \begin{align}
            \left\|f\right\|_{F^s_{p,q}(\T)}^p \sim \sum_{m\in \Z}(1+x^2)^{-\frac{\kappa p}{2}}\left\|\psi_mf\right\|_{F_{p,q}^s(\R)}^p\label{A:lem_loc:period}
        \end{align}
        where $\psi_m(x) = \psi(x-2\pi m)$. Noting that $f$ is a periodic distribution, it can be seen by the commutativity of Fourier multipliers with translation that
        \begin{align*}
            \left\|\psi_mf \right\|_{F^s_{p,q}(\R)}^p \sim  \left\|\psi f\right\|_{F^s_{p,q}(\R)}^p.
        \end{align*}
        Therefore, returning to \eqref{A:lem_loc:period}
        \begin{align*}
            \left\|f\right\|^p_{F^s_{p,q}(\T)}\sim \left\|\psi f\right\|_{F_{p,q}^s(\R)}^p \sum_{m\in \Z}(1+x^2)^{-\frac{\kappa p}{2}} \sim \left\|\psi f\right\|_{F_{p,q}^s(\R)}^p
        \end{align*}
        where the sum converges as $\kappa p >1$.
        \end{proof}
        With this we may prove Proposition \ref{Frac_F_rule}.
        \begin{proof}[Proof of Proposition \ref{Frac_F_rule}]
            Let $\psi\in C^\infty(\R)$ such that $\psi \equiv 1$ on $[0,2\pi]$ and $\supp \psi =[-0.1,2\pi +0.1]$. By Lemma \ref{A:Lem:T_R_loc} we have that
            \begin{align}
                \left\|uv\right\|_{F^s_{p,q}(\T)} \sim \left\|\psi uv\right\|_{F^s_{p,q}(\R)}.\label{A:main_proof:loc_step}
            \end{align}
            Now defining $\phi \in C^\infty(\R)$ with $\phi \equiv 1$ on $[0,2\pi]$ and $\supp \phi = [-0.2,2\pi + 0.2]$. Therefore,
            \begin{align*}
                \left\|\psi uv\right\|_{F^s_{p,q}(\R)} = \left\|(\psi u)(\phi v)\right\|_{F^s_{p,q}(\R)}.
            \end{align*}
            Returning this to \eqref{A:main_proof:loc_step} and applying Theorem \ref{Thm_naibo} we see that,
            \begin{align}
                \left\|uv\right\|_{F^s_{p,q}(\T)} &\lesssim \left\|\psi u\right\|_{F^s_{p,q}(\R)}\left\|\phi v\right\|_{H^{p_2}} + \left\|\psi u\right\|_{H^{p_1}}\left\|\phi v\right\|_{F^s_{p_2,q}(\R)}\label{A:main_proof:pre_transference}.
            \end{align}
           Now using Lemma \ref{A:Lem:T_R_loc} 
            \begin{align*}
                \left\|\psi u\right\|_{F^s_{p_1,q}(\R)}\sim \left\|u\right\|_{F^s_{p_1,q}(\T)}, \quad \left\|\phi v\right\|_{F^s_{p_2,q}(\R)} \sim \left\|v\right\|_{F^s_{p_2,q}(\T)}
            \end{align*}
            with the same holding for the other terms. Returning to \eqref{A:main_proof:pre_transference} with this, we yield the desired estimate.
        \end{proof}

\textbf{Acknowledgements:} This paper was written during the author's honours year at Monash University. The author would like to thank his supervisor, Justin Forlano for his generosity, patience and support.
\bibliographystyle{abbrvurl}
\bibliography{../Honours_research}

\end{document}